\documentclass[12pt,reqno]{amsart}

\usepackage{amstext}
\usepackage{amsthm}
\usepackage{amsmath}
\usepackage{amssymb}
\usepackage{latexsym}
\usepackage{amsfonts}
\usepackage{graphicx}
\usepackage{texdraw}
\usepackage{graphpap}

\usepackage[pagebackref,hypertexnames=false, colorlinks, citecolor=red, linkcolor=red]{hyperref} %,hypertexnames=false,colorlinks,[pagebackref]
\usepackage[backrefs]{amsrefs}

\input txdtools

\begin{document}
%Commands Used%

\newcommand{\ci}[1]{_{ {}_{\scriptstyle #1}}}

\newcommand{\norm}[1]{\ensuremath{\left\|#1\right\|}}
\newcommand{\abs}[1]{\ensuremath{\left\vert#1\right\vert}}
\newcommand{\ip}[2]{\ensuremath{\left\langle#1,#2\right\rangle}}
\newcommand{\p}{\ensuremath{\partial}}
\newcommand{\pr}{\mathcal{P}}

\newcommand{\pbar}{\ensuremath{\bar{\partial}}}
\newcommand{\db}{\overline\partial}
\newcommand{\D}{\mathbb{D}}
\newcommand{\B}{\mathbb{B}}
\newcommand{\Sp}{\mathbb{S}}
\newcommand{\T}{\mathbb{T}}
\newcommand{\R}{\mathbb{R}}
\newcommand{\Z}{\mathbb{Z}}
\newcommand{\C}{\mathbb{C}}
\newcommand{\N}{\mathbb{N}}
\newcommand{\scrH}{\mathcal{H}}
\newcommand{\scrL}{\mathcal{L}}
\newcommand{\td}{\widetilde\Delta}

\newcommand{\La}{\langle }
\newcommand{\Ra}{\rangle }
\newcommand{\rk}{\operatorname{rk}}
\newcommand{\card}{\operatorname{card}}
\newcommand{\ran}{\operatorname{Ran}}
\newcommand{\osc}{\operatorname{OSC}}
\newcommand{\im}{\operatorname{Im}}
\newcommand{\re}{\operatorname{Re}}
\newcommand{\tr}{\operatorname{tr}}
\newcommand{\vf}{\varphi}
\newcommand{\f}[2]{\ensuremath{\frac{#1}{#2}}}

\newcommand{\kzp}{k_z^{(p,\alpha)}}
\newcommand{\klp}{k_{\lambda_i}^{(p,\alpha)}}
\newcommand{\TTp}{\mathcal{T}_p}

%%%%%%%%%%%%%%%%%%%%%%%%%%%%

\newcommand{\entrylabel}[1]{\mbox{#1}\hfill}

\newenvironment{entry}
{\begin{list}{X}%
  {\renewcommand{\makelabel}{\entrylabel}%
      \setlength{\labelwidth}{55pt}%
      \setlength{\leftmargin}{\labelwidth}%\labelsep}%
      \addtolength{\leftmargin}{\labelsep}%
   }%
}%
{\end{list}}

%%%%%%%%%%%%%%%%%%%%%%%%%%%%

\numberwithin{equation}{section}

\newtheorem{thm}{Theorem}[section]
\newtheorem{lm}[thm]{Lemma}
\newtheorem{cor}[thm]{Corollary}
\newtheorem{conj}[thm]{Conjecture}
\newtheorem{prob}[thm]{Problem}
\newtheorem{prop}[thm]{Proposition}
\newtheorem*{prop*}{Proposition}

\theoremstyle{remark}
\newtheorem{rem}[thm]{Remark}
\newtheorem*{rem*}{Remark}

\title{The Essential Norm of Operators on $A^p(\mathbb{D}^n)$}

\author[M. Mitkovski]{Mishko Mitkovski$^\dagger$}
\address{Mishko Mitkovski, Department of Mathematical Sciences\\ Clemson University\\ O-110 Martin Hall, Box 340975\\ Clemson, SC USA 29634}
\email{mmitkov@clemson.edu}
\urladdr{http://people.clemson.edu/~mmitkov/}
\thanks{$\dagger$ Research supported in part by National Science Foundation DMS grant \# 1101251.}

\author[B. D. Wick]{Brett D. Wick$^\ddagger$}
\address{Brett D. Wick, School of Mathematics\\ Georgia Institute of Technology\\ 686 Cherry Street\\ Atlanta, GA USA 30332-0160}
\email{wick@math.gatech.edu}
\urladdr{www.math.gatech.edu/~bwick6}
\thanks{$\ddagger$ Research supported in part by National Science Foundation DMS grants \# 1001098 and \# 955432.}


\begin{abstract}
In this paper we characterize the compact operators on the Bergman space $A^p(\mathbb{D}^n)$.  The main result shows that an operator on $A^p(\mathbb{D}^n)$ is compact if and only if it belongs to the Toeplitz algebra $\mathcal{T}_{p}$ and its Berezin transform vanishes on the boundary.
\end{abstract}

\maketitle

\section{Introduction and Statement of Main Results}

Let $\D^n$ denote the unit polydisc in $\C^n$.  For $1<p<\infty$ the Bergman space $A^p(\D^n):=A^p$ is the collection of holomorphic functions on $\D^n$ such that
$$
\norm{f}_{A^p}^p:=\int_{\D^n}\abs{f(z)}^p\,dv(z)<\infty.
$$
We will also let $L^p\left(\D^n\right):=L^p$ denote the Lebesgue space on $\D^n$ with respect to the normalized volume measure $v$, $dv(z):=\frac{1}{\pi^n}dA(z_1)\cdots dA(z_n)$.

Recall that the projection of $L^2$ onto $A^2$ is given by the integral operator
$$
P(f)(z):=\int_{\D^n}f(w)\prod_{l=1}^n\frac{1}{\left(1-z_l\overline{w}_l\right)^{2}}\,dv(w).
$$
It is well known that this operator is bounded from $L^p$ to $A^p$ when $1<p<\infty$.  Let $M_a$ denote the operator of multiplication by the function $a$, $M_a(f):=af$.  The Toeplitz operator with symbol $a\in L^\infty$ is the operator given by
$$
T_a:=P M_a.
$$
It is immediate to see that $\norm{T_a}_{\mathcal{L}\left(L^p, A^p\right)}\lesssim\norm{a}_{L^\infty}$.   For  $\lambda\in\D^n$, set $K_\lambda(z):=\prod_{l=1}^n\frac{1}{(1-\overline{\lambda}_lz_l)^{2}}$, and for $1<p<\infty$ let $k_{\lambda}^{(p)}(z):=\prod_{l=1}^n\frac{(1-\abs{\lambda_l}^2)^{\frac{2}{q}}}{(1-\overline{\lambda}_lz_l)^{2}}$.  Then we have $\norm{k_{\lambda}^{(p)}}_{A^p}\approx 1$, with implied constants depending on $p$ and $n$.  Here, we are letting $q=\frac{p}{(p-1)}$.

For $z\in\D^n$, the Berezin transform of an operator $S$ is defined by
$$
B(S)(z):=\ip{Sk_z^{(p)}}{k_z^{(q)}}_{A^2}.
$$
It is easy to see that when $S$ is bounded, the function $B(S)(z)$ is bounded for all $z\in \D^n$.  In fact the Berezin transform is one-to-one and so every bounded operator on $A^p$ is determined by its Berezin transform $B(S)$ (this follows from a simple uniqueness argument with analytic functions).  It is also an easy fact to deduce that if $S$ is compact, then $B(S)(z)\to 0$ as $z\to \p\D^n$.  One of the interesting aspects of operator theory on the Bergman space is that the Berezin transform essentially encapsulates all the behavior of the operator $S$.  In this paper we seek to obtain a characterization of the compactness of operators on $A^p$ in terms of the Berezin transform.

As motivation for our project, we highlight some of the major contributions towards obtaining a characterization of compactness in terms of the Berezin transform.  The first major breakthrough was obtained by Axler and Zheng in the case of the unit disc $\D$ for the standard Bergman space $A^2\left(\D\right)$, see \cite{AZ}.  They showed that when $S$ is a finite sum of finite products of Toeplitz operators, then $S$ is compact if and only if the Berezin transform vanishes as $z$ tends to the boundary of the disc.  This characterization was later extended by Engli{\v{s}} to the case of bounded symmetric domains in $\C^n$, see \cite{E}.  In the case of the unit ball $\B_n$, the Axler and Zheng result was also obtained by Raimondo, \cite{R}.

A much more precise characterization was obtained by Su\'arez in the case of the unit ball $\B_n$.  To state his contribution to the area, we need a little more notation.  Let $\mathcal{T}_{p}$ denote the Toeplitz algebra
generated by $L^\infty$ functions.  By a result of Engli{\v{s}}, \cite{E92}, it is known that the compact operators
on $A^p$ belong to $\mathcal{T}_{p}$.  Su\'arez showed in \cite{Sua} that the compact
operators are precisely those that belong to the Toeplitz algebra and have a vanishing Berezin transform on the boundary of the unit ball.  This was extended to the case of weighted Bergman spaces on the ball by Su\'arez and the authors in \cite{MSW}.

On the polydisc, the question  of compactness in terms of the Berezin transform was first studied by Engli{\v{s}} in \cite{E}.  The main result of that paper is that for an operator $S$ that is a finite sum of finite products of Toeplitz operators, it is compact if and only if its Berezin transform vanishes on $\partial\D^n$.  In \cite{NZ}, Nam and Zheng showed that the same result is true for radial operators $S$, i.e., $S$ is compact if and only if the Berezin transform vanishes on $\partial\D^n$.

The main result of this paper is the following Theorem giving a characterization of the compact operators on the Bergman space of the polydisc in terms of the Toeplitz algebra and the Berezin transform.  In particular it extends the results of \cites{E,NZ} to arbitrary operators.
\begin{thm}
Let $1<p<\infty$ and $S\in\mathcal{L}\left(A^p,A^p\right)$.  Then $S$ is compact if and only if $S\in\mathcal{T}_{p}$ and $\lim_{z\to\p\D^n} B(S)(z)=0$.
\end{thm}

The outline of the paper is as follows.  In Section \ref{Prelim} we remind the reader of the additional notation and facts needed throughout this paper.  In Section \ref{UAMIS} introduces an important uniform algebra and its corresponding maximal ideal space.   Section \ref{CMA} provides the connection between Bergman--Carleson measures and an approximation argument.  In Section \ref{Approximation} we show how to approximate $S\in\mathcal{T}_{p}$ by certain operators that will be useful when computing the essential norm.  Finally, in Section \ref{Characterization}, we prove the main results.  This is accomplished by obtaining several different characterizations of the essential norm of an operator on $A^p$.  The proof strategy is the same as what appears in \cite{Sua} and \cite{MSW}, but requires certain routine modifications and verifications since one is considering the Bergman space over the polydisc now.  For completeness, all details are provided.

Throughout this paper we use the standard notation $A\lesssim B$ to denote the existence of a constant $C$ such that $A\leq C B$.  While $A\approx B$ will mean $A\lesssim B$ and $B\lesssim A$.  The value of a constant may change from line to line, but we will frequently attempt to denote the parameters that the constant depends upon.  The expression $:=$ will mean equal by definition.

The authors wish to thank Daniel Su\'arez for some comments on an earlier draft of this manuscript.  Further thanks are given to a dedicated referee who provided numerous comments that improved the overall presentation of the paper.

\section{Preliminaries}
\label{Prelim}
For $z\in\D$, $\varphi_z$ will denote the automorphism of $\D$ such that $\varphi_z(0)=z$, namely
$$
\varphi_z(w):=\frac{z-w}{1-\overline{z}w}.
$$  
Using this automorphism, we can define the pseudohyperbolic and hyperbolic metrics on $\D$, by
$$
\rho\left(z,w\right):=\abs{\varphi_{z}(w)}=\abs{\frac{z-w}{1-\overline{z}w}}\quad\textnormal{ and }\quad \beta\left(z,w\right):=\frac{1}{2}\log\frac{1+\rho\left(z,w\right)}{1-\rho\left(z,w\right)}.
$$
It is well known that these metrics are invariant under the automorphism group of $\D$.  We let
$$
D\left(z,r\right):=\left\{w\in\D:\beta\left(z,w\right)\leq r\right\}=\left\{w\in\D: \rho\left(z,w\right)\leq \tanh r\right\},
$$
denote the hyperbolic disc centered at $z$ of radius $r$.  Also note the following well known identities for the M\"obius maps on $\D$:
\begin{eqnarray*}
1-\abs{\varphi_z(w)}^2 & = & \frac{(1-\abs{z}^2)(1-\abs{w}^2)}{\abs{1-\overline{z}w}^2},\\
1-\overline{\varphi_z(w)}\varphi_z(\xi) & = & \frac{(1-\abs{z}^2)(1-\overline{w}\xi)}{(1-\overline{z}\xi)(1-\overline{w}z)}.
\end{eqnarray*}

We now extend some of this notation to the polydisc.  For $z\in\D^n$ and $1\leq l\leq n$, $z_l$ will denote the $l^{\textnormal{th}}$ component of the vector $z$.  A sequence, or net, of points in the polydisc $\D^n$ will be denoted by $\{z^{k}\}$, or $\{z^{\omega}\}$. Given $z\in\D^n$, the map $\varphi_z$ will denote the map that exchanges $0$ and $z$, in particular we have,
$$
\varphi_z(w)=\left(\varphi_{z_1}(w_1),\ldots,\varphi_{z_n}(w_n)\right).
$$
For $z\in\D^n$ and $r>0$ we form the set
$$
D\left(z,r\right):=\prod_{l=1}^n D\left(z_l,r\right)
$$
where $D\left(z_l,r\right)$ is the hyperbolic disc in one variable.  For $z,w\in\D^n$ we also will let
$$
\rho\left(z,w\right):=\max_{1\leq l\leq n}\abs{\frac{w_l-z_l}{1-\overline{z}_lw_l}}=\max_{1\leq l\leq n}\rho\left(z_l,w_l\right).
$$
In particular, note that we are using similar notation for both the the disc $\D$ and the polydisc $\D^n$.  The precise usage will be clear from context and should cause no confusion.

The next lemma is well known, and the statement is provided for the reader's ease.  The interested reader can consult the book \cite{Zhu}.
\begin{lm}
\label{Growth}
For $z\in \D$, $s$ real and $t>-1$, let
$$
F_{s,t}(z):=\int_{\D}\frac{\left(1-\abs{w}^2\right)^t}{\abs{1-\overline{w}z}^s}\,dv(w).
$$
Then $F_{s,t}$ is bounded if $s<2+t$ and grows as $(1-\abs{z}^2)^{2+t-s}$ when $\abs{z}\to 1$ if $s>2+t$.
\end{lm}
%Note that in particular, when $s<2+t$, we have that
%\begin{equation}
%\label{GrowthinPoly}
%F_{s,t}(z):=\int_{\D^n}\prod_{l=1}^n\frac{(1-\abs{w_l}^2)^t}{\abs{1-\overline{w}_lz_l}^s}\,dv(w)\lesssim 1.
%\end{equation}

\subsection{Bergman--Carleson Measures for \texorpdfstring{$A^p$}{Bergman Space}}

Unless stated otherwise, a measure will always be a positive, finite, regular, Borel measure.  For $p>1$ a measure $\mu$ on $\D^n$ is a Carleson measure for $A^p$ if there is a constant, independent of $f$, such that
\begin{equation}
\label{Carl_Emb}
\int_{\D^n}\abs{f(z)}^p \,d\mu(z)\lesssim \int_{\D^n}\abs{f(z)}^p \,dv(z).
\end{equation}
Let $i_p$ denote the embedding of $A^p$ into $L^p(\D^n;\mu)$ and the best constant such that \eqref{Carl_Emb} holds will be denoted by $\norm{\mu}^p_{\textnormal{CM}}$.  For a measure $\mu$ we will define the operator
$$
T_\mu f(z):=\int_{\D^n}f(w)\prod_{l=1}^n\frac{1}{\left(1-\overline{w}_lz_l\right)^{2}}\,d\mu(w),
$$
which gives rise to an analytic function for all $f\in H^\infty$.  Note that when $\mu=a\,dv$ we have that $T_\mu=T_a$, and so this definition coincides with the one previously given for a Toeplitz operator.  When $1<p<\infty$, we have that $T_\mu$ is densely defined on $A^p$.  Moreover, $T_\mu$ is bounded from $A^p\to A^p$ if and only if $\mu$ is a Carleson measure for $A^p$.

\begin{lm}[Necessary and Sufficient Conditions for Bergman--Carleson Measures]
\label{CM}
Suppose that $1<p<\infty$.  Let $\mu$ be a measure on $\D^n$ and $r>0$.  The following quantities are equivalent, with constants that depend on $n$, and $r$:
\begin{itemize}
\item[(1)] $\norm{\mu}_{\textnormal{RKM}}:=\sup_{z\in\D^n}\int_{\D^n} \prod_{l=1}^n\frac{(1-\abs{z_l}^2)^{2}}{\abs{1-\overline{z}_l w_l}^{4}}\,d\mu(w)$;
\item[(2)] $\norm{\mu}^p_{\textnormal{CM}}:=\inf\left\{C: \int_{\D^n}\abs{f(z)}^p \,d\mu(z)\leq C \int_{\D^n}\abs{f(z)}^p \,dv(z)\right\}$;
\item[(3)] $\norm{\mu}_{\textnormal{Geo}}=\sup_{z\in\D^n}\frac{\mu\left(D\left(z,r\right)\right)}{\prod_{l=1}^n \left(1-\abs{z_l}^2\right)^{2}}\approx\sup_{z\in\D^n}\frac{\mu\left(D\left(z,r\right)\right)}{v\left(D\left(z,r\right)\right)}$;
\item[(4)] $\norm{T_\mu}_{\mathcal{L}\left(A^p, A^p\right)}$.
\end{itemize}
\end{lm}
Here, RKM denotes that the measure $\mu$ is a reproducing kernel measure.  Observe that condition (1) and (3) are actually independent of the exponent $p=2$ and so, the equivalence with (2) is actually true for all $1<p<\infty$.  Since the condition is independent of the value of $p$, we will refer to a measure which satisfies any of the conditions above as a Bergman--Carleson measure.

Another simple observation one should make at this point is the following.  Suppose $\mu$ is a complex-valued measure such that $\abs{\mu}$, the total variation of the measure, is a Bergman--Carleson measure.  Decompose $\mu$ into its real and imaginary parts and then use the Jordan Decomposition to write $\mu=\mu_1-\mu_2+i\mu_3-i\mu_4$ where each $\mu_j$ is a positive measure and $\abs{\mu}\approx \sum_{j=1}^{4}\abs{\mu_j}$.  We then have that $\abs{\mu_j}$ is Bergman--Carleson with $\norm{\abs{\mu}}_{\textnormal{CM}}\approx\sum_{j=1}^{4}\norm{\mu_j}_{\textnormal{CM}}$.  Using Lemma \ref{CM} we have that $T_\mu$ is a bounded operator on $A^p$ when $\mu$ is a complex-valued measure with $\abs{\mu}$ a Bergman--Carleson measure.

\begin{proof}[Proof of Lemma \ref{CM}]
The equivalence between (1), (2) and (3) is well known, see any of \cites{H,J,M, W}.  Finally, to prove the equivalence with (4), first suppose that (2) holds, then using Fubini's Theorem, we have that for $f,g\in H^\infty$ that
\begin{eqnarray*}
\abs{\ip{T_\mu f}{g}_{A^2}} & = & \abs{\int_{\D^n} f(w)\overline{g(w)}\,d\mu(w)}\\
& \lesssim & \norm{\mu}_{\textnormal{CM}}^2\norm{f}_{A^p}\norm{g}_{A^q}.
\end{eqnarray*}
But, this inequality then implies that $T_\mu: A^p\to A^p$ is bounded.  Here we have identified $\left(A^p\right)^*=A^q$.  Conversely, if $T_\mu$ is bounded, then observe that
$$
T_{\mu}\left(k_{\lambda}^{(p)}\right)(z)=\int_{\D^n}\prod_{l=1}^n\frac{1}{\left(1-z_l\overline{w}_l\right)^{2}}\prod_{l=1}^n\frac{\left(1-\abs{\lambda_l}^2\right)^{\frac{2}{q}}}{\left(1-\overline{\lambda}_lw_l\right)^{2}}\,d\mu(w),
$$
and in particular we have
$$
T_{\mu}\left(k_{\lambda}^{(p)}\right)(\lambda)=\int_{\D^n}\prod_{l=1}^n\frac{\left(1-\abs{\lambda_l}^2\right)^{\frac{2}{q}}}{\abs{1-\overline{\lambda}_lw_l}^{4}}\,d\mu(w).
$$
This computation implies
\begin{eqnarray*}
\int_{\D^n}\prod_{l=1}^n\frac{\left(1-\abs{\lambda_l}^2\right)^{2}}{\abs{1-\overline{\lambda}_lw_l}^{4}}\,d\mu(w)=\ip{T_{\mu} k_\lambda^{(p)}}{k_\lambda^{(q)}}_{A^2} & \leq & \norm{T_\mu}_{\mathcal{L}\left(A^p,A^p\right)}\norm{k_\lambda^{(p)}}_{A^p}\norm{k_\lambda^{(q)}}_{A^q}\\
& \approx & \norm{T_\mu}_{\mathcal{L}\left(A^p,A^p\right)}.
\end{eqnarray*}
\end{proof}
For a Bergman--Carleson measure $\mu$, $1<p<\infty$, and for $f\in L^p(\D^n;\mu)$ define
$$
P_\mu f(z):=\int_{\D^n}f(w)\prod_{l=1}^n\frac{1}{\left(1-\overline{w}_lz_l\right)^{2}}\,d\mu(w).
$$
Based on the computations above, it is easy to see that $P_\mu$ is a bounded operator from $L^p(\D^n;\mu)$ to $A^p$ and $T_\mu=P_{\mu}\circ \imath_p$.

\begin{lm}
\label{CM-Cor}
Let $1<p<\infty$ and suppose that $\mu$ is a Bergman--Carleson measure.  Let $F\subset \D^n$ be a compact set, then
$$
\norm{T_{\mu 1_F}f}_{A^p}\lesssim \norm{T_\mu}_{\mathcal{L}\left(A^p,A^p\right)}^{\frac{1}{q}}\norm{1_F f}_{L^p(\mu)}
$$
where $q=\frac{p}{p-1}$.
\end{lm}
\begin{proof}
It is clear that $T_{1_F\mu}f$ is a bounded analytic function for any $f\in A^p$ since $F$ is compact and $\mu$ is a finite measure.  As in the proof of the previous lemma, we have
\begin{eqnarray*}
\abs{\ip{T_{\mu 1_F} f}{g}_{A^2}} & = & \abs{\int_{\D^n} 1_F(w)f(w)\overline{g(w)}\,d\mu(w)}\\
& \leq & \norm{1_F f}_{L^p\left(\D^n;\mu\right)}\norm{g}_{L^q\left(\D^n;\mu\right)}\\
&\lesssim &  \norm{T_\mu}_{\mathcal{L}\left(A^p,A^p\right)}^{\frac{1}{q}}\norm{1_F f}_{L^p\left(\D^n;\mu\right)}\norm{g}_{A^q}.
\end{eqnarray*}
Taking the supremum over $g\in A^q$ gives the desired result.
\end{proof}

\subsection{Geometric Decompositions of \texorpdfstring{$\D^n$}{the Polydisc}}

In \cite{CR}, Coifman and Rochberg demonstrated that the following decomposition of the disc exists.
\begin{lm}
\label{StandardGeo}
Given $\varrho>0$, there is a family of Borel sets $D_m\subset\D$ and points $\left\{w_m:m\in\N\right\}$ such that
\begin{itemize}
\item[(i)] $D\left(w_m,\frac{\varrho}{4}\right)\subset D_m\subset D\left(w_m,\varrho\right)$ for all $m\in\N$;
\item[(ii)] $D_m\cap D_{m'}=\emptyset$ if $m\neq m'$;
\item[(iii)] $\bigcup_{m} D_m=\D$.
\end{itemize}
\end{lm}

\begin{lm}[Lemma 3.1, \cite{Sua}]
\label{SuaGeo}
There is a positive integer $N$ such that for any $\sigma>0$ there is a covering of $\D$ by Borel sets $\left\{B_j\right\}$ that satisfy:
\begin{itemize}
\item[(i)] $B_j\bigcap B_{j'}=\emptyset$ if $j\neq j'$;
\item[(ii)] Every point of $\D$ belongs to at most $N$ sets $\Omega_\sigma(B_j)=\left\{z:\rho\left(z,B_j\right)\leq\tanh\sigma\right\}$;
\item[(iii)] there is a constant $C\left(\sigma\right)>0$ such that $\textnormal{diam}_{\rho} \,B_j\leq C\left(\sigma\right)$ for all $j\in\N$.
\end{itemize}
\end{lm}

Let $\sigma>0$ and let $k$ be a non-negative integer.  Let $\{B_j\}$ be the covering of the disc that satisfies the conditions of Lemma \ref{SuaGeo} with $(k+1)\sigma$ instead of $\sigma$.  For $0\leq i\leq k$ and $j\geq 1$ we write
$$
F_{0,j}=B_j\quad\textnormal{ and }\quad F_{i+1,j}=\left\{z:\rho\left(z,F_{i,j}\right)\leq\tanh\sigma\right\}.
$$
Then we have,
\begin{lm}[Corollary 3.3, \cite{Sua}]
\label{SuaGeo2}
Let $\sigma>0$ and $k$ be a non-negative integer.  For each $0\leq i\leq k$ the family of sets $\mathcal{F}_{i}=\left\{F_{i,j}: j\in\N\right\}$ forms a covering of $\D$ such that
\begin{itemize}
\item[(i)] $F_{0,j}\bigcap F_{0,j'}=\emptyset$ if $j\neq j'$;
\item[(ii)] $F_{0,j}\subset F_{1,j}\subset\cdots\subset F_{k+1,j}$ for all $j\in\N$;
\item[(iii)] $\rho\left(F_{i,j}, F_{i+1,j}^c\right)\geq\tanh\sigma$ for all $0\leq i\leq k$ and $j\in\N$;
\item[(iv)] every point of $\D$ belongs to no more than $N$ elements of $\mathcal{F}_{i}$;
\item[(v)] $\textnormal{diam}_{\rho} \,F_{i,j}\leq C\left(k,\sigma\right)$ for all $0\leq i\leq k$ and $j\in\N$.
\end{itemize}
\end{lm}

We now need to extend some of these constructions to the polydisc $\D^n$.  The interested reader can see where this is done for more general bounded symmetric domains by Coifman and Rochberg in \cite{CR}.  For completeness, we explicitly provide the construction in the case of the polydisc.
\begin{lm}
\label{StandardGeo_Polydisc}
Given $\varrho>0$, there is a family of Borel sets $D_{\vec{m}}\subset\D^n$ and points $\left\{w_{\vec{m}}:\vec{m}\in\N^n\right\}$ such that
\begin{itemize}
\item[(i)] $D\left(w_{\vec{m}},\frac{\varrho}{4}\right)\subset D_{\vec{m}}\subset D\left(w_{\vec{m}},\varrho\right)$ for all $\vec{m}\in\N^n$;
\item[(ii)] $D_{\vec{m}}\bigcap D_{\vec{m}'}=\emptyset$ if $\vec{m}\neq \vec{m}'$;
\item[(iii)] $\bigcup_{\vec{m}} D_{\vec{m}}=\D^n$.
\end{itemize}
\end{lm}
\begin{proof}
For $\vec{m}\in\N^n$ set $D_{\vec{m}}=\prod_{l=1}^nD_{m_l}$ and $w_{\vec{m}}=\left(w_{m_1},\ldots,w_{m_n}\right)$, where $D_{m_l}$ is the Borel set and $w_{m_l}$ is the point guaranteed by Lemma \ref{StandardGeo}.  For $r>0$, set $D\left(w_{\vec{m}},r\right)=\prod_{l=1}^n D\left(w_{m_{l}},r\right)$.  It is then easy to show that properties (i)-(iii) hold for these sets using Lemma \ref{StandardGeo}.

First, for (i), by Lemma \ref{StandardGeo}, we have that for each $m$ that $D\left(w_m,\frac{\varrho}{4}\right)\subset D_{m}\subset D\left(w_m,\varrho\right)$.  From this it is immediate that we have $D\left(w_{\vec{m}},\frac{\varrho}{4}\right)\subset D_{\vec{m}}\subset D\left(w_{\vec{m}},\varrho\right)$ for all $\vec{m}$.

Next, for (ii), suppose that they do not have empty intersection in general.  Then there exists $\vec{m}\neq\vec{m}'$ such that $D_{\vec{m}}\cap D_{\vec{m}'}\neq\emptyset$.  Let $z\in D_{\vec{m}}\cap D_{\vec{m}'}$, and so $z_l\in D_{m_l}\cap D_{m_l'}$ for all $1\leq l\leq n$.  However, $\vec{m}\neq\vec{m}'$ and so there is an index $l_0$ such that $m_{l_0}\neq m_{l_0}'$.  By Lemma \ref{StandardGeo} we have that $D_{m_{l_0}}\cap D_{m_{l_0}'}=\emptyset$, and so our supposition has lead to a contradiction.  Thus, $D_{\vec{m}}\cap D_{\vec{m}'}=\emptyset$ if $\vec{m}\neq \vec{m}'$ as claimed.

Finally, for (iii), we clearly have that $\bigcup_{\vec{m}} D_{\vec{m}}\subset\D^n$.  Let $z\in\D^n$, then for each $z_l\in\D$, $1\leq l\leq n$, we have a set $D_{m_l}$ such that $z_l\in D_{m_l}$.  But, then we have that $z\in D_{\vec{m}}$ for the appropriate choice of $\vec{m}$ (corresponding to the $z$), and so $\D^n\subset \bigcup_{\vec{m}} D_{\vec{m}}$.
\end{proof}

\begin{rem}
\label{Useful2}
We remark that it is easy to see that when the radius $\varrho$ is fixed for $w\in D_{\vec{m}}$, then we have that $\prod_{l=1}^n\left(1-\abs{w_l}^2\right)\approx\prod_{l=1}^{n}\left(1-\abs{w_{m_l}}^2\right)$ and $\prod_{l=1}^n\abs{1-\overline{z}_lw_l}\approx \prod_{l=1}^n\abs{1-\overline{z}w_{m_l}}$ uniformly in $z\in\D^n$.
\end{rem}

We now take the sets from Lemma \ref{SuaGeo2} to construct important sets for the remainder of the paper.  Let $\vec{j}=(j_1,\ldots, j_n)\in \N^n$.  On the polydisc $\D^n$, for $0\leq i\leq k$, we form the sets
$$
F_{i,\vec{j}}:=\prod_{l=1}^n F_{i,j_l}.
$$
Each $F_{i,\vec{j}}$ is then the product of the sets $F_{i,j_l}$ coming from each component of the polydisc.  We then have the following Corollary.

\begin{lm}
\label{SuaGeo2_Poly}
Let $\sigma>0$ and $k$ be a non-negative integer.  For each $0\leq i\leq k$ the family of sets $\mathcal{F}_{i}=\left\{F_{i,\vec{j}}: \vec{j}\in\N^n\right\}$ forms a covering of $\D^n$ such that
\begin{itemize}
\item[(i)] $F_{0,\vec{j}}\bigcap F_{0,\vec{j}'}=\emptyset$ if $\vec{j}\neq \vec{j}'$;
\item[(ii)] $F_{0,\vec{j}}\subset F_{1,\vec{j}}\subset\cdots\subset F_{k+1,\vec{j}}$ for all $\vec{j}\in\N^n$;
\item[(iii)] $\rho\left(F_{i,\vec{j}}, F_{i+1,\vec{j}}^c\right)\geq\tanh\sigma$ for all $0\leq i\leq k$ and $\vec{j}\in\N^n$;
\item[(iv)] every point of $\D^n$ belongs to no more than $N(n)$ elements of $\mathcal{F}_{i}$;
\item[(v)] $\textnormal{diam}_{\rho} \,F_{i,\vec{j}}\leq C\left(k,\sigma\right)$ for all $0\leq i\leq k$ and $\vec{j}\in\N^n$.
\end{itemize}
\end{lm}

\begin{proof}
All of these properties follow essentially from Lemma \ref{SuaGeo2}.  For (i), we proceed by contradiction.  If $F_{0,\vec{j}}\cap F_{0,\vec{j}'}\neq\emptyset$, then there is a $z$ belonging to both sets.  This implies that $z_l\in F_{0, j_l}\cap F_{0,j_l'}$ for all $1\leq l\leq n$.  However, we have that $\vec{j}\neq\vec{j}'$, and so there is an index $l_0$ such that $j_{l_0}\neq j_{l_0}'$.  This then implies that $F_{0,j_{l_0}}\cap F_{0,j_{l_0}'}\neq\emptyset$, which contradicts (i) from Lemma \ref{SuaGeo2}.  

For (ii), suppose that $z\in F_{i,\vec{j}}$ for some $0\leq i\leq k$.  Then we have $z_l\in F_{i,j_l}$ for $1\leq l\leq n$.  However, by Lemma \ref{SuaGeo2}, we have $F_{i,j_l}\subset F_{i+1,j_l}$ for $0\leq i\leq k$ and for $1\leq l\leq n$.  This gives $z\in F_{i+1,\vec{j}}$, and so $F_{i,\vec{j}}\subset F_{i+1,\vec{j}}$ for $0\leq i\leq k$ and for $\vec{j}\in\N^n$.

For (iii), we give the main idea since the notation becomes cumbersome.  Focus on the case of $n=2$, i.e., the bidisc.  Note that $F_{i+1,\vec{j}}^c=F_{i+1,j_1}^c\times F_{i+1,j_2}^c\cup F_{i+1,j_1}^c\times F_{i+1,j_2}\cup F_{i+1,j_1}\times F_{i+1,j_2}^c$, where this is a disjoint decomposition.   We check the distance between $F_{i,\vec{j}}$ and each of these components.  We compute then
\begin{eqnarray*}
\rho\left(F_{i,\vec{j}}, F_{i+1,j_1}^c\times F_{i+1,j_2}\right) & = & \max\left\{\rho\left(F_{i, j_1}, F_{i+1,j_1}^c\right),\rho\left(F_{i, j_2}, F_{i+1,j_2}\right)\right\}\\
& = & \rho\left(F_{i, j_1}, F_{i+1,j_1}^c\right)\geq\tanh\sigma.
\end{eqnarray*}
Here the last inequality follows from Lemma \ref{SuaGeo2}.  The distance from the other two components are computed identically, with the same lower bound obtained.  The case of general $n$ is similar, in that one will always be left with computing $\rho\left(F_{i, j_l}, F_{i+1,j_l}^c\right)$ for some $1\leq l\leq n$, which is always big by Lemma \ref{SuaGeo2}.

Now for (iv) we have that each point of the disc can belong to no more than $N$ elements of the sets $F_{i,j}$.  Thus, we have that each point of $\D^n$ can belong to no more than $N^n$ sets.  Finally, for (v), we have
$$
\textnormal{diam}_{\rho}\, F_{i,\vec{j}}=\max_{1\leq l\leq n} \textnormal{diam}_{\rho}\, F_{i,j_l}\leq C\left(k,\sigma\right).
$$
\end{proof}

\subsection{Technical Lemmas}

We next turn to proving the key technical estimates that will be useful when approximating the operators.  Key to these estimates is the following well known lemma.

\begin{lm}[Schur's Lemma]
\label{Schur}
Let $(X,\mu)$ and $(X,\nu)$ be measure spaces, $K(x,y)$ a non-negative measurable function on $X\times X$, $1<p<\infty$ and $\frac{1}{p}+\frac{1}{q}=1$.  If $h$ is a positive function on $X$ that is measurable with respect to $\mu$ and $\nu$ and $C_p$ and $C_q$ are positive constants with
\begin{itemize}
\item[] $$\int_{X} K(x,y) h(y)^q\,d\nu(y)\leq C_q h(x)^q\textnormal{ for } \mu\textnormal{-almost every } x;$$
\item[] $$\int_{X} K(x,y) h(x)^p\,d\mu(x)\leq C_p h(y)^p\textnormal{ for } \nu\textnormal{-almost every } y,$$
\end{itemize}
then $Tf(x)=\int_{X} K(x,y) f(y)\,d\nu(y)$ defines a bounded operator $T:L^p\left(X;\nu\right)\to L^p\left(X;\mu\right)$ with $\norm{T}_{L^p\left(X;\nu\right)\to L^p\left(X;\mu\right)}\leq C_q^{\frac{1}{q}}C_p^{\frac{1}{p}}$.
\end{lm}

\begin{lm}
\label{Tech1}
Let $1<p<\infty$ and $\mu$ be a Bergman--Carleson measure.  Suppose that $F_j, K_j\subset\D^n$ are Borel sets such that $\{F_j\}$ are pairwise disjoint and $\rho\left(F_j,K_j\right)>\tanh\sigma\geq \tanh1$ for all $j$.  If $0<\gamma<\min\left\{\frac{1}{2p},\frac{p-1}{p}\right\}$, then
\begin{equation}
\label{Tech-Est1}
\int_{\D^n}\sum_{j}1_{F_j}(z) 1_{K_j}(w)\prod_{l=1}^n\frac{\left(1-\abs{w_l}^2\right)^{-\frac{1}{p}}}{\abs{1-\overline{z}_lw_l}^{2}}\,d\mu(w)\lesssim \norm{T_\mu}_{\mathcal{L}\left(A^p, A^p\right)} \left(1-\delta^{2}\right)^{\gamma}\prod_{l=1}^n\left(1-\abs{z_l}^2\right)^{-\frac{1}{p}}
\end{equation}
where $\delta=\tanh\frac{\sigma}{2}$ and the implied constants depend on $n$ and $p$.
\end{lm}

\begin{proof}
By Lemma \ref{StandardGeo_Polydisc}, with $\varrho=\tanh\frac{1}{10}$ there is a sequence of points $\{w_{\vec{m}}\}$ and Borel sets $D_{\vec{m}}$. Standard computations show that there is a constant depending on the dimension and $p$ such that
\begin{equation}
\label{Geo-Obs}
\prod_{l=1}^n\frac{(1-\abs{w_l}^2)^{-\frac{1}{p}}}{\abs{1-\overline{z}_lw_l}^{2}}\approx \prod_{l=1}^n\frac{(1-\abs{w_{m_l}}^2)^{-\frac{1}{p}}}{\abs{1-\overline{z}_lw_{m_l}}^{2}}%\lesssim \prod_{l=1}^n\frac{(1-\abs{w_l}^2)^{-\frac{1}{p}}}{\abs{1-\overline{z}_lw_l}^{2}}
\end{equation}
for all $w\in D_{\vec{m}}$ and $z\in \D^n$.  Now by the Bergman--Carleson measure condition, we have a constant $C$ such that
\begin{equation}
\label{CM-Obs}
\mu\left(D_{\vec{m}}\right)\lesssim \norm{T_\mu}_{\mathcal{L}\left(A^p, A^p\right)}v\left(D_{\vec{m}}\right).
\end{equation}
If $z\in F_j$ and $w\in K_j$, with $\rho\left(z,w\right)>\tanh\sigma$, then $K_j\subset \D^n\setminus D\left(z,\tanh\sigma\right)$ and
$$
\sum_{j} 1_{F_j}(z) 1_{K_j}(w)\leq \sum_{j} 1_{F_j}(z) 1_{\D^n\setminus D\left(z,\tanh\sigma\right)}(w).
$$
For simplicity, in the above display we write for $r>0$, $D(z,r)^{c}:=\D^n\setminus D\left(z,\tanh r\right)$ and going forward in the proof of the lemma, we let $\phi\left(z,w\right):=\prod_{l=1}^n\frac{(1-\abs{w_l}^2)^{-\frac{1}{p}}}{\abs{1-\overline{z}_lw_l}^{2}}$.  Then we have that the integral in \eqref{Tech-Est1} is controlled by
\begin{eqnarray*}
\int_{\D^n}\sum_{j}1_{F_j}(z) 1_{K_j}(w)\prod_{l=1}^n\frac{\left(1-\abs{w_l}^2\right)^{-\frac{1}{p}}}{\abs{1-\overline{z}_lw_l}^{2}}\,d\mu(w) & \leq & \sum_{j} 1_{F_j}(z)\int_{\D^n} 1_{D\left(z,\sigma\right)^{c}}(w) \phi\left(z,w\right)\,d\mu(w)\\
 & := & \sum_{j} 1_{F_j}(z) I_z.
\end{eqnarray*}
Since the sets $\left\{F_j\right\}$ are disjoint, it is enough to prove the desired estimate on each $I_z$.  We now estimate each integral $I_z$ appearing above.  First note,
\begin{eqnarray*}
I_z:=\int_{\D^n} 1_{D\left(z,\sigma\right)^c}(w)\phi\left(z,w\right)\, d\mu(w) & = & \sum_{\vec{m}}\int_{D_{\vec{m}}} 1_{D(z,\sigma)^c}(w)\phi\left(z,w\right)\,d\mu(w)\\
 & \leq & \sum_{D_{\vec{m}}\cap D(z,\sigma)^{c}\neq\emptyset} \int_{D_{\vec{m}}}\phi\left(z,w\right)\,d\mu(w)\\
 & \approx & \sum_{D_j\cap D(z,\sigma)^{c}\neq\emptyset} \phi\left(z,w_{\vec{m}}\right)\int_{D_{\vec{m}}}\,d\mu(w)\\
 & \lesssim & \norm{T_\mu}_{\mathcal{L}\left(A^p, A^p\right)} \sum_{D_{\vec{m}}\cap D(z,\sigma)^{c}\neq\emptyset} \phi(z,w_{\vec{m}})\int_{D_{\vec{m}}}\,dv(w)\\
 & \approx & \norm{T_\mu}_{\mathcal{L}\left(A^p, A^p\right)} \sum_{D_{\vec{m}}\cap D(z,\sigma)^{c}\neq\emptyset} \int_{D_{\vec{m}}}\phi\left(z,w\right)\,dv(w).
\end{eqnarray*}
In the above estimates we used \eqref{Geo-Obs} and \eqref{CM-Obs}.

If $D_{\vec{m}}\cap D\left(z,\sigma\right)^{c}\neq\emptyset$ and $w\in D_{\vec{m}}$, then
$\rho\left(w,D\left(z,\sigma\right)^c\right)\leq\textnormal{diam}_{\rho} \,D_{\vec{m}}\leq 2\varrho=2\tanh\frac{1}{10}\approx\frac{1}{5}$ and since we have
$$
\rho\left(D\left(z,\frac{\tanh\sigma}{2}\right), D\left(z,\sigma\right)^{c}\right)=\frac{\tanh\sigma}{2}\approx\tanh\frac{\sigma}{2}\geq\frac{1}{2}
$$
we have that $D_{\vec{m}}\cap D\left(z,\frac{\tanh\sigma}{2}\right)=\emptyset$ whenever $D_{\vec{m}}\cap D\left(z,\sigma\right)^{c}\neq\emptyset$.
And so we have,
\begin{eqnarray*}
I_z & \lesssim & \norm{T_\mu}_{\mathcal{L}\left(A^p, A^p\right)} \sum_{D_{\vec{m}}\cap D(z,\sigma)^{c}\neq\emptyset}
\int_{D_{\vec{m}}}\phi\left(z,w\right)\,dv(w)\\
 & = & \norm{T_\mu}_{\mathcal{L}\left(A^p, A^p\right)}
\int_{\D^n}1_{D\left(z,\frac{\tanh\sigma}{2}\right)^c}(w)\phi\left(z,w\right)\,dv(w).
\end{eqnarray*}
Thus, continuing the estimate, we have
\begin{eqnarray*}
I_z  & \lesssim &  \norm{T_\mu}_{\mathcal{L}\left(A^p, A^p\right)} \int_{\D^n}1_{D\left(z,\tanh\frac{\sigma}{2}\right)^c}(w)\phi\left(z,w\right)\,dv(w)\\
& = &  \norm{T_\mu}_{\mathcal{L}\left(A^p, A^p\right)} \int_{\D^n}1_{D\left(z,\tanh\frac{\sigma}{2}\right)^c}(w)\prod_{l=1}^n\frac{\left(1-\abs{w_l}^2\right)^{-\frac{1}{p}}}{\abs{1-\overline{z}_lw_l}^{2}}\,dv(w)\\
& = &  \norm{T_\mu}_{\mathcal{L}\left(A^p, A^p\right)} \int_{\D^n}1_{D\left(0,\delta\right)^c}(w)\prod_{l=1}^n \frac{\left(1-\abs{w_l}^2\right)^{-\frac{1}{p}}\left(1-\abs{z_l}^2\right)^{-\frac{1}{p}}}{\abs{1-\overline{z}_lw_l}^{2-\frac{2}{p}}}\,dv(w).
\end{eqnarray*}
Here we have used the change of variable $w'=\varphi_z(w)$ and an obvious computation.  We are also letting $D\left(0,\delta\right)^c=\D^n\setminus D\left(0,\delta\right)$.  To complete the lemma, it suffices to prove
\begin{equation}
\label{LastEstimate}
\int_{\D^n}1_{D\left(0,\delta\right)^c}(w)\prod_{l=1}^n \frac{\left(1-\abs{w_l}^2\right)^{-\frac{1}{p}}}{\abs{1-\overline{z}_lw_l}^{2-\frac{2}{p}}}\,dv(w)\lesssim  \left(1-\delta^{2}\right)^{\gamma},
\end{equation}
with implied constant depending on the dimension and $p$.  Indeed, once we have \eqref{LastEstimate} we then easily conclude 
$$
\int_{\D^n}\sum_{j}1_{F_j}(z) 1_{K_j}(w)\prod_{l=1}^n\frac{\left(1-\abs{w_l}^2\right)^{-\frac{1}{p}}}{\abs{1-\overline{z}_lw_l}^{2}}\,d\mu(w)\lesssim \norm{T_\mu}_{\mathcal{L}\left(A^p, A^p\right)}\left(1-\delta^{2}\right)^{\gamma}\prod_{l=1}^n\left(1-\abs{z_l}^2\right)^{-\frac{1}{p}}
$$
completing the proof of the lemma.

Turning now to \eqref{LastEstimate}, it is easy to see that $D\left(0,\delta\right)^c$ set splits into $2^{n}-1$ sets of the form $\prod_{l=1}^n D\left(0,\delta\right)^{\tau_l}$, where $\tau_l\in\{c,o\}$, with $c$ denoting ``taking the complement'' and $o$ denoting ``taking the original set.''  Furthermore, in these decompositions, we have for at least one $l$ from $1\leq l\leq n$ that $\tau_l=c$.  For each of these components we can obtain a sufficient estimate.  

Before continuing the estimate of \eqref{LastEstimate}, we do some auxiliary computations.  Pick a number $a=a(p)$ satisfying
$$
1<a<p\textnormal{ and } a\left(2-\frac{1}{p}\right)<2.
$$
Note that the second condition can be rephrased as $2p<a'$, so it is clear that we can select the number $a$ with the desired properties.  We then apply H\"older's inequality with $\frac{1}{a}+\frac{1}{a'}=1$ to see that for each $1\leq l\leq n$
$$
\int_{\abs{w_l}>\delta} \frac{\left(1-\abs{w_l}^2\right)^{-\frac{1}{p}}}{\abs{1-\overline{z}_lw_l}^{2-\frac{2}{p}}}\,dv(w_l)\leq \left(\int_{\abs{w_l}>\delta} \frac{\left(1-\abs{w_l}^2\right)^{-\frac{a}{p}}}{\abs{1-\overline{z}_lw_l}^{a\left(2-\frac{2}{p}\right)}}\,dv(w_l)\right)^{\frac{1}{a}} \left(v\{w_l:\abs{w_l}>\delta\}\right)^{\frac{1}{a'}}.
$$
We have that $\left(v\{w:\abs{w}>\delta\}\right)^{\frac{1}{a'}}=(1-\delta^{2})^{\frac{1}{a'}}$, and by Lemma \ref{Growth} with $t=-\frac{a}{p}$ and $s=a\left(2-\frac{2}{p}\right)$ we have
\begin{eqnarray*}
\left(\int_{\abs{w_l}>\delta} \frac{\left(1-\abs{w_l}^2\right)^{-\frac{a}{p}}}{\abs{1-\overline{z}_lw_l}^{a\left(2-\frac{2}{p}\right)}}\,dv(w_l)\right)^{\frac{1}{a}} & \leq & \left(\int_{\D} \frac{\left(1-\abs{w_l}^2\right)^{-\frac{a}{p}}}{\abs{1-\overline{z}_lw_l}^{a\left(2-\frac{2}{p}\right)}}\,dv(w_l)\right)^{\frac{1}{a}}\lesssim 1
\end{eqnarray*}
since $s=a\left(2-\frac{2}{p}\right)=a\left(2-\frac{1}{p}\right)-\frac{a}{p}<2-\frac{a}{p}=2+t$ by choice of $a$.  Thus, we have, recalling that $\gamma=\frac{1}{a'}$, 
\begin{equation}
\label{LastEstimate2}
\int_{\abs{w_l}>\delta} \frac{\left(1-\abs{w_l}^2\right)^{-\frac{1}{p}}}{\abs{1-\overline{z}_lw_l}^{2-\frac{2}{p}}}\,dv(w_l)\lesssim \left(1-\delta^{2}\right)^{\gamma}\quad\forall 1\leq l\leq n,
\end{equation}
with the restrictions on $a$ giving the corresponding restrictions on $\gamma$ in the statement of the lemma.  Also note that using Lemma \ref{Growth} we have
\begin{equation}
\label{LastEstimate3}
\int_{\abs{w_l}\leq\delta} \frac{\left(1-\abs{w_l}^2\right)^{-\frac{1}{p}}}{\abs{1-\overline{z}_lw_l}^{2-\frac{2}{p}}}\,dv(w_l)\lesssim 1\quad\forall 1\leq l\leq n.
\end{equation}
It is now easy to conclude \eqref{LastEstimate}.  Indeed, let $O\subset\left\{1,\ldots, n\right\}$ be where $\tau_l=o$, and let $C\subset\left\{1,\ldots,n\right\}$ where $\tau_l=c$.  Note that $O\cup C=\left\{1,\ldots,n\right\}$, $O\cap C=\emptyset$ and that the cardinality of $C$ is at least 1.  There are $2^n-1$ such sets $C$.  Abusing notation, we can write
$$
\prod_{l=1}^n D\left(0,\delta\right)^{\tau_l}=\prod_{l\in O}D\left(0,\delta\right)\prod_{l\in C}D\left(0,\delta\right)^{c}.
$$
Using \eqref{LastEstimate2} and \eqref{LastEstimate3} we have,
\begin{equation}
\left(\prod_{l\in O}\int_{D\left(0,\delta\right)} \frac{\left(1-\abs{w_l}^2\right)^{-\frac{1}{p}}}{\abs{1-\overline{z}_lw_l}^{2-\frac{2}{p}}}\,dv(w_l)\right)\left(\prod_{l\in C}\int_{D\left(0,\delta\right)^{c}} \frac{\left(1-\abs{w_l}^2\right)^{-\frac{1}{p}}}{\abs{1-\overline{z}_lw_l}^{2-\frac{2}{p}}}\,dv(w_l)\right)\lesssim \left(1-\delta^{2}\right)^{\gamma\,\textnormal{card}(C)},
\end{equation}
and so  
$$
\int_{\prod_{l=1}^n D\left(0,\delta\right)^{\tau_l}}\prod_{l=1}^n \frac{\left(1-\abs{w_l}^2\right)^{-\frac{1}{p}}}{\abs{1-\overline{z}_lw_l}^{2-\frac{2}{p}}}\,dv(w)\lesssim \left(1-\delta^{2}\right)^{\gamma\,\textnormal{card}(C)}.
$$
Consequently, 
\begin{eqnarray*}
\int_{\D^n}1_{D\left(0,\delta\right)^c}(w)\prod_{l=1}^n \frac{\left(1-\abs{w_l}^2\right)^{-\frac{1}{p}}}{\abs{1-\overline{z}_lw_l}^{2-\frac{2}{p}}}\,dv(w) & = & \sum_{C}\int_{\prod_{l=1}^n D\left(0,\delta\right)^{\tau_l}}\prod_{l=1}^n \frac{\left(1-\abs{w_l}^2\right)^{-\frac{1}{p}}}{\abs{1-\overline{z}_lw_l}^{2-\frac{2}{p}}}\,dv(w)\\
& \lesssim & \sum_{C} \left(1-\delta^{2}\right)^{\gamma\,\textnormal{card}(C)}\lesssim \left(1-\delta^{2}\right)^{\gamma}.
\end{eqnarray*}
This gives \eqref{LastEstimate} and we are done.
\end{proof}

\begin{lm}
\label{Tech2}
Let $1<p<\infty$ and $\mu$ be a Bergman--Carleson measure.  Suppose that $F_j, K_j\subset\D^n$ are Borel sets and $a_j\in L^\infty$ and $b_j\in L^\infty(\D^n;\mu)$ are functions of norm at most 1 for all $j$.  If
\begin{itemize}
\item[(i)] $\rho\left(F_j, K_j\right)\geq\tanh\sigma\geq \tanh1$;
\item[(ii)] $\textnormal{supp}\, a_j\subset F_j$ and $\textnormal{supp}\, b_j\subset K_j$;
\item[(iii)] every $z\in\D^n$ belongs to at most $N$ of the sets $F_j$
\end{itemize}
then $\sum_{j} M_{a_j} P_{\mu} M_{b_j}$ is a bounded operator from $A^p$ to $L^p$ and there is a function $\beta_{p}\left(\sigma\right)\to 0$ when $\sigma\to\infty$ such that
\begin{equation}
\label{Tech2-Est1}
\norm{\sum_{j} M_{a_j} P_{\mu} M_{b_j}f}_{L^p}\leq N\beta_{p}\left(\sigma\right)\norm{T_\mu}_{\mathcal{L}\left(A^p, A^p\right)}\norm{f}_{A^p}
\end{equation}
and for every $f\in A^p$ we have
\begin{equation}
\label{Tech2-Est2}
\sum_{j}\norm{M_{a_j} P_{\mu} M_{b_j} f}_{L^p}^p\leq N\beta_{p}^p\left(\sigma\right)\norm{T_\mu}_{\mathcal{L}\left(A^p, A^p\right)}^p\norm{f}^{p}_{A^p}.
\end{equation}
\end{lm}

\begin{proof}
Since $\mu$ is a Bergman--Carleson measure we have that $\imath_p: A^p\to L^p(\D^n;\mu)$ is bounded and to prove the lemma it is then enough to prove the following two estimates:
\begin{equation}
\label{Tech2-Est1-Red}
\norm{\sum_{j} M_{a_j} P_{\mu} M_{b_j}f}_{L^p}\leq N\kappa_{p}\left(\delta\right)\norm{T_\mu}_{\mathcal{L}\left(A^p, A^p\right)}^{1-\frac{1}{p}}\norm{f}_{L^p(\D^n;\mu)},
\end{equation}
and
\begin{equation}
\label{Tech2-Est2-Red}
\sum_{j} \norm{M_{a_j} P_{\mu} M_{b_j} f}_{L^p}^p\leq N\kappa_{p}^p\left(\delta\right)\norm{T_\mu}_{\mathcal{L}\left(A^p, A^p\right)}^{p-1}\norm{f}^{p}_{L^p(\D^n;\mu)}
\end{equation}
where $\delta=\tanh\frac{\sigma}{2}$ and $\kappa_{p}\left(\delta\right)\to 0$ as $\delta\to 1$.  Estimates \eqref{Tech2-Est1-Red} and \eqref{Tech2-Est2-Red} imply \eqref{Tech2-Est1} and \eqref{Tech2-Est2} via an application of Lemma \ref{CM}.

First, consider the case when $N=1$, and so the sets $\{F_j\}$ are pairwise disjoint.  Set
$$
\Phi\left(z,w\right)=\sum_{j} 1_{F_j}(z) 1_{K_j}(w) \prod_{l=1}^n\frac{1}{\abs{1-\overline{z}_lw_l}^{2}}.
$$
Suppose now that $f\in L^p\left(\D^n;\mu\right)$, $\norm{a_j}_{L^\infty}\leq 1$ and $\norm{b_j}_{L^\infty(\D^n;\mu)}\leq 1$.  Easy estimates show
\begin{eqnarray*}
\abs{\sum_{j} M_{a_j} P_{\mu} M_{b_j} f(z)} & = & \abs{\sum_{j} a_j(z)\int_{\D^n} b_j(w) f(w)\prod_{l=1}^n\frac{1}{\left(1-\overline{w}_l z_l\right)^{2}} \,d\mu(w)}\\
& \leq & \int_{\D^n} \Phi\left(z,w\right) \abs{f(w)}\,d\mu(w).
\end{eqnarray*}
Thus, it suffices to prove the operator with kernel $\Phi\left(z,w\right)$ is bounded between the necessary spaces.  Set $h(z)=\prod_{l=1}^n\left(1-\abs{z_l}^2\right)^{-\frac{1}{pq}}$ and observe that Lemma \ref{Tech1} implies
$$
\int_{\D^n} \Phi\left(z,w\right) h(w)^q \,d\mu(w)\lesssim \norm{T_\mu}_{\mathcal{L}\left(A^p, A^p\right)} \left(1-\delta^{2}\right)^{\gamma}h(z)^{q}.
$$
While Lemma \ref{Growth} plus a simple computation implies
$$
\int_{\D^n}\Phi\left(z,w\right) h(z)^{p} \,dv(z)\lesssim h(w)^{p}.
$$
Indeed, since $F_j$ are disjoint and form a cover of $\D^n$ (since $N=1$) we have
\begin{eqnarray*}
\int_{\D^n}\Phi\left(z,w\right) h(z)^{p} \,dv(z) & = & \int_{\D^n}\prod_{l=1}^n \frac{\left(1-\abs{z_l}^2\right)^{-\frac{1}{q}}}{\abs{1-\overline{z}_lw_l}^2}\, dv(z)\approx \prod_{l=1}^n(1-\abs{w_l}^2)^{-\frac{1}{q}}=h(w)^p.
\end{eqnarray*}

Schur's Lemma, Lemma \ref{Schur}, then implies that the operator with kernel $\Phi\left(z,w\right)$ is bounded from $L^p\left(\D^n;\mu\right)$ to $L^p$ with norm controlled by a constant $C\left(n,p\right)$ times $k_{p}\left(\delta\right)\norm{T_\mu}_{\mathcal{L}\left(A^p, A^p\right)}^{1-\frac{1}{p}}$.  We thus have \eqref{Tech2-Est1-Red} when $N=1$.  Since the sets $F_j$ are disjoint in this case, then we also have \eqref{Tech2-Est2-Red} since
$$
\sum_{j} \norm{M_{a_j} P_{\mu} M_{b_j} f}_{L^p}^p=\norm{\sum_{j} M_{a_j}P_{\mu} M_{b_j} f}_{L^p}^p.
$$

Now suppose that $N>1$.  Let $z\in \D^n$ and let $S(z)=\left\{j: z\in F_j\right\}$, ordered according to the index $j$.
Each $F_j$ admits a disjoint decomposition $F_j=\bigcup_{k=1}^{N} A_{j}^{k}$ where $A_{j}^{k}$ is the set of $z\in F_j$ such that
$j$ is the $k^{\textnormal{th}}$ element of $S(z)$.
We then have that for $1\leq k\leq N$ the sets $\left\{A_{j}^k: j\geq 1\right\}$ are pairwise disjoint.  Hence, we can apply the computations from above to conclude the following:
\begin{eqnarray*}
\sum_{j} \norm{M_{a_j} P_{\mu} M_{b_j} f}_{L^p}^p & = & \sum_{j} \sum_{k=1}^{N} \norm{M_{a_j 1_{A^k_j}} P_\mu M_{b_j} f}_{L^p}^p\\
 & = & \sum_{k=1}^{N}\sum_{j} \norm{M_{a_j 1_{A^k_j}} P_\mu M_{b_j} f}_{L^p}^p\\
 & \lesssim & N k_{p}^p\left(\delta\right)\norm{T_\mu}_{\mathcal{L}\left(A^p, A^p\right)}^{p-1}\norm{f}_{A^p}^p.
\end{eqnarray*}
This gives \eqref{Tech2-Est2-Red} and \eqref{Tech2-Est1-Red} follows from similar computations.
\end{proof}

\begin{rem}
\label{Rem_Countable}
Note that Lemmas \ref{Tech1} and \ref{Tech2} are stated for arbitrary countable collection of sets $F_j$ and $K_j$.  However, when we apply then, they will be applied to the sets $\left\{F_{\vec{j}}\right\}$ from Lemma \ref{SuaGeo2_Poly} and the complement of an enlargement of these sets.
\end{rem}

\section{A Uniform Algebra and Its Maximal Ideal Space}
\label{UAMIS}

We consider the algebra $\mathcal{A}$ of all bounded functions that are uniformly continuous from the metric space $\left(\D^n,\rho\right)$ into the metric space $\left(\C,\abs{\,\cdot\,}\right)$.  We then associate to $\mathcal{A}$ its maximal ideal space $M_{\mathcal{A}}$ which is the set of all non-zero multiplicative linear functionals from $\mathcal{A}$ to $\C$.  Endowed with the weak-star topology, this is a compact Hausdorff space.  Via the Gelfand transform we can view the elements of $\mathcal{A}$ as continuous functions on $M_{\mathcal{A}}$ as given by $\hat{a}\left(f\right)=f\left(a\right)$ where $f$ is a multiplicative linear functional.  Since $\mathcal{A}$ is a commutative $C^*$ algebra, the Gelfand transform is an isomorphism.  It is also obvious that point evaluation is a multiplicative linear functional, and so $\D^n\subset M_{\mathcal{A}}$.  Moreover, since $\mathcal{A}$ is a $C^*$ algebra we have that $\D^n$ is dense in $M_{\mathcal{A}}$.  Also, one can see that the Euclidean topology on $\D^n$ agrees with the topology induced by $M_{\mathcal{A}}$.

We next state several lemmas and facts that will be useful going forward.  For a set $E\subset M_{\mathcal{A}}$, the closure of $E$ in the space $M_{\mathcal{A}}$ will be denoted $\overline{E}$.  Note that if $E\subset r\D^n$ where $0<r<1$ then this closure is the same as the Euclidean closure.

\begin{lm}
\label{Separation}
Let $E, F\subset\D^n$.  Then $\overline{E}\cap\overline{F}=\emptyset$ if and only if $\rho\left(E,F\right)>0$.
\end{lm}
\begin{proof}
If $\overline{E}\cap\overline{F}=\emptyset$, then by Tietsze's Theorem there is $f\in\mathcal{A}$ such that $f\equiv 1$ on $E$ and $f\equiv 0$ on $F$.  The uniform continuity of $f$ on $\D^n$ with respect to the metric $\rho$ gives that
$$
\rho\left( E,F\right)=\rho\left(\overline{E}\cap\D^n,\overline{F}\cap\D^n\right)>0.
$$
Conversely, suppose $\rho\left(E,F\right)>0$, and set $f(z)=\rho\left(z,E\right)$.  Then we have that $f\in\mathcal{A}$ and that it separates the points $\overline{E}$ from $\overline{F}$, and so $\overline{E}\cap\overline{F}=\emptyset$.
\end{proof}

\begin{lm}
\label{Invariance}
Let $z,w,\xi\in\D^n$.  Then there is a positive constant such that
$$
\rho\left(\varphi_z(\xi),\varphi_w(\xi)\right)\lesssim \max_{1\leq l\leq n}\frac{\rho\left(z_l,w_l\right)}{\left(1-\abs{\xi_l}^2\right)^2}\lesssim\rho\left(z,w\right)\prod_{l=1}^n\frac{1}{\left(1-\abs{\xi_l}^2\right)^2}.
$$
\end{lm}
\begin{proof}
In the case of the disc we have that this is true.  Namely, 
$$
\rho\left(\varphi_{z_l}(\xi_l),\varphi_{w_l}(\xi_l)\right)\lesssim \frac{\rho\left(z_l,w_l\right)}{\left(1-\abs{\xi_l}^2\right)^2}.
$$
See any of \cites{MSW, Sua,Sua2,Sua3} for the proof of this fact.  Using this, we then have that
$$
\rho\left(\varphi_z(\xi),\varphi_w(\xi)\right)=\max_{1\leq l\leq n} \rho\left(\varphi_{z_l}(\xi_l),\varphi_{w_l}(\xi_l)\right)\lesssim \max_{1\leq l\leq n}\frac{\rho\left(z_l,w_l\right)}{\left(1-\abs{\xi_l}^2\right)^2}.
$$
To see the last inequality, note that $\prod_{l=1}^n\left(1-\abs{\xi_l}^2\right)^2\leq \left(1-\abs{\xi_l}^2\right)^2$ for all $l=1,\ldots, n$.
\end{proof}

The next lemma is a translation to the polydisc of a result from Su\'arez \cite{Sua3} on the unit disc.  It is easy to see that the proof by Su\'arez in \cite{Sua3}*{Theorem 4.1} is abstract enough to include much more general domains, such as the ball and polydisc.  For completeness however, we provide a proof. 

\begin{lm}
\label{Extend}
Let $\left(E,d\right)$ be a metric space and $f:\D^n\to E$ be a continuous map.  Then $f$ admits a continuous extension from $M_{\mathcal{A}}$ into $E$ if and only if $f$ is $\left(\rho,d\right)$ uniformly continuous  and $\overline{f(\D^n)}$ is compact.
\end{lm}

\begin{proof}
Assume that $f$ is uniformly $\left(\rho, d\right)$ continuous on $\D^n$ and $\overline{f(\D^n)}$ is compact.  Let $x\in M_{\mathcal{A}}$ and set
$$
F(x)=\left\{e\in E: f(z^{\omega})\to e\,\textnormal{ for some net }\, z^{\omega}\to x, z^{\omega}\in\D^n\right\}.
$$
Since $\overline{f(\D^n)}$ is compact, we have that $F(x)$ is nonempty.  The function $F(x)$ defined on $M_{\mathcal{A}}$ is multi-valued, and a diagonalization argument shows that $f$ can be extended continuously to $M_{\mathcal{A}}$ if and only if $F(x)$ is single-valued for every $x\in M_{\mathcal{A}}$.  

To show that it is single-valued, we proceed by contradiction.  Let $x\in M_{\mathcal{A}}$ and suppose that $e_1, e_2\in F(x)$ with $d\left(e_1, e_2\right)>0$.  Let $B_r(e)$ denote the ball in $E$ of radius $r$ and center $e$, and consider the sets
$$
V_i=\left\{z\in\D^n: f(z)\in B_{\frac{d\left(e_1,e_2\right)}{4}}(e_i))\right\}\quad i=1,2.
$$ 
Since $e_i\in F(x)$, we have that $x\in \overline{V_i}^{M_{\mathcal{A}}}$ for $i=1,2$.  By Lemma \ref{Separation}, we have that $\rho\left(V_1,V_2\right)=0$.  However, 
$$
d\left(f\left(V_1\right), f\left( V_2\right)\right)\geq d\left(B_{\frac{d\left(e_1,e_2\right)}{4}}(e_1),B_{\frac{d\left(e_1,e_2\right)}{4}}(e_2)\right)\geq\frac{1}{2}d\left(e_1,e_2\right)>0.
$$
But, $f$ is uniformly $\left(\rho,d\right)$ continuous, and so this last inequality implies that $\rho\left(V_1,V_2\right)>0$.  This gives the desired contradiction, and so $F(x)$ is single-valued.

For the converse, suppose $f$ admits a continuous extension from $M_{\mathcal{A}}$ into $E$.  Since $\D^n$ is dense in $M_{\mathcal{A}}$, $\overline{f(\D^n)}=f(M_{\mathcal{A}})$, and so $\overline{f(\D^n)}$ is compact.  It remains to show that $f$ is uniformly $\left(\rho,d\right)$ continuous.  If $f$ is not uniformly $\left(\rho,d\right)$ continuous, there are two sequences $z^{k}, w^{k}\in\D^n$ such that $\rho\left(z^{k},w^{k}\right)\to 0$, but $d\left(f(z^{k}), f(w^{k})\right)\geq\delta>0$ for all $k$.  By continuity of $f$ on $\D^n$, we can not have the sequence accumulate on $\D^n$.  Let $x\in \overline{\{z^{k}\}}^{M_{\mathcal{A}}}\setminus \D^n$ and let $\{z^{\omega}\}$ be a subnet of $\{z^{k}\}$ that tends to the point $x$.  Every element of $\{z^{\omega}\}$ is given by $z^{k(\omega)}$, and so we can select a corresponding subnet $\left\{w^{\omega}\right\}$ of the sequence $\{w^{k}\}$ with the property
\begin{equation*}
\rho\left(z^{\omega}, w^{\omega}\right)\to 0\quad\textnormal{ and }\quad d\left( f(z^{\omega}),f(w^{\omega})\right)\geq\delta\quad \forall\omega.
\end{equation*}
Since the subnet $\{z^{\omega}\}$ tends to $x$, we have by the first condition above that the subnet $\{w^{\omega}\}$ tends to $x$ as well.  This is because the first condition implies that $g(w^{\omega})\to g(x)$ for all $g\in\mathcal{A}$.  But, since $f$ is continuous on $M_{\mathcal{A}}$ we have
$$
\lim f(w^{\omega})=f(x)=\lim f(z^{\omega})
$$
which contradicts the second condition above.  Thus, we have that $f$ is uniformly $\left(\rho,d\right)$ continuous.
\end{proof}

Let $x\in M_{\mathcal{A}}$ and suppose that $\{z^{\omega}\}$ is a net in $\D^n$ that converges to $x$.  By compactness, the net $\{\varphi_{z^{\omega}}\}$ in the product space $M^{\D^n}_{\mathcal{A}}$ admits a convergent subnet $\{\varphi_{z^{\omega_\tau}}\}$.  This means there is a function $\varphi:\D^n\to M_{\mathcal{A}}$ such that $f\circ \varphi_{z^{\omega_{\tau}}}\to f\circ\varphi$ for all $f\in \mathcal{A}$ and pointwise on $\D^n$.  We now show that the whole net converges to $\varphi$ and $\varphi$ is independent of the net.  Suppose that $\{w^{\gamma}\}$ is another net in $\D^n$ converging to $x$ such that $\varphi_{w^{\gamma}}$ tends to some $\psi\in M_{\mathcal{A}}^{\D^n}$.  If there is a $\xi\in\D^n$, such that $\varphi(\xi)\neq\psi(\xi)$, then there are tails of both nets such that the sets
$$
E=\left\{\varphi_{z^{\omega_{\tau}}}(\xi):\tau\geq\tau_0\right\}\textnormal{ and } F=\left\{\varphi_{w^{\gamma}}(\xi):\gamma\geq\gamma_0\right\}
$$
have disjoint closures in $M_{\mathcal{A}}$.  By Lemma \ref{Separation} we have that $\rho\left(E,F\right)>0$, but by Lemma \ref{Invariance} we have
\begin{eqnarray*}
\rho\left(E,F\right) & = & \inf\left\{\rho\left(\varphi_{z^{\omega_{\tau}}}(\xi),\varphi_{w^{\gamma}}(\xi)\right):\tau\geq\tau_0,\gamma\geq\gamma_0\right\}\\
 & \lesssim & \prod_{l=1}^n\frac{1}{\left(1-\abs{\xi_l}^2\right)^2}\inf\left\{\rho\left(z^{\omega_{\tau}}, w^{\gamma}\right):\tau\geq\tau_0,\gamma\geq\gamma_0\right\}=0.
\end{eqnarray*}
The last inequality holds since $\{w^{\gamma}\}$ and $\{z^{\omega}\}$ both tend to $x$, and then applying Lemma \ref{Separation} gives the equality.  But, this implies that the entire net tends to $\varphi$ and is independent of the net, giving the claim.  We then denote the limit by $\varphi_x$ and one can easily observe that $\varphi_x(0)=x$.  We then have the following Lemma, (see \cite{CLNZ}*{Lemma 4.2} and \cite{Sua}*{Lemma 6.3}).

\begin{lm}
\label{Maximal3}
Let $\left\{z^{\omega}\right\}$ be a net in $\D^n$ converging to $x\in M_{\mathcal{A}}$.  Then
\begin{itemize}
\item[(i)] $a\circ\varphi_x\in\mathcal{A}$ for every $a\in\mathcal{A}$.  In particular, $\varphi_x:\D^n\to M_{\mathcal{A}}$ is continuous;
\item[(ii)] $a\circ\varphi_{z^{\omega}}\to a\circ\varphi_x$ uniformly on compact sets of $\D^n$ for every $a\in\mathcal{A}$.
\end{itemize}
\end{lm}
\begin{proof}
If $a\in\mathcal{A}$, given $\epsilon>0$, there is a $\delta>0$ such that if $z,w\in\D^n$, then 
$\rho\left(z,w\right)<\delta$ implies that
$$
\abs{a(z)-a(w)}<\epsilon.
$$
Since $\rho\left(\varphi_{z^{\omega}}(z),\varphi_{z^{\omega}}(w)\right)=\rho\left(z,w\right)$ and
$$
\abs{a(\varphi_x(z))-a(\varphi_x(w))}=\lim_{\omega}\abs{a(\varphi_{z^{\omega}}(z))-a(\varphi_{z^{\omega}}(w))},
$$ 
we have that (i) holds.

Suppose by contradiction that (ii) fails.  Then there is a $a\in\mathcal{A}$, some $0<r<1$ and $\epsilon>0$ such that 
$$
\abs{a\circ\varphi_{z^{\omega}}(\xi^{\omega})-a\circ\varphi_{x}(\xi^{\omega})}>\epsilon
$$
for some points $\xi^{\omega}\in r\D^n$.  Passing to a subnet, if necessary, we can assume $\xi^{\omega}\to\xi\in  \overline{r\D^n}$.  But, this leads to a contradiction since we have
\begin{eqnarray*}
\epsilon & < & \abs{a\circ\varphi_{z^{\omega}}(\xi^{\omega})-a\circ\varphi_{x}(\xi^{\omega})}\\
 & \leq & \abs{a\circ\varphi_{z^{\omega}}(\xi^{\omega})-a\circ\varphi_{z^{\omega}}(\xi)}+\abs{a\circ\varphi_{z^{\omega}}(\xi)-a\circ\varphi_{x}(\xi)}+\abs{a\circ\varphi_{x}(\xi^{\omega})-a\circ\varphi_{x}(\xi)}.
\end{eqnarray*}
The first and third terms go to zero by the $\rho$ continuity of $a$ and $a\circ\varphi_x$.  The second term goes to zero by the pointwise convergence of $a\circ\varphi_{z^{\omega}}\to a\circ\varphi_x$.  This is the desired contradiction.
\end{proof}

\subsection{Maps from \texorpdfstring{$M_{\mathcal{A}}$}{the maximal ideal space} into \texorpdfstring{$\mathcal{L}\left(A^p, A^p\right)$}{the bounded operators on the weighted Bergman space}}

For $z\in\D^n$, define a map
$$
U^{p}_z f(w):= f(\varphi_z(w))\prod_{l=1}^n\frac{\left(1-\abs{z_l}^2\right)^{\frac{2}{p}}}{\left(1-w_l\overline{z}_l\right)^{\frac{4}{p}}}
$$
where the argument of $(1-w\overline{z})$ is used to define the root that appears above.  Then a standard change of variable argument and straightforward computation gives
$$
U^{p}_zU^{p}_z=I_{A^p}\quad\textnormal{ and }\quad \norm{U^{p}_z f}_{A^p}=\norm{f}_{A^p} \quad\forall f\in A^p.
$$
For a real number $r$, we set
\begin{equation}
\label{Jr}
J^r_z(w):=\prod_{l=1}^n \frac{\left(1-\abs{z_l}^2\right)^{r}}{\left(1-w_l\overline{z}_l\right)^{2r}}.
\end{equation}
Observe that
$$
U^{p}_z f(w)=f\left(\varphi_z(w)\right)J_z^{\frac{2}{p}}(w),
$$
and so 
$$
U^{p}_z= T_{J_z^{\frac{2}{p}-1}}U^{2}_z=U^{2}_z T_{J_z^{1-\frac{2}{p}}}.
$$
If $q$ is the conjugate exponent of $p$, we have that
$$
\left(U^{q}_z\right)^{*}=U^{2}_z T_{\overline{J_z}^{\frac{2}{q}-1}}
=T_{\overline{J_z}^{1-\frac{2}{q}}}U^{2}_z.
$$
And then using that $U^{2}_zU^{2}_z=I_{A^2}$ and straightforward computations one finds
$$
\left(U^{q}_z\right)^{*}U^{p}_z=T_{b_z}
\quad\textnormal{ and }\quad U^{p}_z\left(U^{q}_z\right)^{*}=T_{b_z}^{-1}
$$
where
\begin{equation}
\label{bofz}
b_z(w)=\prod_{l=1}^n\frac{(1-\overline{w}_lz_l)^{2\left(\frac{1}{q}-\frac{1}{p}\right)}}{(1-\overline{z}_lw_l)^{2\left(\frac{1}{q}-\frac{1}{p}\right)}}.
\end{equation}
Also observe at this point that when $p=q=2$ that $b_z(w)=1$.  This will be important later on when we consider the special case of $A^2$.

For $z\in\D^n$ and $S\in\mathcal{L}\left(A^p, A^p\right)$ we then define the map 
$$
S_z :=U^{p}_z S\left(U^{q}_z\right)^{*},
$$
which induces a map $\Psi_S:\D^n\to \mathcal{L}\left(A^p, A^p\right)$ given by $$\Psi_S(z)=S_z.$$  One should think of the map $S_z$ in the following way.  This is an operator on $A^p$ and so it first acts as ``translation'' in $\D^n$, then the action of $S$, then ``translation'' back.  We now show that it is possible to extend the map $\Psi_S$ continuously to a map from $M_\mathcal{A}$ to $\mathcal{L}\left(A^p, A^p\right)$ when endowed with both the weak and strong operator topologies.

First observe that $C(\overline{\D^n})\subset\mathcal{A}$ induces a natural projection $\pi:M_{\mathcal{A}}\to M_{C(\overline{\D^n})}$.  Note that $M_{C(\overline{\D^n})}=\overline{\D^n}$, and so the coordinates of $\pi(x)$ can be denoted by $\pi_l(x)$.  If $x\in M_{\mathcal{A}}$, let
\begin{equation}
\label{bofx}
b_x(w)=\prod_{l=1}^n\frac{(1-\overline{w}_l\pi_l(x))^{2\left(\frac{1}{q}-\frac{1}{p}\right)}}{(1-\overline{\pi_l(x)}w_l)^{2\left(\frac{1}{q}-\frac{1}{p}\right)}}.
\end{equation}
When $\{z^{\omega}\}$ is a net in $\D^n$ that tends to $x\in M_{\mathcal{A}}$,
then $z^{\omega}=\pi(z^{\omega})\to \pi(x)$ in the Euclidean metric (which follows from the continuity of $\pi:M_{\mathcal{A}}\to M_{C(\overline{\D^n})}$).
So, we have that $b_{z^{\omega}}\to b_x$ uniformly on compact sets of $\D^n$ and boundedly.  And, further,
$$
(U^{q}_z)^{*}U^{p}_z=T_{b_z}\to T_{b_x}
\ \textnormal{ and }\ (U^{p}_z)^{*}  U^{q}_z=T_{\overline{b_z}}\to T_{\overline{b_x}}
$$
with convergence in the strong operator topology in $\mathcal{L}\left(A^p,A^p\right)$ and
$\mathcal{L}(A^q,A^q)$ respectively.  If $a\in\mathcal{A}$ then Lemma \ref{Maximal3} implies that
$a\circ \varphi_{z^{\omega}}\to a\circ\varphi_x$ uniformly on compact sets of $\D^n$.  The above discussion implies
$$
T_{\left(a\circ\varphi_{z^{\omega}} \right) b_{z^{\omega}}}\to T_{\left(a\circ\varphi_x\right)b_x}
$$
in the strong operator topology associated with $\mathcal{L}\left(A^p,A^p\right)$.

Recall that $k_{z}^{(p)}(w)=\prod_{l=1}^n\frac{\left(1-\abs{z_l}^2\right)^{\frac{2}{q}}}{\left(1-\overline{z}_lw_l\right)^{2}}$ and that $\norm{k^{(p)}_z}_{A^p}\approx 1$.  Then note
\begin{eqnarray*}
\prod_{j=1}^n\left(1-\abs{\xi_j}^2\right)^{\frac{2}{p}}J_z^{\frac{2}{p}}(\xi) & = & \prod_{l=1}^n\left(1-\abs{\varphi_{z_l}(\xi_l)}^2\right)^{\frac{2}{p}}\prod_{l=1}^n\frac{\abs{1-\overline{z}_l\xi_l}^{\frac{4}{p}}}{(1-\xi_l\overline{z}_l)^{\frac{4}{p}}}\\
 & = & \prod_{l=1}^n\left(1-\abs{\varphi_{z_l}(\xi_l)}^2\right)^{\frac{2}{p}}\lambda_{(p)}(\xi,z).
\end{eqnarray*}
Here the constant $\lambda_{(p)}(\xi,z)$ is unimodular.  To see this, if $f\in A^{p}$, then
\begin{eqnarray}
\ip{f}{\left(U^{p}_z\right)^* k_\xi^{(q)}}_{A^2} & = &
\ip{U^{p}_zf}{ k_\xi^{(q)}}_{A^2}=
\ip{J_z^{\frac{2}{p}} \left(f\circ\varphi_z\right)}{ k_\xi^{(q)}}_{A^2}\notag\\
& = & f(\varphi_z(\xi))\prod_{l=1}^n\left(1-\abs{\xi_l}^2\right)^{\frac{2}{p}}J_z^{\frac{2}{p}}(\xi)\notag\\
& = & f(\varphi_z(\xi))\prod_{l=1}^n\left(1-\abs{\varphi_{z_l}(\xi_l)}^2\right)^{\frac{2}{p}}\lambda_{(p)}(\xi,z)\notag\\
& = & \ip{f}{\overline{\lambda_{(p)}(\xi,z)} k_{\varphi_z(\xi)}^{(q)}}_{A^2}\notag,
\end{eqnarray}
which gives
\begin{equation}
\label{Comp}
\left(U^{p}_z\right)^* k_\xi^{(q)}=\lambda_{(p)}(\xi,z)k_{\varphi_z(\xi)}^{(q)}.
\end{equation}
\begin{lm}
\label{UniformCon}
Fix $\xi\in\D^n$.  Then the map $z\mapsto \left(U^{p}_z\right)^* k_\xi^{(q)}$ is uniformly continuous from $\left(\D^n,\rho\right)$ into $\left(A^{q},\norm{\,\cdot\,}_{A^q}\right)$.
\end{lm}
\begin{proof}
By \eqref{Comp} we only need to prove that the maps $z\mapsto \lambda_{(p)}(\xi,z)$ and
$z\mapsto k_{\varphi_z(\xi)}^{(q)}$ are uniformly continuous from $\left(\D^n,\rho\right)$ into
$\left(\C,\abs{\,\cdot\,}\right)$ and $\left(A^{q},\norm{\,\cdot\,}_{A^q}\right)$, respectively.  It is obvious that the map $z\mapsto \lambda_{(p)}(z,\xi)$  has the desired property.  So we just focus on the continuity in the second map.

By Lemma \ref{Invariance} we have that $z\mapsto\varphi_z(\xi)$ is uniformly continuous from $\left(\D^n,\rho\right)$ into itself.  So, it suffices to prove the uniform continuity of the map $w\mapsto k^{(q)}_w$.  Namely, for any $\epsilon>0$, there is a $\delta>0$ such that if $\rho\left(w,0\right)=\max_{1\leq l\leq n}\abs{w_l}<\delta$ then
$$
\sup_{z\in\D^n}\norm{k_{z}^{(q)}-k_{\varphi_z(w)}^{(q)}}_{A^q}<\epsilon.
$$
We use the duality between $A^{p}$ and $A^{q}$ to have that
\begin{equation*}
\sup_{z\in\D^n}\norm{k_{z}^{(q)}-k_{\varphi_z(w)}^{(q)}}_{A^q}\approx\sup_{z\in\D^n}\sup_{f\in A^{p}:\norm{f}_{A^p}\leq 1}\abs{\prod_{l=1}^n\left(1-\abs{z_l}^2\right)^{\frac{2}{p}}f(z)-\prod_{l=1}^n\left(1-\abs{\varphi_{z_l}(w_l)}^2\right)^{\frac{2}{p}}f(\varphi_z(w))}.
\end{equation*}
Consider the term inside the supremums, and observe that it can be dominated by
$$
\abs{\prod_{l=1}^n\left(1-\abs{z_l}^2\right)^{\frac{2}{p}}\left(f(z)-f(\varphi_z(w))\right)} +\abs{f(\varphi_z(w))}\abs{\prod_{l=1}^n\left(1-\abs{z_l}^2\right)^{\frac{2}{p}}-\prod_{l=1}^n\left(1-\abs{\varphi_{z_l}(w_l)}^2\right)^{\frac{2}{p}}}
$$
through adding and subtracting a common term.  For the first term, set 
$$
g_z(w)=\prod_{l=1}^n \left(1-\abs{z_l}\right)^{\frac{2}{p}}f(\varphi_z(w))=\prod_{l=1}^n\left(1-\overline{z}_lw_l\right)^{\frac{4}{p}}U_z^p f(w).
$$
Since $U_z^p$ is an isometry, we have that $\norm{g_z}_{A^p}\lesssim\norm{f}_{A^p}$.  Now observe that we have
\begin{eqnarray*}
\prod_{l=1}^n\left(1-\abs{z_l}^2\right)^{\frac{2}{p}}\abs{\left(f(z)-f(\varphi_z(w))\right)}  & = & \abs{g_z(0)-g_z(w)}\\
& \lesssim &  \norm{f}_{A^p}\norm{K_0-K_w}_{L^\infty}.
\end{eqnarray*}
While for the second term, it is easy to see using the reproducing property of the kernel $k_z^{(q)}$ that this is dominated by a constant times
$$
\norm{f}_{A^p}\abs{1-\prod_{l=1}^n\frac{(1-\abs{z_l}^\frac{2}{p})}{\left(1-\abs{\varphi_{z_l}(w_l)}^2\right)^{\frac{2}{p}}}}=\norm{f}_{A^p}\abs{1-\prod_{l=1}^n\frac{\abs{1-\overline{w}_l z_l}^{\frac{4}{p}}}{\left(1-\abs{w_l}^2\right)^{\frac{2}{p}}}}.
$$
Combining these estimates, we find that
$$
\sup_{z\in\D^n}\norm{k_{z}^{(q)}-k_{\varphi_z(w)}^{(q)}}_{A^q}\lesssim \norm{K_w-K_{0}}_{L^\infty}+\sup_{z\in\D^n}\abs{1-\prod_{l=1}^n\frac{\abs{1-\overline{w}_l z_l}^{\frac{4}{p}}}{\left(1-\abs{w_l}^2\right)^{\frac{2}{p}}}}.
$$
However, it is now easy to see that by taking $\max_{1\leq l\leq n}\abs{w_l}$ small enough both of the remaining terms can be made sufficiently small, independent of $z$.  This gives the desired continuity. 
\end{proof}

\begin{prop}
\label{WOTCon}
Let $S\in \mathcal{L}\left(A^p, A^p\right)$.  Then the map $\Psi_S: \D^n \to \left(\mathcal{L}\left(A^p, A^p\right), WOT\right)$ extends continuously to $M_{\mathcal{A}}$.
\end{prop}
\begin{proof}
Bounded sets in $\mathcal{L}\left(A^p, A^p\right)$ are metrizable and have compact closure in the weak operator topology.  Since $\Psi_S\left(\D^n\right)$ is bounded, by Lemma \ref{Extend} we only need to show that $\Psi_S$ is uniformly continuous from $\left(\D^n,\rho\right)$ into $\left(\mathcal{L}\left(A^p, A^p\right), WOT\right)$, where $WOT$ is the weak operator topology.  Namely, we need to demonstrate that for $f\in A^p$ and $g\in A^{q}$ the function $z\mapsto \ip{S_z f}{g}_{A^2}$ is uniformly continuous from $\left(\D^n,\rho\right)$ into $\left(\C, \abs{\,\cdot\,}\right)$.

For $z_1,z_2\in\D^n$ we have
\begin{eqnarray*}
S_{z_1}-S_{z_2} & = & U_{z_1}^{p}S\left(U_{z_1}^{q}\right)^*-U_{z_2}^{p}S\left(U_{z_2}^{q}\right)^*\\
 & = & U_{z_1}^{p}S[\left(U_{z_1}^{q}\right)^*-\left(U_{z_2}^{q}\right)^*]+\left(U_{z_1}^{p}-U_{z_2}^{p}\right)S\left(U_{z_2}^{q}\right)^*\\
 & = & I+II.
\end{eqnarray*}
Note that the terms $I$ and $II$ have a certain symmetry, and so it is enough to deal with either term since the argument will work in the other case as well.  Observe that
\begin{eqnarray*}
\abs{\ip{If}{g}_{A^2}} & \leq & \norm{U_{z_1}^{p}S}_{\mathcal{L}\left(A^p, A^p\right)}\norm{\left[\left(U_{z_1}^{q}\right)^*-\left(U_{z_2}^{q}\right)^*\right]f}_{A^p}\norm{g}_{A^q}\\
\abs{\ip{IIf}{g}_{A^2}} & \leq & \norm{\left(U_{z_1}^{q}\right)^{*}S}_{\mathcal{L}\left(A^p, A^p\right)}\norm{\left[\left(U_{z_1}^{p}\right)^*-\left(U_{z_2}^{p}\right)^*\right]g}_{A^q}\norm{f}_{A^p}.
\end{eqnarray*}
Since $S$ is bounded and since $\norm{U_{z}^{p}}_{\mathcal{L}\left(A^p, A^p\right)}\lesssim 1$ for all $z$, we just need to show that the expression $\norm{[(U_{z_1}^{p})^*-(U_{z_2}^{p})^*]g}_{A^q}$ can be made small.  It suffices to do this on a dense set of functions, and in particular we can take the linear span of $\left\{k_{w}^{(p)}:w\in\D^n\right\}$.  However, then we can apply Lemma \ref{UniformCon} to conclude the result.
\end{proof}

This proposition allow us to define $S_x$ for all $x\in M_{\mathcal{A}}$. Namely, we let $S_x:=\Psi_S(x)$. In particular, if $\{z^{\omega}\}$ is a net in $\D^n$ that tends to $x\in M_{\mathcal{A}}$ then $S_{z^{\omega}}\to S_x$ in the weak operator topology. In Proposition~\ref{SOTCon} below we will show that if $S\in\mathcal{T}_p$ then we also have that $S_{z^{\omega}}\to S_x$ in the strong operator topology. 

\begin{lm}
\label{Inverse}
If $\{z^{\omega}\}$ is a net in $\D^n$ converging to $x\in M_{\mathcal{A}}$ then $T_{b_x}$ is invertible and $T_{b_{z^{\omega}}}^{-1}\to T_{b_x}^{-1}$ in the strong operator topology.
\end{lm}

\begin{proof}
By Proposition \ref{WOTCon} applied to the operator $S=I_{A^p}$ we have that $U_{z^{\omega}}^{p}(U_{z^{\omega}}^{q})^*=T_{b_{z^{\omega}}}^{-1}$ has a weak operator limit in $\mathcal{L}(A^p, A^p)$, denote this by $Q$.  The Uniform Boundedness Principle provides a constant $C$ such that $\norm{T_{b_{z^{\omega}}}^{-1}}_{\mathcal{L}\left(A^p, A^p\right)}\leq C$ for all $\omega$.  Then given $f\in A^p$ and $g\in A^q$, since we know
$$
\norm{(T_{\overline{b}_{z^{\omega}}}-T_{\overline{b}_{x}})g}_{A^q}\to 0
$$
we have
\begin{eqnarray*}
\ip{T_{b_x}Qf}{g}_{A^2}=\ip{Qf}{T_{\overline{b}_x}g}_{A^2} & = & \lim_{\omega}\ip{T^{-1}_{b_{z^{\omega}}}f}{T_{\overline{b}_x}g}_{A^2}\\
& = & \lim_{\omega}\left(\ip{T^{-1}_{b_{z^{\omega}}}f}{(T_{\overline{b}_x}-T_{\overline{b}_{z^{\omega}}})g}_{A^2}+\ip{T^{-1}_{b_{z^{\omega}}}f}{T_{\overline{b}_{z^{\omega}}}g}_{A^2}\right)\\
& = & \ip{f}{g}_{A^2}+\lim_{\omega}\ip{T^{-1}_{b_{z^{\omega}}}f}{(T_{\overline{b}_x}-T_{\overline{b}_{z^{\omega}}})g}_{A^2}=\ip{f}{g}_{A^2}.
\end{eqnarray*}
%Here the equality follows since,
%$$
%\abs{\ip{T^{-1}_{b_{z^{\omega}}}f}{(T_{\overline{b}_x}-T_{\overline{b}_{z^{\omega}}})g}_{A^2}}\leq C\norm{f}_{A^p}\norm{\left(T_{\overline{b}_{z^{\omega}}}-T_{\overline{b}_{x}}\right)g}_{A^q}\to 0.
%$$
This gives $T_{b_x}Q=I_{A^p}$.  Since taking adjoints is a continuous operation in the $WOT$, $T_{\overline{b}_{z^{\omega}}}^{-1}\to Q^*$, and interchanging the roles of $p$ and $q$, we have that $T_{\overline{b}_x}Q^*=I$, which implies that $QT_{b_x}=I_{A^p}$.  This gives that $Q=T_{b_x}^{-1}$ and $T_{b_{z^{\omega}}}^{-1}\to T_{b_x}^{-1}$ in the weak operator topology.  Finally, note that
$$
T_{b_{z^{\omega}}}^{-1}-T_{b_x}^{-1}=T_{b_{z^{\omega}}}^{-1}(T_{b_x}-T_{b_{z^{\omega}}}) T_{b_x}^{-1}
$$
and since $\norm{T_{b_{z^{\omega}}}^{-1}}_{\mathcal{L}\left(A^p, A^p\right)}\leq C$ and $T_{b_{z^{\omega}}}-T_{b_x}\to 0$ in the strong operator topology, and so $T_{b_{z^{\omega}}}^{-1}\to T_{b_x}^{-1}$ in the strong operator topology as claimed.
\end{proof}

\begin{prop}
\label{SOTCon}
If $S\in\mathcal{T}_{p}$ and $\{z^{\omega}\}$ is a net in $\D^n$ that tends to $x\in M_{\mathcal{A}}$ then $S_{z^{\omega}}\to S_x$ in the strong operator topology.  In particular, $\Psi_S:\D^n\to \left(\mathcal{L}\left(A^p,A^p\right), SOT\right)$ extends continuously to $M_{\mathcal{A}}$.
\end{prop}
\begin{proof}
First observe that if $A, B\in \mathcal{L}\left(A^p, A^p\right)$ then
\begin{eqnarray*}
\left(AB\right)_z= U^{p}_zAB\left(U^{q}_z\right)^{*} & = & U^{p}_zA\left(U^{q}_z\right)^{*}\left(U^{q}_z\right)^{*} U^{p}_zU^{p}_zB\left(U^{q}_z\right)^{*}\\
 & = & A_z T_{b_z}B_z.
\end{eqnarray*}
In general this applies to longer products of operators.

For $S\in \mathcal{T}_{p}$ and $\epsilon>0$, by Theorem \ref{Sua2Thm} there is a finite sum of finite products of Toeplitz operators with symbols in $\mathcal{A}$, call it $R$, such that $\norm{R-S}_{\mathcal{L}\left(A^p,A^p\right)}<\epsilon$.  This gives, $\norm{R_z-S_z}_{\mathcal{L}\left(A^p,A^p\right)}\lesssim \epsilon$, and upon passing to the $WOT$ limit, we have $\norm{R_x-S_x}_{\mathcal{L}\left(A^p,A^p\right)}\lesssim \epsilon$ for all $x\in M_{\mathcal{A}}$.  These observations imply it suffices to prove the Lemma for $R$ alone.  By linearity, it suffices to show it in the special case of $R=\prod_{j=1}^m T_{a_j}$ where $a_j\in\mathcal{A}$.  A simple computation shows that
$$
U^{2}_z T_a U^{2}_z= T_{a\circ\varphi_z}
$$
and more generally,
\begin{eqnarray*}
\left(T_a\right)_z & = & U^{p}_z\left(U^{q}_z\right)^*\left(U^{q}_z\right)^*T_a U^{p}_zU^{p}_z\left(U^{q}_z\right)^*\\
& = & U^{p}_z\left(U^{q}_z\right)^*\left(T_{\overline{J}_z^{1-\frac{2}{q}}}U^{2}_zT_aU^{2}_zT_{J_z^{1-\frac{2}{p}}}\right)U^{p}_z\left(U^{q}_z\right)^*\\
& = & T_{b_{z}}^{-1}T_{\left(a\circ\varphi_z\right) b_z}T_{b_{z}}^{-1}.
\end{eqnarray*}
We now combine this computation with the observation at the beginning of the lemma to see that
\begin{eqnarray*}
\left(\prod_{j=1}^{m} T_{a_j}\right)_z & = & \left(T_{a_1}\right)_z T_{b_z}\cdots T_{b_z}\left(T_{a_m}\right)_z\\
 & = & T^{-1}_{b_z} T_{\left(a_1\circ\varphi_z\right) b_z}T^{-1}_{b_z}T_{\left(a_2\circ\varphi_z\right) b_z}T^{-1}_{b_z}\cdots
 T^{-1}_{b_z}T_{\left(a_m\circ\varphi_z\right) b_z}T^{-1}_{b_z}.
\end{eqnarray*}
But, since the product of $SOT$ nets is $SOT$ convergent, Lemma \ref{Inverse} and the fact that
$T_{\left(a\circ\varphi_{z^{\omega}}\right)b_{z^{\omega}}}\to T_{(a\circ\varphi_x) b_x}$ in the $SOT$ gives that
$$
\left(\prod_{j=1}^{m} T_{a_j}\right)_{z^{\omega}}\to T^{-1}_{b_x}
T_{\left(a_1\circ\varphi_x\right) b_x}T^{-1}_{b_x}T_{\left(a_2\circ\varphi_x\right) b_x}T^{-1}_{b_x}\cdots T^{-1}_{b_x}
T_{\left(a_m\circ\varphi_x\right) b_x}T^{-1}_{b_x}.
$$
But, this is exactly the statement that $R_{z^{\omega}}\to R_x$ for the operator $\prod_{j=1}^{m} T_{a_j}$, and proves the continuous extension we desired.
\end{proof}

We now collect a very simple Lemma based on these computations that gives information about the Berezin transform vanishing in terms of $S_x$.

\begin{lm}
\label{BerezinVanish}
Let $S\in\mathcal{L}\left(A^p,A^p\right)$.  Then $B(S)(z)\to 0$ as $z\to \partial\D^n$ if and only if $S_x=0$ for all $x\in M_{\mathcal{A}}\setminus\D^n$.
\end{lm}
\begin{proof}
First, some general computations.  Let $z,\xi\in\D^n$, then we have that
\begin{eqnarray*}
B(S_z)(\xi) & = & \ip{S\left(U^{q}_z\right)^*k_\xi^{(p)}}{\left(U^{p}_z\right)^*k_\xi^{(q)}}_{A^2}\\
& = & \lambda_{(q)}(\xi,z)\overline{\lambda_{(p)}(\xi,z)}\ip{Sk_{\varphi_z(\xi)}^{(p)}}{k_{\varphi_z(\xi)}^{(q)}}_{A^2}.
\end{eqnarray*}
Here the last equality follow by \eqref{Comp}.  Thus, we have that $\abs{B(S_z)(\xi)}=\abs{B(S)(\varphi_z(\xi))}$ since $\lambda_{(p)}$ and $\lambda_{(q)}$ are unimodular numbers.  For $x\in M_{\mathcal{A}}\setminus\D^n$, there is a net $\{z^{\omega}\}$ tending to $x$ and for $\xi\in\D^n$ fixed, the continuity of $\Psi_S$ in the $WOT$, Lemma \ref{WOTCon}, gives that $B(S_{z^{\omega}})(\xi)\to B(S_x)(\xi)$ and by the computations above $\abs{B(S)(\varphi_{z^{\omega}}(\xi))}\to \abs{B(S_x)(\xi)}$.

Now, suppose $B(S)(z)$ vanishes as $z\to \partial\D^n$.  Since $x\in M_{\mathcal{A}}\setminus\D^n$ and $z^{\omega}\to x$ we have that $z^{\omega}\to\partial\D^n$ and similarly $\varphi_{z^{\omega}}(\xi)\to \partial\D^n$ too.  Since $B(S)(z)$ vanishes as we approach the boundary, then we have that $B(S_x)(\xi)=0$, and since $\xi\in\D^n$ was arbitrary and the Berezin transform is one-to-one we have that $S_x=0$

Conversely, suppose the Berezin transform does not vanish as we approach the boundary.  Then we have a sequence $\{z^{k}\}$ in $\D^n$ such that $z^{k}\to \partial\D^n$ and $\abs{B(S)(z^k)}\geq\delta>0$.  Since $M_{\mathcal{A}}$ is compact, extract a subnet $\{z^{\omega}\}$ of $\{z^{k}\}$ converging to $x\in M_{\mathcal{A}}\setminus\D^n$.  The computations above imply $\abs{B(S_x)(0)}\geq\delta>0$, which in turn gives that $S_x\neq 0$.
\end{proof}

\section{Bergman--Carleson Measures and Approximation}
\label{CMA}
Given a complex-valued measure $\mu$ whose total variation is Bergman--Carleson,
our next goal is to construct a sequence of functions $B_k(\mu) \in \mathcal{A}$ such that $T_{B_k(\mu)} \rightarrow T_\mu$ in the norm of $\mathcal{L}\left(A^p, A^p\right)$ for $1<p<\infty$.
As a consequence, we have the following Theorem.
\begin{thm}
\label{Sua2Thm}
The Toeplitz algebra $\mathcal{T}_{p}$ equals the closed algebra generated by $\left\{T_a:a\in\mathcal{A}\right\}$.
\end{thm}

To prove this Theorem requires some additional notation.  Let $\mu$ be a complex-valued, Borel, regular measure on $\D^n$ of finite total variation.  Following Nam and Zheng, \cite{NZ}, we define the $k$-Berezin transform of $\mu$ to be the function
\begin{equation}
\label{Berezin-k-poly}
B_k(\mu)(z):=\left(k+1\right)^n\int_{\D^n} \prod_{l=1}^n\frac{\left(1-\abs{z_l}^{2}\right)^2}{\abs{1-\overline{w}_l z_l}^4}\prod_{l=1}^n\left(1-\abs{\varphi_{z_l}(w_l)}^2\right)^k\,d\mu(w).
\end{equation}
It is immediate that we have an equivalent definition given by
\begin{equation}
\label{Berezin-k-poly2}
B_k(\mu)(z)=\left(k+1\right)^n\int_{\D^n} \prod_{l=1}^n\frac{\left(1-\abs{z_l}^{2}\right)^{2+k}\left(1-\abs{w_l}^2\right)^k}{\abs{1-\overline{w}_l z_l}^{2(2+k)}}\,d\mu(w).
\end{equation}
\begin{rem}
The astute reader will not see the definition \eqref{Berezin-k-poly} or \eqref{Berezin-k-poly2} in the paper \cite{NZ}.  However, the definition we are using is simply the definition of Nam and Zheng in \cite{NZ}*{Proposition 2.2} applied to the operator $T_{\mu}$.  Indeed, to see this, for a general operator one has
$$
B_k\left(S\right)(z)=\left(k+1\right)^n\prod_{l=1}^n\left(1-\abs{z_l}^2\right)^{2+k}\int_{\D^n}S\left(\prod_{l=1}^n\frac{\left(1-\cdot\,\overline{w}_l\right)^k}{\left(1-\cdot\,\overline{z}_l\right)^{k+2}}\right)(w) \prod_{l=1}^n\frac{1}{\left(1-\overline{w}_lz_l\right)^{k+2}} \,dv(w).
$$ 
Now, letting $S=T_\mu$ and a computation yields 
$$
T_{\mu}\left(\prod_{l=1}^n\frac{\left(1-\cdot\,\overline{w}_l\right)^k}{\left(1-\cdot\,\overline{z}_l\right)^{k+2}}\right)(w)=\int_{\D^n}\prod_{l=1}^n\frac{\left(1-\overline{w}_l\xi_l\right)^{k}}{\left(1-\overline{z}_l\xi_l\right)^{k+2}\left(1-\overline{\xi}_lw_l\right)^{2}} \,d\mu(\xi).
$$
Using this expression, substitution into the above formula of Nam and Zheng, an application of Fubini and using the reproducing property of the kernel $K_z$  then yields $B_k(T_\mu)=B_k(\mu)$ as given in \eqref{Berezin-k-poly} and \eqref{Berezin-k-poly2}. 
\end{rem}

For $z\in\D^n$, we consider a related measure defined by $\mu_z(w):=\left(\prod_{l=1}^n\frac{(1-\abs{z_l}^2)}{\abs{1-\overline{z}_lw_l}^2}\right)^{-2}\mu\circ\varphi_z(w)$.  A straightforward change of variables argument gives
\begin{equation}
\label{ChangeofVariables}
\int_{\D^n} f\left(\varphi_z(\xi)\right) \prod_{l=1}^n \frac{(1-\abs{z_l}^2)^2}{\abs{1-\overline{z}_l\xi_l}^4}\,d\mu(\xi)=\int_{\D^n} f\left(\xi\right) \,d\mu_z(\xi).
\end{equation}
Which in turn gives a relationship between the $k$-Berezin transform and the automorphism $\varphi_z$,
\begin{equation}
\label{MagicalEquality}
B_k(\mu)\left(\varphi_z(w)\right)=B_k(\mu_z)(w).
\end{equation}
Indeed, from the definition of $\mu_z$, we have
\begin{eqnarray*}
B_k(\mu_z)(w) & = & \left(k+1\right)^n\int_{\D^n} \prod_{l=1}^n\frac{\left(1-\abs{w_l}^{2}\right)^2}{\abs{1-\overline{w}_l \xi_l}^4}\prod_{l=1}^n\left(1-\abs{\varphi_{w_l}(\xi_l)}^2\right)^k\,d\mu_z(\xi)\\
 & = & \left(k+1\right)^n\int_{\D^n} \prod_{l=1}^n\frac{\left(1-\abs{w_l}^{2}\right)^2}{\abs{1-\overline{w}_l \varphi_{z_l}(\xi_l)}^4}\prod_{l=1}^n \frac{\left(1-\abs{\varphi_{w_l}\left(\varphi_{z_l}(\xi_l)\right)}^2\right)^k\left(1-\abs{z_l}^2\right)^2}{\abs{1-\overline{z}_l\xi_l}^4}\,d\mu(\xi)\\
% & = & \left(k+1\right)^n\int_{\D^n} \prod_{l=1}^n\frac{(1-\abs{w_l}^{2})^2}{\abs{1-\overline{w}_l \varphi_{z_l}(\xi_l)}^4}\prod_{l=1}^n\left(1-\abs{\varphi_{w_l}\left(\varphi_{z_l}(\xi_l)\right)}^2\right)^k\prod_{l=1}^n \frac{(1-\abs{z_l}^2)^2}{\abs{1-\overline{z}_l\xi_l}^4}\,d\mu(\xi)\\
%& = & \left(k+1\right)^n\int_{\D^n} \prod_{l=1}^n\frac{(1-\abs{w_l}^{2})^2}{\abs{1-\overline{w}_l \varphi_{z_l}(\xi_l)}^4}\prod_{l=1}^n\left(1-\abs{\xi_l}^2\right)^k\prod_{l=1}^n \frac{(1-\abs{z_l}^2)^2}{\abs{1-\overline{z}_l\xi_l}^4}\,d\mu(\xi)\\ 
& = & \left(k+1\right)^n\int_{\D^n} \prod_{l=1}^n\frac{\left(1-\abs{\varphi_{z_l}(w_l)}^{2}\right)^2}{\abs{1-\overline{\xi}_l \varphi_{z_l}(w_l)}^4}\prod_{l=1}^n\left(1-\abs{\varphi_{\varphi_{z_l}(w_l)}(\xi_l)}^2\right)^k\,d\mu(\xi)\\
& = & B_k(\mu)(\varphi_z(w)).
\end{eqnarray*}
In the display above, the first equality is just the definition of the $k$-Berezin transform, the second is just \eqref{ChangeofVariables}, and the third follows from computation.  

It is immediate to see that when $k=0$, that $B_0(\mu)(z)=B(T_\mu)(z)$, and so if $\mu$ is a Bergman--Carleson measure then we have
$$
\norm{\mu_z}_{\textnormal{RKM}}=\norm{B_0(\mu_z)}_{L^\infty}=\norm{B_0(\mu)}_{L^\infty}=\norm{B(T_\mu)}_{L^\infty}=\norm{\mu}_{\textnormal{RKM}}.
$$
A similar result holds when $\abs{\mu}$ is a Bergman--Carleson measure.  

\begin{lm}
\label{Lemma_Change}
Let $0<\eta<1$ and let $\mu$ be a complex-valued measure such that its total variation $\abs{\mu}$ is a Bergman--Carleson measure.  If $\frac{1}{p_1}+\frac{1}{q_1}=1$, where $q_1>1$ is close enough to $1$ so that $q_1\eta<1$ and $q_1\left(2-\eta\right)<2$, then there is a constant depending on $p_1$ and $n$ such that
\begin{eqnarray*}
\int_{\D^n}\abs{T_\mu K_z(w)}\prod_{l=1}^n\frac{1}{(1-\abs{w_l}^2)^{\eta}}\,dv(w)\leq C\left(p_1,n\right)\norm{T_{\mu_z} 1}_{L^{p_1}}\prod_{l=1}^n\frac{1}{\left(1-\abs{z_l}^2\right)^{\eta}}.
\end{eqnarray*}
\end{lm}
\begin{proof}
First, observe
$$
\prod_{l=1}^n\left(1-\abs{z_l}^2\right)T_\mu K_z= T_\mu J_z^1.
$$
Recall that $J_z^1$ was defined in \eqref{Jr}.  Next, note 
\begin{eqnarray*}
T_{\mu_z}1(w) & = & \int_{\D^n} \prod_{l=1}^n\frac{1}{\left(1-\overline{\xi}_lw_l\right)^2}\,d\mu_z(\xi)\\
 & = & \int_{\D^n} \prod_{l=1}^n\frac{1}{\left(1-\overline{\varphi_{z_l}(\xi_l)}w_l\right)^2}\prod_{l=1}^n\frac{\left(1-\abs{z_l}^2\right)^2}{\abs{1-\overline{z}_l\xi_l}^4}\,d\mu(\xi).
\end{eqnarray*}
On the other hand, we have
\begin{eqnarray*}
J_z^1(w)\cdot T_\mu J_z^1\left(\varphi_z(w)\right) & = & J_z^1\left(w\right)\int_{\D^n}\prod_{l=1}^n\frac{1}{\left(1-\overline{\xi}_l\varphi_{z_l}(w_l)\right)^2}\prod_{l=1}^n\frac{\left(1-\abs{z_1}^2\right)}{\left(1-\overline{z}_l\xi_l\right)^2}\, d\mu(\xi)\\
& = & J_z^1(w) \prod_{l=1}^n\frac{\left(1-\overline{z}_lw_l\right)^2}{\left(1-\abs{z_l}^2\right)}\int_{\D^n} \prod_{l=1}^n\frac{\left(1-\abs{z_l}^2\right)^2}{\left(1-\overline{\varphi_{z_l}(\xi_l)}w_l\right)^2\abs{1-\overline{z}_l\xi_l}^4}\,d\mu(\xi)\\
& = & T_{\mu_z}1(w).
\end{eqnarray*}
Using this equality we have that 
$$
\prod_{l=1}^n\left(1-\abs{z_l}^2\right)T_\mu K_z(w)=T_\mu J_z^1(w)=T_{\mu_z}1(\varphi_z(w))\left(J_z^1(\varphi_z(w))\right)^{-1}=T_{\mu_z}1\left(\varphi_z(w)\right)J_z^1(w).
$$
We now use this computation to prove the lemma, 
\begin{eqnarray*}
 \int_{\D^n}\frac{\abs{T_\mu K_z(w)}}{\prod_{l=1}^n\left(1-\abs{w_l}^2\right)^\eta}\,dv(w) & = & \prod_{l=1}^n\frac{1}{\left(1-\abs{z_l}^2\right)}\int_{\D^n} \frac{\abs{T_{\mu_z}1(\varphi_z(w))}\abs{J_z^1(w)}}{\prod_{l=1}^n\left(1-\abs{w_l}^2\right)^{\eta}}\,dv(w) \\
 & = & \int_{\D^n} \frac{\abs{T_{\mu_z}1(w)}\abs{J_z^1(\varphi_z(w))}}{\prod_{l=1}^n\left(1-\abs{\varphi_{z_l}(w_l)}^2\right)^{\eta}}\prod_{l=1}^n \frac{\left(1-\abs{z_l}^2\right)}{\abs{1-\overline{z}_lw_l}^4}\,dv(w)\\
& =  & \prod_{l=1}^n\frac{1}{\left(1-\abs{z_l}^2\right)^{\eta}}\int_{\D^n} \abs{T_{\mu_z}1(w)} \prod_{l=1}^n \frac{\left(1-\abs{w_l}^2\right)^{-\eta}}{\abs{1-\overline{z}_lw_l}^{2-2\eta}}\,dv(w)\\
& \leq & \frac{\norm{T_{\mu_z}1}_{p_1}}{\prod_{l=1}^n\left(1-\abs{z_l}^2\right)^{\eta}}\left(\int_{\D^n} \prod_{l=1}^n \frac{\left(1-\abs{w_l}^2\right)^{-q_1\eta}}{\abs{1-\overline{z}_lw_l}^{q_1\left(2-2\eta\right)}}\,dv(w)\right)^{\frac{1}{q_1}}\\
& =  & \frac{\norm{T_{\mu_z}1}_{p_1}}{\prod_{l=1}^n\left(1-\abs{z_l}^2\right)^{\eta}} \prod_{l=1}^n\left(\int_{\D} \frac{\left(1-\abs{w_l}^2\right)^{-q_1\eta}}{\abs{1-\overline{z}_lw_l}^{q_1\left(2-2\eta\right)}}\,dv(w_l)\right)^{\frac{1}{q_1}}.
\end{eqnarray*}
Now apply Lemma \ref{Growth} using the conditions on $q_1$ to see that 
$$
\prod_{l=1}^n\left(\int_{\D} \frac{\left(1-\abs{w_l}^2\right)^{-q_1\eta}}{\abs{1-\overline{z}_lw_l}^{q_1(2-2\eta)}}\,dv(w)\right)^{\frac{1}{q_1}}\lesssim 1,
$$
and so the lemma follows.
\end{proof}

\begin{lm}
\label{Lemma_Change_Est}
Let $1<p<\infty$ and $\mu$ be a complex-valued measure such that its total variation $\abs{\mu}$ is a Bergman--Carleson measure.  If $\frac{1}{p_1}+\frac{1}{q_1}=1$, where $q_1$ satisfies the condition of Lemma \ref{Lemma_Change} for $\eta=\frac{1}{p}$ and $\eta=\frac{p-1}{p}$, then
\begin{eqnarray*}
\norm{T_\mu}_{\mathcal{L}\left(A^p,A^p\right)}\lesssim \left(\,\sup_{z\in\D^n}\norm{T_{\mu_z} 1}_{p_1}\right)^{\frac{1}{p}}\left(\,\sup_{z\in\D^n}\norm{T_{\mu_z}^* 1}_{p_1}\right)^{\frac{1}{q}}.
\end{eqnarray*}
\end{lm}
\begin{proof}
Let $f\in A^p$ and $w\in\D^n$.  Observe that $T_\mu K_\lambda(w)=\overline{T_\mu^* K_w(\lambda)}$, and so
$$
T_\mu f(w)=\ip{T_\mu f}{K_w}_{A^2}=\ip{f}{T_{\mu}^* K_w}_{A^2}=\int_{\D^n}f(\lambda)\overline{T_\mu^* K_w(\lambda)}\,dv(\lambda)=\int_{\D^n}f(\lambda) T_\mu K_\lambda(w)\,dv(\lambda).
$$
Set $\Phi\left(\lambda,w\right)=\abs{T_\mu K_\lambda(w)}=\abs{T_\mu^* K_w(\lambda)}$ and $h(\lambda)=\prod_{l=1}^n(1-\abs{\lambda_l}^2)^{-\frac{1}{pq}}$, where $\frac{1}{p}+\frac{1}{q}=1$.  Now apply Lemma \ref{Lemma_Change} with $\eta=\frac{1}{q}=\frac{p-1}{p}$ to see that
\begin{eqnarray*}
\int_{\D^n}\Phi\left(\lambda,w\right) h(w)^p\,dv(w) & = & \int_{\D^n}\abs{T_\mu K_\lambda(w)}\prod_{l=1}^n\frac{1}{\left(1-\abs{w_l}^2\right)^{\frac{1}{q}}}\,dv(w)\\
& \lesssim & \norm{T_{\mu_z} 1}_{L^{p_1}}\prod_{l=1}^n\frac{1}{\left(1-\abs{\lambda_l}^2\right)^{\frac{1}{q}}}\\
& \lesssim & \left(\,\sup_{z\in\D^n} \norm{T_{\mu_z} 1}_{L^{p_1}}\right) h(\lambda)^{p}.
\end{eqnarray*}
Similarly, apply Lemma \ref{Lemma_Change} with $\eta=\frac{1}{p}$ to see that
\begin{eqnarray*}
\int_{\D^n}\Phi\left(\lambda,w\right) h(\lambda)^q\,dv(\lambda) & = & \int_{\D^n}\abs{T_\mu^* K_w(\lambda)}\prod_{l=1}^n\frac{1}{\left(1-\abs{\lambda_l}^2\right)^{\frac{1}{p}}}\,dv(\lambda)\\
& \lesssim & \norm{T_{\mu_z}^* 1}_{L^{p_1}}\prod_{l=1}^n\frac{1}{\left(1-\abs{w_l}^2\right)^{\frac{1}{p}}}\\
& \lesssim & \left(\,\sup_{z\in\D^n} \norm{T_{\mu_z} 1}_{L^{p_1}}\right) h(w)^{q}.
\end{eqnarray*}
Lemma \ref{Schur} gives the desired result.

\end{proof}

\begin{lm}
\label{bo-lip}
If $|\mu|$ is a Bergman--Carleson measure and $k\ge 0$, then $B_k(\mu)$ is a bounded Lipschitz function from $\left(\D^n, \rho\right)$ into $\left(\mathbb{C}, \abs{\,\cdot\,}\right)$. Specifically, there are constants depending on $n, k$ such that
$$
|B_k(\mu)(z)| \lesssim  \,  B_0(|\mu|)(z) \ \ \mbox{ and }\ \
|B_{k}(\mu) (z_1) - B_{k}(\mu) (z_2)| \lesssim    \,
\norm{B_{0}(|\mu|)}_\infty  \, \rho\left(z_1,z_2\right).
$$
\end{lm}

\begin{proof}
If $\abs{\mu}$ is a Bergman--Carleson measure, then from \eqref{Berezin-k-poly} we have that 
\begin{eqnarray*}
\abs{B_k(\mu)(z)} & \leq & \left(k+1\right)^n\abs{B_0(\abs{\mu})(z)}.
\end{eqnarray*}
The fact that $B_k(\mu)$ is Lipschitz continuous follows from \cite{NZ}*{Theorem 2.8}.
\end{proof}
As a consequence of Lemma \ref{bo-lip}, we have that $B_k(\mu)\in\mathcal{A}$ for all $k\geq 0$.  Indeed, to see that $B_k(\mu)$ is bounded, simply note that since $\abs{\mu}$ is a Bergman--Carleson measure one has 
$$
\norm{B_k(\mu)}_{L^\infty}\lesssim \norm{B_0(\abs{\mu})}_{L^\infty}=\norm{\mu}_{\textnormal{RKM}}.
$$
While the second condition in Lemma \ref{bo-lip} implies that $B_k(\mu)$ is uniformly continuous from $\left(\D^n,\rho\right)$ to $\left(\C,\abs{\,\cdot\,}\right)$.

Now let $\mu$ denote an absolutely continuous measure with respect to $dv$, so we have that $\mu=a\,dv$, with $a\in L^1$.  From \eqref{Berezin-k-poly}, we have
$$
B_k(a\,dv)(z):=B_k(a)(z)=\left(k+1\right)^n\int_{\D^n}\prod_{l=1}^n\frac{\left(1-\abs{z_l}^{2}\right)^2}{\abs{1-\overline{w}_l z_l}^4}\prod_{l=1}^n\left(1-\abs{\varphi_{z_l}(w_l)}^2\right)^k a(w)\, dv(w).
$$
Now, observe that upon making the change of variables, $w=\varphi_z(\xi)$ in the integrand gives,
$$
B_k(a)(z)=\left(k+1\right)^n\int_{\D^n}\prod_{l=1}^n\left(1-\abs{\xi_l}^2\right)^k a\left(\varphi_z(\xi)\right)\,dv(\xi).
$$
Recall that we are letting $\varphi_z(\xi):=\left(\varphi_{z_1}(\xi_1),\ldots, \varphi_{z_n}(\xi_n)\right)$.  We then observe the following lemma.

\begin{lm}
Let $a\in\mathcal{A}$, then we have 
\begin{equation}
\label{ApproxinA}
\lim_{k\to\infty}\norm{B_k(a)-a}_{L^\infty}=0.
\end{equation}
\end{lm}
\begin{proof}
First, a simple computation shows
$$
\left(k+1\right)^n\int_{\D^n}\prod_{l=1}^n\left(1-\abs{\xi_l}^2\right)^k\,dv(\xi)=1.
$$
Then we have
\begin{eqnarray*}
\abs{B_k(a)(z)-a(z)} & = & \abs{\left(k+1\right)^n\int_{\D^n}\prod_{l=1}^n\left(1-\abs{\xi_l}^2\right)^k a\left(\varphi_z(\xi)\right)\, dv(\xi)-a(z)}\\
& = & \abs{\left(k+1\right)^n\int_{\D^n}\prod_{l=1}^n\left(1-\abs{\xi_l}^2\right)^k \left(a\left(\varphi_z(\xi)\right)-a(z)\right)\, dv(\xi)}\\
& \leq & \left(k+1\right)^n\int_{\D^n}\abs{a(\varphi_z(\xi))-a(\varphi_z(0))}\prod_{l=1}^n\left(1-\abs{\xi_l}^2\right)^k \,dv(\xi).
\end{eqnarray*}
Since $a\in\mathcal{A}$ is uniformly continuous from $\left(\D^n,\rho\right)$ into $\left(\C,\abs{\,\cdot\,}\right)$, for any $\epsilon>0$, there exists a $\delta>0$, such that if $\max_{1\leq l\leq n}\abs{\xi_l}<\delta$ then $\abs{a(\varphi_z(\xi))-a(z)}=\abs{a(\varphi_z(\xi))-a(\varphi_z(0))}<\epsilon$.  Using this, we have
\begin{eqnarray*}
\abs{B_k(a)(z)-a(z)} & \leq & \left(k+1\right)^n\int_{\left\{\xi:\max_{1\leq l\leq n}\abs{\xi_l}<\delta\right\}}\abs{a(\varphi_z(\xi))-a(\varphi_z(0))}\prod_{l=1}^n\left(1-\abs{\xi_l}^2\right)^k \,dv(\xi)\\
&  & + \left(k+1\right)^n\int_{\D^n\setminus \left\{\xi:\max_{1\leq l\leq n}\abs{\xi_l}<\delta\right\}} \abs{a(\varphi_z(\xi))-a(\varphi_z(0))}\prod_{l=1}^n\left(1-\abs{\xi_l}^2\right)^k \,dv(\xi)\\
& \leq & \epsilon+\left(k+1\right)^n\int_{\D^n\setminus \left\{\xi:\max_{1\leq l\leq n}\abs{\xi_l}<\delta\right\}}\abs{a\left(\varphi_z(\xi)\right)-a(\varphi_z(0))}\prod_{l=1}^n\left(1-\abs{\xi_l}^2\right)^k \,dv(\xi)\\
& \leq & \epsilon + 2\sup_{z\in\D^n}\abs{a(z)}\left(k+1\right)^n\int_{\D^n\setminus\left\{\xi:\max_{1\leq l\leq n}\abs{\xi_l}<\delta\right\}}\prod_{l=1}^n\left(1-\abs{\xi_l}^2\right)^k \,dv(\xi).
\end{eqnarray*}
Note now that this last expression can be made as small as desired by taking $k$ sufficiently large.  Indeed, a simple computation using the splitting of the domain $\D^n\setminus\left\{\xi:\max_{1\leq l\leq n}\abs{\xi_l}<\delta\right\}$ as in the proof of Lemma \ref{Tech1}, shows
$$
\left(k+1\right)^n\int_{\D^n\setminus\left\{\xi:\max_{1\leq l\leq n}\abs{\xi_l}<\delta\right\}}\prod_{l=1}^n\left(1-\abs{\xi_l}^2\right)^k \,dv(\xi)\lesssim \left(1-\delta^2\right)^{k+1}.
$$
Since $z\in\D^n$ was arbitrary, we have \eqref{ApproxinA}.
\end{proof}

\begin{thm}
Let $1<p<\infty$ and $\mu$ a complex-valued measure such that $\abs{\mu}$ is a Bergman--Carleson measure.  Then $T_{B_k(\mu)}\to T_\mu$ in $\mathcal{L}\left(A^p,A^p\right)$.  In particular, we have that $\mathcal{T}_p$ is the closed algebra generated by $\{T_a: a\in\mathcal{A}\}$. 
\end{thm}

\begin{proof}
First, we note that by \cite{NZ}*{Proposition 2.11} we have $B_0B_k(\mu)=B_kB_0(\mu)$.  Now, $B_0(\mu)\in\mathcal{A}$, and so by \eqref{ApproxinA},
$$
\norm{B_0\left(B_k(\mu)dv-\mu\right)}_{L^\infty} = \norm{B_0B_k(\mu)- B_0(\mu)}_{L^\infty}=\norm{B_kB_0(\mu)-B_0(\mu)}_{L^\infty}\to 0.
$$
Using Lemma \ref{bo-lip} we have $\norm{B_k(\mu)dv}_{\textnormal{CM}}\lesssim\norm{\mu}_{\textnormal{CM}}$, with the implied constant independent of $k$.  This in turn gives that $\norm{T_{B_k(\mu)}-T_\mu}_{\mathcal{L}\left(A^p,A^p\right)}$ is bounded by a multiple of $\norm{\mu}_{\textnormal{CM}}$, independent of $k$.  By these computations, we have a sequence of bounded operators $\{T_{B_k(\mu)-d\mu}\}$ in $\mathcal{L}\left(A^2,A^2\right)$ such that $\norm{B_0(B_k(\mu)dv-d\mu)}_{L^\infty}\to 0$ as $k\to\infty$.  We can now apply \cite{NZ}*{Lemma 3.4} to  conclude that
\begin{equation}
\label{UniformLimit}
\sup_{z\in\D^n}\abs{\left(T_{B_k(\mu)dv-d\mu}\right)_z 1}\to 0
\end{equation}
uniformly on compact subsets of $\D^n$ as $k\to\infty$.  Now observe that we have
$$
\left(T_{B_k(\mu)dv-d\mu}\right)_z=T_{\left(B_k(\mu)dv-d\mu\right)_z},
$$
where this is the measure obtained by composing with the change of variables $\varphi_z$.  By Lemma \ref{Lemma_Change_Est} we have that
\begin{eqnarray*}
\norm{T_{B_k(\mu)}-T_\mu}_{\mathcal{L}\left(A^p,A^p\right)} & = & \norm{T_{B_k(\mu)dv-d\mu}}_{\mathcal{L}\left(A^p,A^p\right)}\\
 & \lesssim & \left(\,\sup_{z\in\D^n}\norm{T_{(B_k(\mu)dv-d\mu)_z} 1}_{p_1}\right)^{\frac{1}{p}}\left(\,\sup_{z\in\D^n}\norm{T_{(B_k(\mu)dv-d\mu)_z}^* 1}_{p_1}\right)^{\frac{1}{q}}.
\end{eqnarray*}
The goal is now to show that this last expression can be made small as $k\to\infty$.

Let $\epsilon>0$ and set $F_{k,z}(w):=T_{\left(B_k(\mu)dv-d\mu\right)_z}1(w)$.  For $0<r<1$ set $\left(r\D^n\right)^c=\D^n\setminus r\D^n$.  Choose $1<p_1<\infty$ so that Lemma \ref{Lemma_Change_Est} holds for this value of $p$.  Also, select $0<r<1$ so that $\norm{\mu}_{\textnormal{CM}}^{p_1}\norm{1_{(r\D^n)^c}}_{L^2}<\frac{\epsilon}{2}$.  Then by splitting the integral, we have
$$
\norm{F_{k,z}}_{L^{p_1}}^{p_1}=\norm{F_{k,z}1_{r\D^n}}_{L^{p_1}}^{p_1}+\norm{F_{k,z}1_{(r\D^n)^c}}_{L^{p_1}}^{p_1}.
$$
We will show that each of these terms can be made sufficiently small.  First, observe by Cauchy-Schwarz that
\begin{eqnarray*}
\norm{F_{k,z}1_{(r\D^n)^c}}_{L^{p_1}}^{p_1} & \leq & \norm{F_{k,z}}_{L^{2p_1}}^{p_1}\norm{1_{(r\D^n)^c}}_{L^2}\\
& = & \norm{T_{\left(B_k(\mu)dv-d\mu\right)_z}1}_{L^{2p_1}}^{p_1}\norm{1_{(r\D^n)^c}}_{L^2}\\
& \lesssim & \left(\norm{(B_k(\mu)dv)_z}_{\textnormal{CM}}+\norm{(d\mu)_z}_{\textnormal{CM}}\right)^{p_1}\norm{1_{(r\D^n)^c}}_{L^2}\\
& \lesssim & \norm{\mu}_{\textnormal{CM}}^{p_1}\norm{1_{(r\D^n)^c}}_{L^2}\lesssim \frac{\epsilon}{2}.
\end{eqnarray*}
In the second inequality we have used \eqref{MagicalEquality} and Lemma \ref{CM}.  Now, for the given value of $r$, we have by \eqref{UniformLimit} that for $w\in r\D^n$ and for $k$ sufficiently large that
$$
\sup_{z\in\D^n}\abs{\left(T_{B_k(\mu)dv-d\mu}\right)_z 1(w)}<\frac{\epsilon}{2}\quad \forall w\in r\D^n.
$$
Using this, we have
\begin{eqnarray*}
\sup_{z\in\D^n}\norm{F_{k,z}1_{r\D^n}}_{L^{p_1}}^{p_1} & = & \sup_{z\in\D^n} \int_{r\D^n}\abs{T_{\left(B_k(\mu)dv-d\mu\right)_z}1(w)}^{p_1}dv(w)\\
& \leq & \int_{r\D^n}\sup_{z\in\D^n}\abs{T_{\left(B_k(\mu)dv-d\mu\right)_z}1(w)}^{p_1}dv(w)\\
& \leq & \left(\frac{\epsilon}{2}\right)^{p_1}.
\end{eqnarray*}
Combining these estimates we see that
$$
\left(\,\sup_{z\in\D^n}\norm{T_{(B_k(\mu)dv-d\mu)_z} 1}_{p_1}\right)^{\frac{1}{p}}\lesssim \epsilon^{\frac{1}{p}}
$$
provided $k$ is sufficiently large.  These computations can then be repeated when we observe that
$$
T_{(B_k(\mu)dv-d\mu)_z}^*=T_{(B_k(\overline{\mu})dv-d\overline{\mu})_z}.
$$
Thus, we have
\begin{eqnarray*}
\norm{T_{B_k(\mu)}-T_\mu}_{\mathcal{L}\left(A^p,A^p\right)} & \lesssim & \left(\,\sup_{z\in\D^n}\norm{T_{(B_k(\mu)dv-d\mu)_z} 1}_{p_1}\right)^{\frac{1}{p}}\left(\,\sup_{z\in\D^n}\norm{T_{(B_k(\mu)dv-d\mu)_z}^* 1}_{p_1}\right)^{\frac{1}{q}}\lesssim\epsilon
\end{eqnarray*}
provided $k$ is chosen large enough.
\end{proof}

\section{Approximation by Segmented Operators}
\label{Approximation}

\begin{lm}
\label{Approx1}
Let $1<p<\infty$ and $\sigma\geq 1$.  Suppose that $a_1,\ldots, a_k\in L^\infty$ are functions of norm at most 1 and that $\mu$ is a Bergman--Carleson measure.  Consider the covering of $\D^n$ given by Lemma \ref{SuaGeo2} for these values of $k$ and $\sigma$.  Then there is a positive constant depending on $p$, $k$, and the dimension such that
\begin{equation}
\label{Approx-Est1}
\norm{\left[\, \prod_{i=1}^{k} T_{a_i} \right] T_{\mu}-\sum_{\vec{j}} M_{1_{F_{0,\vec{j}}}}
\left[\, \prod_{i=1}^{k}T_{a_i}  \right] T_{1_{F_{k+1,\vec{j}}}\mu}}_{\mathcal{L}\left(A^p,L^p\right)}\lesssim \beta_{p}(\sigma)\norm{T_{\mu}}_{\mathcal{L}\left(A^p, A^p\right)}
\end{equation}
where $\beta_{p}(\sigma)\to 0$ as $\sigma\to\infty$.
\end{lm}
\begin{proof}
We break the proof up into two steps.  We will prove that
\begin{equation}
\label{Step1}
\norm{\left[\, \prod_{i=1}^{k}T_{a_i} \right] T_{\mu}-\sum_{\vec{j}} M_{1_{F_{0,\vec{j}}}}
\left[\, \prod_{i=1}^{k} T_{a_i1_{F_{i,\vec{j}}}} \right] T_{1_{F_{k+1,\vec{j}}}\mu}}_{\mathcal{L}\left(A^p,L^p\right)} \lesssim  \beta_{p}(\sigma)\norm{T_{\mu}}_{\mathcal{L}\left(A^p, A^p\right)},
\end{equation}
and
\begin{equation}
\label{Step2}
\norm{\sum_{\vec{j}}M_{1_{F_{0,\vec{j}}}}\left[\left[\, \prod_{i=1}^{k}T_{a_i}  \right] T_{1_{F_{k+1,\vec{j}}}\mu}-\left[\, \prod_{i=1}^{k} T_{a_i 1_{F_{i,\vec{j}}}}  \right]T_{1_{F_{k+1,\vec{j}}}\mu}\right]}_{\mathcal{L}\left(A^p, L^p\right)}\lesssim \beta_{p}(\sigma)\norm{T_\mu}_{\mathcal{L}\left(A^p, A^p\right)}.
%\norm{\sum_{\vec{j}}M_{1_{F_{0,\vec{j}}}}\left[\, \prod_{i=1}^{k}T_{a_i}  \right] T_{1_{F_{k+1,\vec{j}}}\mu}-\sum_{\vec{j}}M_{1_{F_{0,\vec{j}}}}\left[\, \prod_{i=1}^{k} T_{a_i 1_{F_{i,\vec{j}}}}  \right]T_{1_{F_{k+1,\vec{j}}}\mu}}_{\mathcal{L}\left(A^p, L^p\right)}\lesssim \beta_{p}(\sigma)\norm{T_\mu}_{\mathcal{L}\left(A^p, A^p\right)}.
\end{equation}
It is obvious that each of these inequalities when combined give the desired estimate in the statement of the lemma.

For $0\leq m\leq k+1$, define the operators $X_m\in \mathcal{L}\left(A^p,L^p\right)$ by
$$
X_m=\sum_{\vec{j}} M_{1_{F_{0,\vec{j}}}}\left[\, \prod_{i=1}^{m} T_{1_{F_{i,\vec{j}}}a_i}\prod_{i=m+1}^{k} T_{a_i} \right] T_\mu.
$$
Then clearly, 
$X_0=\sum_{\vec{j}} M_{1_{F_{0,\vec{j}}}} \left[\,\prod_{i=1}^{k} T_{a_i}\right] T_\mu
= \left[\,\prod_{i=1}^{k} T_{a_i}\right]T_\mu$, with convergence in the strong operator topology.  Similarly, we have 
$$
X_{k+1}=\sum_{\vec{j}} M_{1_{F_{0,\vec{j}}}}\left[\,\prod_{i=1}^{k} T_{a_i1_{F_{i,\vec{j}}}} \right] T_{\mu 1_{F_{k+1,\vec{j}}}}.
$$
We now seek and estimate on the operator norm of $X_0-X_{k+1}$, with the idea being to use a telescoping sum and estimate each difference.  When $0\leq m\leq k-1$, a simple computation shows
$$
X_m-X_{m+1}=\sum_{\vec{j}} M_{1_{F_{0,\vec{j}}}} \left[\,\prod_{i=1}^{m}T_{1_{F_{i,\vec{j}}}a_i} \right]
T_{1_{F_{m+1,\vec{j}}^c}a_{m+1}}  \left[\, \prod_{i=m+2}^{k}T_{a_i}\right]  T_\mu.
$$
Here, of course, we should interpret this product as the identity when the lower index is greater than the upper index.  Take any $f\in A^p$ and apply Lemma \ref{Tech2}, in particular \eqref{Tech2-Est2}, Lemma \ref{SuaGeo2} to the measure $dv$ (see Remark \ref{Rem_Countable} as well) along with some obvious estimates to see that
\begin{eqnarray*}
\norm{\left(X_m-X_{m+1}\right)f}_{L^p}^p & \lesssim & \sum_{\vec{j}}\norm{M_{1_{F_{m,\vec{j}}}a_m}P M_{1_{F_{m+1,\vec{j}}^c}a_{m+1}}\left[\,\prod_{i=m+2}^k T_{a_i} \right] T_\mu f}_{L^p}^p\\
& \lesssim & N\beta_{p}^p(\sigma)\norm{\left[\,\prod_{i=m+2}^k T_{a_i} \right] T_\mu f}_{L^p}^p\\
& \lesssim & N\beta_{p}^p(\sigma)\norm{T_\mu}_{\mathcal{L}\left(A^p,A^p\right)}^p\norm{f}_{A^p}^p.
\end{eqnarray*}
Also, we have that
$$
X_{k}-X_{k+1}=\sum_{\vec{j}} M_{1_{F_{0,\vec{j}}}} \left[\, \prod_{i=1}^{k}T_{1_{F_{i,\vec{j}}}a_i} \right] T_{\mu 1_{F_{k+1,\vec{j}}^c}}, 
$$
and so
\begin{eqnarray*}
\norm{(X_k-X_{k+1})f}_{L^p}^p & = & \norm{\sum_{\vec{j}} M_{1_{F_{0,\vec{j}}}} \left[\, \prod_{i=1}^{k}T_{1_{F_{i,\vec{j}}}a_i} \right] T_{\mu 1_{F_{k+1,\vec{j}}^c}}f}_{L^p}^p\\
& = & \sum_{\vec{j}} \norm{M_{1_{F_{0,\vec{j}}}} \left[\, \prod_{i=1}^{k}T_{1_{F_{i,\vec{j}}}a_i} \right] T_{\mu 1_{F_{k+1,\vec{j}}^c}}f}_{L^p}^p\\
& \lesssim & \sum_{\vec{j}}\norm{M_{1_{F_{k,\vec{j}}}a_k}T_{\mu 1_{F_{k+1,\vec{j}}^c}}f}_{L^p}^p\\
 & \lesssim & N\beta_{p}^p(\sigma)\norm{T_\mu}_{\mathcal{L}\left(A^p,A^p\right)}^p\norm{f}^p_{A^p}.
\end{eqnarray*}
Here we used that $\{F_{0,\vec{j}}\}$ is a disjoint cover of $\D^n$, obvious estimates and applying Lemma \ref{Tech2}, and in particular \eqref{Tech2-Est2} at the second to last inequality.  Since $N=N(n)$, we have the following estimates for $0\leq m\leq k$,
$$
\norm{\left(X_m-X_{m+1}\right)f}_{L^p}\lesssim \beta_{p}(\sigma)\norm{T_\mu}_{\mathcal{L}\left(A^p,A^p\right)}\norm{f}_{A^p}.
$$
But from this it is immediate that estimate \eqref{Step1} holds, since
$$
\norm{\left(X_0-X_{k+1}\right)f}_{L^p}\leq \sum_{m=0}^{k}\norm{\left(X_m-X_{m+1}\right)f}_{L^p}\lesssim \beta_{p}(\sigma)\norm{T_\mu}_{\mathcal{L}\left(A^p,A^p\right)}\norm{f}_{A^p}.
$$

The idea behind \eqref{Step2} is similar.  For $0\leq m\leq k$, define the operator
$$
\tilde{X}_m=\sum_{\vec{j}} M_{1_{F_{0,\vec{j}}}}\left[\,\prod_{i=1}^{m} T_{1_{F_{i,\vec{j}}}a_i}\prod_{i=m+1}^{k} T_{a_i}\right] T_{\mu 1_{F_{k+1,\vec{j}}}},
$$
and so when $m=0$ and $m=k$ we have 
\begin{eqnarray*}
\tilde{X}_0 & = & \sum_{\vec{j}} M_{1_{F_{0,\vec{j}}}}\left[\,\prod_{i=1}^{k} T_{a_i}\right] T_{\mu 1_{F_{k+1,\vec{j}}}}\\
\tilde{X}_k & = & \sum_{\vec{j}} M_{1_{F_{0,\vec{j}}}}\left[\,\prod_{i=1}^{k}T_{a_i1_{F_{i,\vec{j}}}}\right]T_{\mu 1_{F_{k+1,\vec{j}}}}.
\end{eqnarray*}
For $0\leq m\leq k-1$, a simple computation shows
$$
\tilde{X}_m-\tilde{X}_{m+1}=\sum_{\vec{j}} M_{1_{F_{0,\vec{j}}}}\left[\,\prod_{i=1}^{m} T_{1_{F_{i,\vec{j}}}a_i}\right]
T_{1_{F_{m+1,\vec{j}}^c}a_{m+1}}  \left[\,\prod_{i=m+2}^{k}T_{a_i}\right]  T_{\mu 1_{F_{k+1,\vec{j}}}}.
$$
Again, applying obvious estimates and using Lemma \ref{Tech2} one can conclude that
\begin{eqnarray*}
\norm{\left(\tilde{X}_m-\tilde{X}_{m+1}\right)f}^p_{L^p} & \lesssim & \beta_{p}^{p}(\sigma)\sum_{\vec{j}}\norm{T_{1_{F_{k+1,\vec{j}}}\mu}f}_{A^p}^{p}\\
 & \lesssim & \beta_{p}^{p}(\sigma) \norm{T_{\mu}}_{\mathcal{L}\left(A^p,A^p\right)}^{\frac{p}{q}} \sum_{\vec{j}} \norm{1_{F_{k+1,\vec{j}}} f}_{L^p(\mu)}^p\\
 & \lesssim & N\beta_{p}^{p}(\sigma)
 \norm{T_{\mu}}_{\mathcal{L}\left(A^p,A^p\right)}^p  \norm{f}_{A^p}^p.
\end{eqnarray*}
In the above estimates, we used Lemma \ref{CM-Cor} twice and that the sets $\left\{F_{k+1,\vec{j}}\right\}$ form a covering of $\D^n$ having at most $N=N(n)$ overlap.  Concluding, for $0\leq m\leq k-1$, 
$$
\norm{\left(\tilde{X}_m-\tilde{X}_{m+1}\right)f}_{L^p}\lesssim \beta_{p}(\sigma)\norm{T_{\mu}}_{\mathcal{L}\left(A^p,A^p\right)}\norm{f}_{A^p}.
$$
But, then this implies that
$$
\norm{\left(\tilde{X}_0-\tilde{X}_{k}\right)f}_{L^p}\leq \sum_{m=0}^{k-1}\norm{\left(\tilde{X}_m-\tilde{X}_{m+1}\right)f}_{L^p}\lesssim \beta_{p}(\sigma)\norm{T_{\mu}}_{\mathcal{L}\left(A^p,A^p\right)}\norm{f}_{A^p},
$$
which is clearly \eqref{Step2}.
\end{proof}

\begin{lm}
\label{Approx2}
Let
$$
X=\sum_{i=1}^{m} \left[\,\prod_{l=1}^{k_i} T_{a^i_l} \right] T_{\mu_i}
$$
where $a^i_j\in L^\infty$, $k_1,\ldots, k_m\leq k$ and $\mu_i$ are complex-valued measures on $\D^n$ such that $\abs{\mu_i}$ are Bergman--Carleson measures.  Given $\epsilon>0$ there is $\sigma=\sigma\left(X,\epsilon\right)\geq 1$ such that if $\left\{F_{i,\vec{j}}\right\}$ and $0\leq i\leq k+1$ are the sets given by Lemma \ref{SuaGeo2} for these values of $\sigma$ and $k$, then
$$
\norm{X-\sum_{\vec{j}}M_{1_{F_{0,\vec{j}}}}\sum_{i=1}^{m} \left[\,\prod_{l=1}^{k_i} T_{a^i_l} \right] T_{1_{F_{k+1,\vec{j}}}\mu_i}}_{\mathcal{L}\left(A^p,L^p\right)}<\epsilon.
$$
\end{lm}
\begin{proof}
First, suppose that $\mu_i$ are are non-negative measures.  It suffices to prove the result in this situation since for general $\mu_i$ we can decompose $\mu_i=\mu_{i,1}-\mu_{i,2}+i\mu_{i,3}-i\mu_{i,4}$ where each $\mu_{i,j}$ is a non-negative measure, and hence a Bergman--Carleson measure.
%% Then can then use the case when $\mu_i$ is non-negative, which we will show below, plus the linearity of limits in the
%% strong operator conclude the Lemma.

Without loss of generality, set $k_i=k$ for all $i=1,\ldots, m$.  This can be accomplished by placing copies of the identity in each product if necessary.  We now apply Lemma \ref{Approx1} to each term in the product defining the operator $X$.  By Lemma \ref{Approx1}, for $\sigma=\sigma\left(X,\epsilon\right)$ sufficiently large we have
$$
\norm{\left[\,\prod_{j=1}^{k} T_{a^i_j} \right] T_{\mu_i}-
\sum_{\vec{j}} M_{1_{F_{0,\vec{j}}}}\left[\,\prod_{l=1}^{k} T_{a^i_l} \right] T_{1_{F_{k+1,\vec{j}}}\mu_i}}_{\mathcal{L}\left(A^p,L^p\right)}<\frac{\epsilon}{m}
$$
for $i=1,\ldots, m$, and then summing in $i$ one obtains
$$
\norm{X-\sum_{i=1}^{m}\sum_{\vec{j}}M_{1_{F_{0,\vec{j}}}} \left[\,\prod_{l=1}^{k_i} T_{a^i_l} \right] T_{1_{F_{k+1,\vec{j}}}\mu_i}}_{\mathcal{L}\left(A^p,L^p\right)}<\epsilon.
$$
But, for every $i=1,\ldots, m$ we have that
$\sum_{\vec{j}} M_{1_{F_{0,\vec{j}}}} \left[\,\prod_{j=1}^{k_i} T_{a^i_j}\right] T_{1_{F_{k+1,\vec{j}}}\mu_i}$ converges in the strong operator topology, so
$$
\norm{X-\sum_{\vec{j}} M_{1_{F_{0,\vec{j}}}}\sum_{i=1}^{m} \left[\,\prod_{l=1}^{k_i} T_{a^i_l}\right] T_{1_{F_{k+1,\vec{j}}}\mu_i}}_{\mathcal{L}\left(A^p,L^p\right)}<\epsilon
$$
as desired.
\end{proof}

\begin{lm}
\label{Approx3}
Let $S\in\mathcal{T}_{p}$, $\mu$ be a Bergman--Carleson measure and $\epsilon>0$.  Then there are Borel sets $F_{\vec{j}}\subset G_{\vec{j}}\subset\D^n$ such that
\begin{itemize}
\item[(i)] $\D^n=\bigcup F_{\vec{j}}$;
\item[(ii)] $F_{\vec{j}}\bigcap F_{\vec{j}'}=\emptyset$ if $\vec{j}\neq \vec{j}'$;
\item[(iii)] each point of $\D^n$ lies in no more than $N(n)$ of the sets $G_{\vec{j}}$;
\item[(iv)] $\textnormal{diam}_{\rho}\, G_{\vec{j}}\leq d(p,S,\epsilon)$
\end{itemize}
and
$$
\norm{ST_\mu-\sum_{\vec{j}} M_{1_{F_{\vec{j}}}}ST_{1_{G_{\vec{j}}}\mu}}_{\mathcal{L}\left(A^p, L^p\right)}<\epsilon.
$$
\end{lm}
\begin{proof}
Since $S\in \mathcal{T}_{p}$ there is $X_0=\sum_{i=1}^{m}\prod_{l=1}^{k_i}T_{a_l^i}$ such that
$$
\norm{S-X_0}_{\mathcal{L}\left(A^p, A^p\right)}\lesssim\frac{\epsilon}{\norm{T_\mu}_{\mathcal{L}\left(A^p,A^p\right)}},
$$
where $a_l^i\in L^\infty$ and $k_i$ are positive integers.
%%  Here $C$ is some explicit constant depending only on the dimension, and the norm of the Carleson measure $\mu$.
Set $k=\max\left\{k_i:i=1,\ldots,m\right\}$.  By Lemma \ref{Approx2} we can choose $\sigma=\sigma(X_0,\epsilon)$ and
sets $F_{\vec{j}}=F_{0,\vec{j}}$ and $G_{\vec{j}}=F_{k+1,\vec{j}}$ such that
$$
\norm{X_0T_\mu-\sum_{\vec{j}} M_{1_{F_{\vec{j}}}} X_0 T_{\mu 1_{G_{\vec{j}}}}}_{\mathcal{L}\left(A^p,L^p\right)}
<\epsilon.
$$
We have that (i), (ii), (iii) and (iv) clearly hold by Lemma \ref{SuaGeo2}.  Now, for $f\in A^p$ we have
\begin{eqnarray*}
\norm{\sum_{\vec{j}} M_{1_{F_{\vec{j}}}}\left(S-X_0\right)T_{\mu 1_{G_{\vec{j}}}}f}_{L^p}^p & = & \sum_{\vec{j}}\norm{ M_{1_{F_{\vec{j}}}}\left(S-X_0\right)T_{\mu 1_{G_{\vec{j}}}}f}_{L^p}^p\\
& \leq & \left(\frac{\epsilon}{\norm{T_\mu}_{\mathcal{L}\left(A^p, A^p\right)}}\right)^p\sum_{\vec{j}} \norm{T_{1_{G_{\vec{j}}}\mu} f}_{A^p}^p\\
& \leq & \left(\frac{\epsilon}{\norm{T_\mu}_{\mathcal{L}\left(A^p, A^p\right)}}\right)^p\sum_{\vec{j}} \norm{1_{G_{\vec{j}}}f}_{L^p(\mu)}^p\\
& \lesssim & \epsilon^p\norm{f}_{A^p}^p.
\end{eqnarray*}
Therefore, the triangle inequality gives
\begin{eqnarray*}
\norm{ST_\mu-\sum_{\vec{j}} M_{1_{F_{\vec{j}}}}ST_{\mu 1_{G_{\vec{j}}}}}_{\mathcal{L}\left(A^p, L^p\right)} &\leq&
\norm{S-X_0}_{\mathcal{L}\left(A^p, A^p\right)} \norm{T_\mu}_{\mathcal{L}\left(A^p, A^p\right)} + \epsilon \\
&+&
\norm{\sum_{\vec{j}} M_{1_{F_{\vec{j}}}}\left(S-X_0\right)T_{\mu 1_{G_{\vec{j}}}}}_{\mathcal{L}\left(A^p, L^p\right)}
\lesssim \epsilon,
\end{eqnarray*}
which gives the lemma.
\end{proof}

\section{Characterization of the Essential Norm on \texorpdfstring{$A^p$}{Bergman Spaces}}
\label{Characterization}

We have now collected enough tools to provide the characterization of the essential norm of an operator on $A^p$.  Fix $\varrho>0$ and let $\{w_{\vec{m}}\}$ and $D_{\vec{m}}$ be the sets in Lemma \ref{StandardGeo_Polydisc}.  Define the measure
$$
\mu_{\varrho}:=\sum_{\vec{m}} v(D_{\vec{m}}) \delta_{w_{\vec{m}}}\approx\sum_{\vec{m}} \prod_{l=1}^n(1-\abs{w_{m_l}}^2)^2 \delta_{w_{\vec{m}}}.
$$
Then we have that $\mu_\varrho$ is a Bergman--Carleson measure and $T_{\mu_\varrho}:A^p\to A^p$ is bounded.  Looking in Coifman and Rochberg, \cite{CR}, one can see that the following lemma holds.  The interested reader can also see results of this type in results of Amar, \cite{Amar} and Rochberg, \cite{Roch}.  Again for completeness, we provide the details.
\begin{lm}
For $1<p<\infty$,  $T_{\mu_\varrho}\to I_{A^p}$ on $\mathcal{L}\left(A^p,A^p\right)$ when $\varrho\to 0$.
\end{lm}

\begin{proof}
The main idea behind the proof is to compare $f(z)$ via the reproducing formula to $T_{\mu_\varrho}f(z)$ and to obtain an estimate of the form
\begin{equation}
\label{Key_Est}
\norm{\left(T_{\mu_\varrho}-I_{A^p}\right)f}_{A^p}\lesssim\varrho\norm{f}_{A^p},
\end{equation}
with the implied constant depending only on $p$ and the dimension.  This estimate would prove the desired result.  Furthermore, it suffices to prove this estimate on a dense class of functions, and so without loss of generality suppose that $f\in \textnormal{Pol}\left(\D^n\right)$, where $\textnormal{Pol}\left(\D^n\right)$ is the collection of analytic polynomials on $\D^n$.  First, note that
\begin{eqnarray*}
\abs{f(z)-T_{\mu_\varrho}f(z)} & = & \abs{\int_{\D^n} f(w)\prod_{l=1}^n \frac{1}{\left(1-\overline{w}_lz_l\right)^2}\,dv(w)-\sum_{\vec{m}}f(w_{\vec{m}})v(D_{\vec{m}})\prod_{l=1}^n\frac{1}{\left(1-\overline{w}_{m_l}z_l\right)^2}}\\
% & = & \abs{\sum_{\vec{m}}\left(\int_{D_{\vec{m}}}f(w)\prod_{l=1}^n \frac{1}{\left(1-\overline{w}_lz_l\right)^2}\,dv(w)-f(w_{\vec{m}})v(D_{\vec{m}})\prod_{l=1}^n\frac{1}{\left(1-\overline{w}_{m_l}z_l\right)^2}\right)}\\
  & = & \abs{\sum_{\vec{m}}\int_{D_{\vec{m}}}\left(f(w)\prod_{l=1}^n \frac{1}{\left(1-\overline{w}_lz_l\right)^2}-f(w_{\vec{m}})\prod_{l=1}^n\frac{1}{\left(1-\overline{w}_{m_l}z_l\right)^2}\right)\,dv(w)}.
\end{eqnarray*}
Here we have used (iii) of Lemma \ref{StandardGeo_Polydisc}.  Consider the integrand in the above expression, by adding and subtracting a common term we can write this in two different ways.  In particular, we have
\begin{eqnarray*}
\int_{D_{\vec{m}}}\left(f(w)\prod_{l=1}^n \frac{1}{\left(1-\overline{w}_lz_l\right)^2}-f(w_{\vec{m}})\prod_{l=1}^n\frac{1}{\left(1-\overline{w}_{m_l}z_l\right)^2}\right)\,dv(w) & = & I_{\vec{m}}+II_{\vec{m}},
\end{eqnarray*}
where
\begin{eqnarray*}
I_{\vec{m}} & := & \int_{D_{\vec{m}}}f(w)\left(\,\prod_{l=1}^n \frac{1}{(1-\overline{w}_lz_l)^2}-\prod_{l=1}^n \frac{1}{\left(1-\overline{w}_{m_l}z_l\right)^2}\right)\,dv(w)\\
II_{\vec{m}} & := & \int_{D_{\vec{m}}}\left(f(w)-f(w_{\vec{m}})\right)\prod_{l=1}^n \frac{1}{\left(1-\overline{w}_{m_l}z_l\right)^2}\,dv(w).
\end{eqnarray*}
Consider term $I_{\vec{m}}$, and note
\begin{eqnarray}
\abs{I_{\vec{m}}}  & \leq & \int_{D_{\vec{m}}}\abs{f(w)}\abs{\prod_{l=1}^n \frac{1}{\left(1-\overline{w}_lz_l\right)^2}-\prod_{l=1}^n \frac{1}{(1-\overline{w}_{m_l}z_l)^2}}\,dv(w)\notag\\
 & \lesssim & \int_{D_{\vec{m}}}\abs{f(w)}\rho\left(w,w_{\vec{m}}\right)\,dv(w) \prod_{l=1}^n \frac{1}{\abs{1-\overline{w}_{m_l}z_l}^2}\notag\\
 & \lesssim & \varrho \int_{D_{\vec{m}}}\abs{f(w)}\,dv(w)\prod_{l=1}^n \frac{1}{\abs{1-\overline{w}_{m_l}z_l}^2}\notag\\
 & \leq & \varrho \left(\int_{D_{\vec{m}}}\abs{f(w)}^p\,dv(w)\right)^{\frac{1}{p}}\left(\int_{D_{\vec{m}}}\,dv(w)\right)^{\frac{1}{q}}\prod_{l=1}^n \frac{1}{\abs{1-\overline{w}_{m_l}z_l}^2}\notag\\
 & \lesssim & \varrho \left(\int_{D(w_{\vec{m}},\varrho)}\abs{f(w)}^p\,dv(w)\right)^{\frac{1}{p}}\prod_{l=1}^n \frac{(1-\abs{w_{m_l}}^2)^{\frac{2}{q}}}{\abs{1-\overline{w}_{m_l}z_l}^2}\label{Estimate1}.
\end{eqnarray}
Here, for the above estimates, for the third inequality we used the fact that
$$
\abs{\prod_{l=1}^n \frac{1}{\left(1-\overline{w}_lz_l\right)^2}-\prod_{l=1}^n \frac{1}{(1-\overline{w}_{m_l}z_l)^2}}\lesssim \rho\left(w,w_{\vec{m}}\right)\prod_{l=1}^n \frac{1}{\abs{1-\overline{w}_{m_l}z_l}^2}
$$
and for the fourth inequality we used that $\rho\left(w,w_{\vec{m}}\right)\leq \varrho$ when $w\in D_{\vec{m}}$, and for the last inequality we used (i) from Lemma \ref{StandardGeo_Polydisc}.  All other estimates in this string of inequalities are obvious, and the implied constants at each step depends only on the dimension.  We will obtain a similar estimate for the second term.  
From the definition of $II_{\vec{m}}$, we have
\begin{eqnarray}
\abs{II_{\vec{m}}} & \leq & \prod_{l=1}^n\frac{1}{\abs{1-\overline{w}_{m_l}z_l}^2}\int_{D_{\vec{m}}} \abs{f(w)-f(w_{\vec{m}})}\,dv(w)\notag \\
 & \lesssim & \varrho\prod_{l=1}^n\frac{1}{\abs{1-\overline{w}_{m_l}z_l}^2}\int_{D(w_{\vec{m}},\varrho)}\abs{f(w)}\,dv(w)\notag \\
 & \lesssim & \varrho \left(\,\int_{D(w_{\vec{m}},\varrho)}\abs{f(w)}^p\,dv(w)\right)^{\frac{1}{p}}\prod_{l=1}^n \frac{(1-\abs{w_{m_l}}^2)^{\frac{2}{q}}}{\abs{1-\overline{w}_{m_l}z_l}^2}\label{Estimate2}.
\end{eqnarray}
Here, the first and third inequalities are obvious, and the second inequality uses the mean value inequality, an obvious estimate, and the sub-mean value inequality for holomorphic functions.  Combining estimates from \eqref{Estimate1} and \eqref{Estimate2}, we have
\begin{eqnarray*}
\abs{f(z)-T_{\mu_\varrho}f(z)} & \leq & \sum_{\vec{m}}\abs{I_{\vec{m}}}+\abs{II_{\vec{m}}} \lesssim  \varrho \sum_{\vec{m}}\left(\,\int_{D(w_{\vec{m}},\varrho)}\abs{f(w)}^p\,dv(w)\right)^{\frac{1}{p}}\prod_{l=1}^n \frac{(1-\abs{w_{m_l}}^2)^{\frac{2}{q}}}{\abs{1-\overline{w}_{m_l}z_l}^2}.
\end{eqnarray*}
Now, we claim
\begin{equation}
\label{Maybe}
\norm{\sum_{\vec{m}}\left(\,\int_{D(w_{\vec{m}},\varrho)}\abs{f(w)}^p\,dv(w)\right)^{\frac{1}{p}}\prod_{l=1}^n \frac{(1-\abs{w_{m_l}}^2)^{\frac{2}{q}}}{\abs{1-\overline{w}_{m_l}z_l}^2}}_{L^p}^p\lesssim \sum_{\vec{m}}\int_{D(w_{\vec{m}},\varrho)}\abs{f(w)}^p\,dv(w)
\end{equation}
with implied constant depending on the dimension and $p$.  Assuming \eqref{Maybe} we have
\begin{eqnarray*}
\norm{f-T_{\mu_\varrho}f}_{A^p}^p & \lesssim & \varrho^p \norm{\sum_{\vec{m}}\left(\,\int_{D(w_{\vec{m}},\varrho)}\abs{f(w)}^p\,dv(w)\right)^{\frac{1}{p}}\prod_{l=1}^n \frac{(1-\abs{w_{m_l}}^2)^{\frac{2}{q}}}{\abs{1-\overline{w}_{m_l}z_l}^2}}_{L^p}^p\\
& \lesssim & \varrho^p \sum_{\vec{m}}\int_{D(w_{\vec{m}},\varrho)}\abs{f(w)}^p\,dv(w)\\
& \lesssim & \varrho^p\norm{f}_{A^p}^p.
\end{eqnarray*}
Here we have used that the sets $\left\{D\left(w_{\vec{m}},\varrho\right):\vec{m}\in\N^n\right\}$ are essentially disjoint, and we pick up again an implied constant depending on the dimension.  Combining things, we see
$$
\norm{f-T_{\mu_\varrho}f}_{A^p}^p\lesssim \varrho^p\norm{f}_{A^p}^p,
$$
with implied constant depending on the dimension and $p$, which is \eqref{Key_Est}.  We now turn to proving \eqref{Maybe}.  First, note that
$$
\abs{D_{\vec{m}}}\approx\prod_{l=1}^n(1-\abs{w_{m_l}}^2)^2.
$$
Define
$$
H(z):=\sum_{\vec{m}} \left(\,\int_{D\left(w_{\vec{m}},\varrho\right)}\abs{f(w)}^p\,dv(w)\right)^{\frac{1}{p}}\abs{D_{\vec{m}}}^{-\frac{1}{p}}1_{D_{\vec{m}}}(z),
$$
and by Lemma \ref{StandardGeo_Polydisc} we have
$$
\norm{H}_{L^p}^p=\sum_{\vec{m}} \int_{D\left(w_{\vec{m}},\varrho\right)}\abs{f(w)}^p\,dv(w).
$$
Let $T:L^p\to L^p$ be the operator given by
$$
T f(z)=\int_{\D^n} f(w) \prod_{l=1}^n\frac{1}{\abs{1-\overline{w}_lz_l}^2}\,dv(w)
$$
and recall that $T$ is a bounded operator on $L^p$ with norm depending only on $p$ and the dimension.  Now, we clam that
\begin{equation}
\label{Maybe2}
\sum_{\vec{m}}\left(\,\int_{D(w_{\vec{m}},\varrho)}\abs{f(w)}^p\,dv(w)\right)^{\frac{1}{p}}\prod_{l=1}^n \frac{(1-\abs{w_{m_l}}^2)^{\frac{2}{q}}}{\abs{1-\overline{w}_{m_l}z_l}^2}\approx TH(z),
\end{equation}
which would imply
\begin{eqnarray*}
\norm{\sum_{\vec{m}}\left(\,\int_{D(w_{\vec{m}},\varrho)}\abs{f(w)}^p\,dv(w)\right)^{\frac{1}{p}}\prod_{l=1}^n \frac{(1-\abs{w_{m_l}}^2)^{\frac{2}{q}}}{\abs{1-\overline{w}_{m_l}z_l}^2}}_{L^p}^p & \approx & \norm{TH}_{L^p}^p\\
& \lesssim & \norm{H}_{L^p}^p\\
 & = & \sum_{\vec{m}}\int_{D(w_{\vec{m}},\varrho)}\abs{f(w)}^p\,dv(w),
\end{eqnarray*}
proving \eqref{Maybe}.  To see that \eqref{Maybe2} holds, observe that
\begin{eqnarray*}
TH(z) & = & \sum_{\vec{m}} \left(\,\int_{D\left(w_{\vec{m}},\varrho\right)}\abs{f(w)}^p\,dv(w)\right)^{\frac{1}{p}}\abs{D_{\vec{m}}}^{-\frac{1}{p}}\int_{D_{\vec{m}}} \prod_{l=1}^n\frac{1}{\abs{1-\overline{w}_lz_l}^2}\,dv(w)\\
 & \approx & \sum_{\vec{m}} \left(\,\int_{D\left(w_{\vec{m}},\varrho\right)}\abs{f(w)}^p\,dv(w)\right)^{\frac{1}{p}}\abs{D_{\vec{m}}}^{-\frac{1}{p}}\int_{D_{\vec{m}}} \prod_{l=1}^n\frac{1}{\abs{1-\overline{w}_{m_l}z_l}^2}\,dv(w)\\
 & \approx & \sum_{\vec{m}} \left(\,\int_{D\left(w_{\vec{m}},\varrho\right)}\abs{f(w)}^p\,dv(w)\right)^{\frac{1}{p}}\abs{D_{\vec{m}}}^{\frac{1}{q}} \prod_{l=1}^n\frac{1}{\abs{1-\overline{w}_{m_l}z_l}^2}\\
 & \approx & \sum_{\vec{m}} \left(\,\int_{D\left(w_{\vec{m}},\varrho\right)}\abs{f(w)}^p\,dv(w)\right)^{\frac{1}{p}} \prod_{l=1}^n\frac{(1-\abs{w_{m_l}}^2)^\frac{2}{q}}{\abs{1-\overline{w}_{m_l}z_l}^2}.
\end{eqnarray*}
Here, we again used Lemma \ref{StandardGeo_Polydisc}, in particular Remark \ref{Useful2}, in the first approximate equality.
\end{proof}

Now choose $0<\varrho\leq 1$ so that $\norm{T_{\mu_\varrho}-I_{A^p}}_{\mathcal{L}\left(A^p,A^p\right)}<\frac{1}{4}$.  We then have that $\norm{T_{\mu_\varrho}}_{\mathcal{L}\left(A^p, A^p\right)}$ and $\norm{T_{\mu_\varrho}^{-1}}_{\mathcal{L}\left(A^p, A^p\right)}$ are less than $\frac{4}{3}$.  Fix this value of $\varrho$, and denote $\mu_\varrho:=\mu$ for the rest of the paper.

For $S\in\mathcal{L}\left(A^p, A^p\right)$ and $r>0$, let
$$
\mathfrak{a}_S(r):=\varlimsup_{z\to\partial\D^n}\sup\left\{\norm{Sf}_{A^p}: f\in T_{\mu 1_{D\left(z,r\right)}}(A^p),\norm{f}_{A^p}\leq 1\right\}.
$$
Then define
$$
\mathfrak{a}_S:=\lim_{r\to 1}\mathfrak{a}_S(r).
$$
Since for $r_1<r_2$ we have that $T_{\mu 1_{D(z,r_1)}}(A^p)\subset T_{\mu 1_{D(z,r_2)}}(A^p)$ and $\mathfrak{a}_S(r)\leq\norm{S}_{\mathcal{L}\left(A^p,A^p\right)}$ this limit is well defined.  We define two other measures of the size of an operator, which are given in a very intrinsic and geometric way:
\begin{eqnarray*}
\mathfrak{b}_{S} & := & \sup_{r>0} \varlimsup_{z\to\partial\D^n}\norm{M_{1_{D(z,r)}}S} _{\mathcal{L}\left(A^p,L^p\right)},\\
\mathfrak{c}_{S} & := & \lim_{r\to 1}\norm{M_{1_{(r\D^n)^c}}S}_{\mathcal{L}\left(A^p,L^p\right)}.
\end{eqnarray*}
Recall that for notational simplicity we are letting $\left(r\D^n\right)^c=\D^n\setminus r\D^n$.  Finally, for $S\in \mathcal{L}\left(A^p,A^p\right)$ recall that
$$
\norm{S}_e=\inf\left\{\norm{S-Q}_{\mathcal{L}\left(A^p,A^p\right)}:Q\textnormal{ is compact}\right\}
$$
is the essential norm of the operator $S$.  We first show how to compute the essential norm of an operator $S$ in terms of the operators $S_x$ where $x\in M_{\mathcal{A}}\setminus\D^n$.

\begin{thm}
\label{EssentialviaSx}
Let $S\in\mathcal{T}_{p}$.  Then there exists constants depending on $p$ and the dimension such that
\begin{equation}
\label{EssentialNorm&Sx}
\sup_{x\in M_{\mathcal{A}}\setminus\D^n}\norm{S_x}_{\mathcal{L}\left(A^p, A^p\right)}\lesssim\norm{S}_e\lesssim \sup_{x\in M_{\mathcal{A}}\setminus\D^n}\norm{S_x}_{\mathcal{L}\left(A^p, A^p\right)}.
\end{equation}
\end{thm}
\begin{proof}
Note that if $S$ is compact, then \eqref{EssentialNorm&Sx} is easy.  Since $k_\xi^{(p)}\to 0$ weakly as $\xi\to\partial\D^n$, for $S$ compact we have $\norm{Sk_\xi^{(p)}}_{A^p}\to 0$ as $\xi\to\partial\D^n$.  This in turn implies that the Berezin transform vanishes as we go to the boundary since
\begin{equation}
\label{Vanishing}
\abs{B(S)(\xi)}=\abs{\ip{S k_\xi^{(p)}}{k_\xi^{(q)}}_{A^2}}\leq\norm{Sk_\xi^{(p)}}_{A^p}\norm{k_\xi^{(q)}}_{A^q}\approx\norm{Sk_\xi^{(p)}}_{A^p}.
\end{equation}
Then, using Lemma \ref{BerezinVanish}, we have that $S_x=0$ for all $x\in M_{\mathcal{A}}\setminus\D^n$.

Now let $S$ be any bounded operator on $A^p$.  Suppose that $Q$ is a compact operator on $A^p$.  Select $x\in M_{\mathcal{A}}\setminus \D^n$ and a corresponding net $\{z^{\omega}\}$ tending to $x$.  Since the maps $U_{z^{\omega}}^{p}$ and $U_{z^{\omega}}^{q}$ are isometries on $A^p$ and $A^q$ we have that
$$
\norm{S_{z^{\omega}}+Q_{z^{\omega}}}_{\mathcal{L}\left(A^p, A^p\right)}\leq\norm{S+Q}_{\mathcal{L}\left(A^p, A^p\right)}.
$$
Since $S_{z^{\omega}}+Q_{z^{\omega}}\to S_x$ in the $WOT$, and passing to the limits we have
$$
\norm{S_x}_{\mathcal{L}\left(A^p, A^p\right)}\lesssim \varliminf\norm{S_{z^{\omega}}+Q_{z^{\omega}}}_{\mathcal{L}\left(A^p, A^p\right)}\leq \norm{S+Q}_{\mathcal{L}\left(A^p, A^p\right)}.
$$
But, this gives that
$$
\sup_{x\in M_{\mathcal{A}}\setminus\D^n}\norm{S_x}_{\mathcal{L}\left(A^p, A^p\right)}\lesssim\norm{S}_e
$$
which is the first inequality in \eqref{EssentialNorm&Sx}.  It only remains to address the last inequality and to accomplish this, we will instead prove 
\begin{equation}
\label{AlphaControlled}
\mathfrak{a}_S\lesssim \sup_{x\in M_{\mathcal{A}}\setminus\D^n}\norm{S_x}_{\mathcal{L}\left(A^p, A^p\right)}.
\end{equation}
Then we compare this with the first inequality in \eqref{Redux2}, $\norm{S}_e\lesssim\mathfrak{a}_S$, (shown below) to obtain
$$
\norm{S}_e\lesssim \sup_{x\in M_{\mathcal{A}}\setminus\D^n}\norm{S_x}_{\mathcal{L}\left(A^p, A^p\right)}.
$$
A remark that will be important later is that if we have \eqref{AlphaControlled}, then we also have
\begin{equation}
\label{LastStep}
\mathfrak{a}_S\lesssim\norm{S}_e.
\end{equation}

We now turn to addressing \eqref{AlphaControlled}.  It suffices to demonstrate that
$$
\mathfrak{a}_S(r)\lesssim\sup_{x\in M_{\mathcal{A}}\setminus\D^n}\norm{S_x}_{\mathcal{L}\left(A^p, A^p\right)}\quad\forall r>0.
$$
Fix a radius $r>0$, and then using the definition of $\mathfrak{a}_S(r)$ we have a sequence $\{z^{j}\}\subset\D^n$ tending to $\partial\D^n$ and a normalized sequence of functions $f_j\in T_{\mu 1_{D\left(z^j,r\right)}}(A^p)$ with $\norm{Sf_j}_{A^p}\to\mathfrak{a}_S(r)$.  To each $f_j$ we have a corresponding $h_j\in A^p$, and so
\begin{eqnarray*}
f_j(w)=T_{\mu 1_{D\left(z^j,r\right)}}h_j(w) & = & \sum_{w_{\vec{m}}\in D(z^j,r)}\frac{v(D_m)}{\prod_{l=1}^n(1-\overline{w_{m_l}}w_l)^{2}}h_j(w_{\vec{m}})\\
 & = & \sum_{w_{\vec{m}}\in D\left(z^j,r\right)} a_{j,\vec{m}} \prod_{l=1}^n\frac{\left(1-\abs{w_{m_l}}^2\right)^{\frac{2}{q}}}{\left(1-\overline{w_{m_l}}w\right)^{2}}\\
 & = & \sum_{w_{\vec{m}}\in D\left(z^j,r\right)} a_{j,\vec{m}} k_{w_{\vec{m}}}^{(p)}(w),
\end{eqnarray*}
where $a_{j,\vec{m}}=v(D_{\vec{m}})\prod_{l=1}^n\left(1-\abs{w_{m_l}}^2\right)^{-\frac{2}{q}}h_j(w_{\vec{m}})$.  We then have that
$$
\left(U^{q}_{z^j}\right)^{*}f_j(w)= \sum_{\varphi_{z^j}(w_{\vec{m}})\in D(0,r)} a_{j,\vec{m}}' k_{\varphi_{z^j}(w_{\vec{m}})}^{(p)}(w)
$$
where $a_{j,m}'$ is simply the original constant $a_{j,m}$ multiplied by the unimodular constant $\lambda_{(q)}$ from \eqref{Comp}.

Observe that the points $\abs{\varphi_{z^j}(w_{\vec{m}})}\leq r$.  %For $j$ fixed, arrange the points $\varphi_{z_j}(w_m)$ such that $\abs{\varphi_{z_j}(w_m)}\leq\abs{\varphi_{z_j}(w_{m+1})}$ and $\textnormal{arg}\,\varphi_{z_j}(w_m)\leq\textnormal{arg}\, \varphi_{z_j}(w_{m+1})$.  
Since the M\"obius map $\varphi_{z^j}$ preserves the hyperbolic distance between the points $\{w_{\vec{m}}\}$ we have that when $\vec{m}\neq \vec{k}$ that
$$
\rho\left(\varphi_{z^j}(w_{\vec{m}}\right),\varphi_{z^j}(w_{\vec{k}}))=\rho\left(w_{\vec{m}},w_{\vec{k}}\right)\geq\frac{\varrho}{4}>0.
$$
By volume considerations there can only be at most $N_j\leq M(\varrho, r)$ points in the collection $\varphi_{z^j}(w_{\vec{m}})$ that lie in the disc $D(0,r)$.  This follows since we are looking in a compact set, $r\D^n$, and the points $w_{\vec{m}}$ are at a fixed distance from each other.  By passing to a subsequence, we can assume that $N_j=M$ and is independent of $j$.

For the fixed $j$ and $\vec{m}$, and select $g_{j,\vec{m}}\in H^\infty$ with $\norm{g_{j,\vec{m}}}_{H^\infty}\leq C\left(r,\varrho\right)$ and $g_{j,\vec{m}}(\varphi_{z^j}(w_{\vec{k}}))=\delta_{\vec{k},\vec{m}}$, the Dirac delta.  The existence of such a function is easy to deduce from a result of Berndtsson, Chang and Lin \cite{BCL}.  Indeed, from the main result of \cite{BCL} one can see that for a collection of points $\{z^j\}$ such that
$$
\inf_{j}\prod_{j\neq k}\rho\left(z^j,z^k\right)\geq \tau>0,
$$
then for any $k$, there exists a function $g_k\in H^\infty(\D^n)$ such that $\norm{g_k}_{H^\infty}\leq C(\tau)$ and $g_k(z_j)=\delta_{j,k}$. The main result of \cite{BCL}*{Theorem 1} actually says more, but what appears above suffices for our purposes.  We then have that
\begin{eqnarray*}
\ip{\left(U_{z^j}^{q}\right)^{*}f_j}{g_{j,\vec{k}}}_{A^2} & = & \sum_{\varphi_{z^j}(w_{\vec{m}})\in D(0,r)} a_{j,\vec{m}}' g_{j,\vec{k}}(\varphi_{z^j}(w_{\vec{m}}))\prod_{l=1}^n\left(1-\abs{\varphi_{z^j_l}(w_{m_l})}^2\right)^{\frac{2}{q}}\\
 & = & a_{j,\vec{k}}'\prod_{l=1}^n\left(1-\abs{\varphi_{z^j_l}(w_{k_l})}^2\right)^{\frac{2}{q}}.
\end{eqnarray*}
This expression then implies that $\abs{a_{j,\vec{k}}'}\leq C=C(n,p,\varrho,r)$ independent of $j$ and $\vec{k}$.  This follows since $g_{j,\vec{k}}\in H^\infty$ with norm controlled by $C(r,\varrho)$, $\left(U^{q}_z\right)^*$ is bounded, $\abs{\varphi_{z^j_l}(w_{k_l})}\leq r$, and $f_j$ is a normalized sequence.

Relabel the points in the collection $\varphi_{z^j}(w_{\vec{m}})$ that lie in the disc $D(0,r)$ to be $\varphi_{z^j}(w_i)$ where $1\leq i\leq M$.  Additionally, relabel the corresponding elements $a_{j,\vec{m}}'$ as $a_{j,i}'$.  Now we then have that $(\varphi_{z^j}(w_1),\ldots,\varphi_{z^j}(w_M), a_{j,1}',\ldots,a_{j,M}')\in\C^{M(n+1)}$ is a bounded sequence in $j$.  Passing to a subsequence, we have convergence to $(v_1,\ldots, v_M,a_{1}',\ldots a_{M}')$.  Here $\abs{v_i}\leq r$ and $\abs{a_{i}'}\leq C$ for $i=1,\ldots, M$.  This then gives that
$$
\left(U^{q}_{z^j}\right)^{*}f_j\to \sum_{k=1}^M a_k' k_{v_k}^{(p)} := h
$$
in the $L^p$ norm.  Moreover,
$$
\norm{h}_{L^p}=\norm{\sum_{k=1}^M a_k' k_{v_k}^{(p)}}_{L^p}=\lim_{j}\norm{\left(U^{q}_{z^j}\right)^{*}f_j}_{L^p}\lesssim1.
$$
Since the operator $U_{z^j}^{p}$ is isometric, and $\norm{S_{z^j}}$ is bounded independent of $j$, we have that
$$
\mathfrak{a}_S(r)=\lim_{j}\norm{S f_j}_{A^p}=\lim_{j}\norm{S_{z^j}\left(U_{z^j}^{q}\right)^* f_j}_{A^p}=\lim_j\norm{S_{z^j} h}_{A^p}.
$$
Recall that we have a sequence $\{z^j\}$ such that $z^j\to\partial\D^n$.  Now use the compactness of $M_{\mathcal{A}}$ and extract a subnet $\{z^\omega\}$ converging to some point $x\in M_{\mathcal{A}}\setminus\D^n$.  Then we have that $S_{z^\omega}h\to S_x h$ in $A^p$, and so
$$
\mathfrak{a}_S(r)=\lim_{\omega}\norm{S_{z^\omega} h}_{A^p}=\norm{S_xh}_{A^p}\lesssim\norm{S_x}_{\mathcal{L}\left(A^p, A^p\right)}\lesssim\sup_{x\in M_{\mathcal{A}}\setminus\D^n}\norm{S_x}_{\mathcal{L}\left(A^p, A^p\right)}.
$$
In the above, we used the continuity in the $SOT$ as guaranteed by Proposition \ref{SOTCon}.
\end{proof}

Our next main result is the following Theorem.
\begin{thm}
Let $1<p<\infty$ and $S\in \mathcal{T}_{p}$.  Then there exists constants depending only on $n$ and $p$ such that:
$$
\mathfrak{a}_S\approx\mathfrak{b}_S\approx\mathfrak{c}_S\approx\norm{S}_e.
$$
\end{thm}

\begin{proof}
By Lemma \ref{Approx3} there are Borel sets $F_{\vec{j}}\subset G_{\vec{j}}\subset\D^n$ such that
\begin{itemize}
\item[(i)] $\D^n=\bigcup F_{\vec{j}}$;
\item[(ii)] $F_{\vec{j}}\bigcap F_{\vec{j}'}=\emptyset$ if $\vec{j}\neq\vec{j}'$;
\item[(iii)] each point of $\D^n$ lies in no more than $N(n)$ of the sets $G_{\vec{j}}$;
\item[(iv)] $\textnormal{diam}_{\rho}\, G_{\vec{j}}\leq d(p,S,\epsilon)$
\end{itemize}
and
\begin{equation}
\label{RecallEst}
\norm{ST_\mu-\sum_{\vec{j}} M_{1_{F_{\vec{j}}}}ST_{\mu 1_{G_{\vec{j}}}}}_{\mathcal{L}\left(A^p, L^p\right)}<\epsilon.
\end{equation}
For $m\in\N$, set
$$
S_m:=\sum_{\abs{\vec{j}}\geq m} M_{1_{F_{\vec{j}}}} S T_{\mu 1_{G_{\vec{j}}}}.
$$
We then form one more measure of the size of $S$, and this is given by
$$
\varlimsup_{m\to\infty}\norm{S_m}_{\mathcal{L}\left(A^p, L^p\right)}=\varlimsup_{m\to\infty} \norm{\sum_{\abs{\vec{j}}\geq m} M_{1_{F_{\vec{j}}}} S T_{\mu 1_{G_{\vec{j}}}}}_{\mathcal{L}\left(A^p, L^p\right)}.
$$
First some observations.  Since every $z\in\D^n$ belongs to only $N(n)$ sets $G_{\vec{j}}$ we have by Lemma~\ref{CM-Cor} that
$$
\sum_{\abs{\vec{j}}\geq m}\norm{T_{1_{G_{\vec{j}}\mu}}f}_{A^p}^p\lesssim \sum_{\vec{j}} \norm{1_{G_{\vec{j}}}f}_{L^p(\mu)}^p \lesssim\norm{f}_{A^p}^p.
$$
Also, since $T_\mu$ is bounded and invertible, we have that $\norm{S}_e\approx\norm{ST_\mu}_e$.  Finally, we will need to compute both norms in $\mathcal{L}\left(A^p, A^p\right)$ and $\mathcal{L}\left(A^p, L^p\right)$.  When necessary, we will denote the respective essential norms as $\norm{\,\cdot\,}_{e}$ and $\norm{\,\cdot\,}_{ex}$.  However, we always have
$$
\norm{R}_{ex}\leq\norm{R}_e\leq\norm{P}_{L^p\to A^p}\norm{R}_{ex}\lesssim\norm{R}_{ex}.
$$
The strategy is to show that
\begin{equation}
\label{Redux1}
\mathfrak{b}_S\leq\mathfrak{c}_S\lesssim \varlimsup_{m\to\infty}\norm{S_m}_{\mathcal{L}\left(A^p, L^p\right)}\lesssim\mathfrak{b}_S
\end{equation}
and 
\begin{equation}
\label{Redux2}
\norm{S}_{e}\lesssim\varlimsup_{m\to\infty}\norm{S_m}_{\mathcal{L}\left(A^p, L^p\right)}\lesssim \mathfrak{a}_S\lesssim \norm{S}_e.
\end{equation}
Combining \eqref{Redux1} and \eqref{Redux2} we have the Theorem.  We first turn to proving the first two inequalities in \eqref{Redux2}.

Fix a $f\in A^p$ of norm 1 and note
\begin{eqnarray}
\norm{S_m f}^p_{L^p} & = & \sum_{\abs{\vec{j}}\geq m} \norm{M_{1_{F_{\vec{j}}}} ST_{\mu 1_{G_{\vec{j}}}}f}_{L^p}^p\notag\\
& = & \sum_{\abs{\vec{j}}\geq m} \left(\frac{\norm{M_{1_{F_{\vec{j}}}} ST_{\mu 1_{G_{\vec{j}}}}f}_{L^p}}{\norm{T_{1_{G_{\vec{j}}\mu}}f}_{A^p}}\right)^{p}\norm{T_{\mu 1_{G_{\vec{j}}}}f}_{A^p}^p\notag\\
& \leq & \sup_{\abs{\vec{j}}\geq m}\sup\left\{\norm{M_{1_{F_{\vec{j}}}} Sg}^p_{L^p}: g\in T_{\mu 1_{G_{\vec{j}}}}(A^p),\norm{g}_{A^p}= 1\right\}\sum_{\abs{\vec{j}}\geq m}\norm{T_{\mu 1_{G_{\vec{j}}}}f}_{A^p}^p\notag\\
& \lesssim & \sup_{\abs{\vec{j}}\geq m}\sup\left\{\norm{M_{1_{F_{\vec{j}}}} Sg}^p_{L^p}: g\in T_{\mu 1_{G_{\vec{j}}}}(A^p),\norm{g}_{A^p}= 1\right\}\label{Last}.
\end{eqnarray}
Now observe that since $\textnormal{diam}_{\rho}\, G_{\vec{j}}\leq d$, then select $z^{\vec{j}}\in G_{\vec{j}}$ and we have that
$G_{\vec{j}}\subset D\left(z^{\vec{j}},d\right)$, and so $T_{\mu 1_{G_{\vec{j}}}}(A^p)\subset T_{\mu 1_{D\left(z^{\vec{j}},d\right)}}(A^p)$.
Since $z^{\vec{j}}$ must approach the boundary, we can select an additional sequence $0<\gamma_m<1$ tending to $1$ such that
$\rho\left(0,z^{\vec{j}}\right)\geq\gamma_m$ when $\abs{\vec{j}}\geq m$.  Using \eqref{Last} we have that
\begin{eqnarray}
\norm{S_m}_{\mathcal{L}\left(A^p,L^p\right)} & \lesssim & \sup_{\abs{\vec{j}}\geq m}\sup\left\{\norm{M_{1_{F_{\vec{j}}}} Sg}_{L^p}: g\in T_{\mu 1_{G_{\vec{j}}}}(A^p),\norm{g}_{A^p}= 1\right\}\notag\\
 & \lesssim & \sup_{\rho\left(0,z^{\vec{j}}\right)\geq\gamma_m}\sup\left\{\norm{M_{1_{D\left(z^{\vec{j}},d\right)}} Sg}_{L^p}: g\in T_{\mu 1_{D\left(z^{\vec{j}},d\right)} }(A^p),\norm{g}_{A^p}= 1\right\}\label{important}\\
 & \lesssim & \sup_{\rho\left(0,z^{\vec{j}}\right)\geq\gamma_m}\sup\left\{\norm{Sg}_{L^p}: g\in T_{\mu 1_{D\left(z^{\vec{j}},d\right)}}(A^p),\norm{g}_{A^p}= 1\right\}\notag.
\end{eqnarray}
Then send $m\to \infty$ and note that since $\gamma_m\to 1$ we have that
$$
\varlimsup_{m\to\infty} \norm{S_m}_{\mathcal{L}\left(A^p,L^p\right)}\lesssim \mathfrak{a}_S(d).
$$
Now using \eqref{RecallEst} we see
$$
\norm{ST_\mu}_{ex}\leq\varlimsup_{m\to\infty} \norm{S_m}_{\mathcal{L}\left(A^p,L^p\right)}+\epsilon\lesssim \mathfrak{a}_S(d)+\epsilon\lesssim\mathfrak{a}_S+\epsilon.
$$
and so $\norm{ST_\mu}_{ex}\leq\varlimsup_{m\to\infty} \norm{S_m}_{\mathcal{L}\left(A^p,L^p\right)}\lesssim\mathfrak{a}_S$ since $\epsilon$ is arbitrary.  This in turn implies, 
\begin{equation}
\norm{S}_e\approx\norm{ST_\mu}_{e}\lesssim \norm{ST_\mu}_{ex}\leq \varlimsup_{m} \norm{S_m}_{\mathcal{L}\left(A^p,L^p\right)}\lesssim \mathfrak{a}_S.
\end{equation}
This then gives the first two inequalities in \eqref{Redux2}.  The remaining inequality is simply \eqref{LastStep} which was proved in Theorem \ref{EssentialviaSx}.

We now consider \eqref{Redux1}.  Let $0<r<1$, and note that there exists a positive integer $m(r)$
such that $\bigcup_{{\vec{j}}<m(r)}F_{\vec{j}}\subset r\D^n$.  We then have
\begin{eqnarray*}
\norm{M_{1_{(r\D^n)^c}}S}_{\mathcal{L}\left(A^p,L^p\right)}\norm{T^{-1}_\mu}^{-1}_{\mathcal{L}\left(A^p,A^p\right)} & \leq & \norm{M_{1_{(r\D^n)^c}}ST_\mu}_{\mathcal{L}\left(A^p,L^p\right)}\\
& \leq & \norm{M_{1_{(r\D^n)^c}}\left(ST_\mu-\sum_{{\vec{j}}} M_{1_{F_{\vec{j}}}}ST_{\mu 1_{G_{\vec{j}}}}\right)}_{\mathcal{L}\left(A^p,L^p\right)}\\
& & +\norm{M_{1_{(r\D^n)^c}}\sum_{{\vec{j}}} M_{1_{F_{\vec{j}}}}ST_{\mu 1_{G_{\vec{j}}}}}_{\mathcal{L}\left(A^p,L^p\right)}\\
& \leq & \epsilon +\norm{\sum_{\abs{{\vec{j}}}\geq m(r)} M_{1_{F_{\vec{j}}}}ST_{\mu 1_{G_{\vec{j}}}}}_{\mathcal{L}\left(A^p,L^p\right)}\\
& = & \epsilon+\norm{S_{m(r)}}_{\mathcal{L}\left(A^p,L^p\right)}.
\end{eqnarray*}
This string of inequalities easily gives
\begin{equation}
\label{Redux1-1}
\mathfrak{c}_{S}=\varlimsup_{r\to 1}\norm{M_{1_{(r\D^n)^c}}S}_{\mathcal{L}\left(A^p,L^p\right)}\lesssim \varlimsup_{m\to\infty} \norm{S_m}_{\mathcal{L}\left(A^p,L^p\right)}.
\end{equation}
But, \eqref{important} gives that
\begin{equation}
\label{Redux1-2}
\varlimsup_{m\to\infty} \norm{S_m}_{\mathcal{L}\left(A^p,L^p\right)}\lesssim\varlimsup_{z\to \partial\D^n}
\norm{M_{1_{D\left(z,r\right)}}S} _{\mathcal{L}\left(A^p,L^p\right)}\lesssim\mathfrak{b}_S.
\end{equation}
Combining the trivial inequality $\mathfrak{b}_S\leq\mathfrak{c}_S$ with \eqref{Redux1-1} and \eqref{Redux1-2} we obtain \eqref{Redux1}.
\end{proof}

From these Theorems we can deduce two results of interest.
\begin{thm}
Let $1<p<\infty$ and $S\in\mathcal{T}_{p}$.  Then
$$
\norm{S}_e\approx \sup_{\norm{f}_{A^p}=1}\varlimsup_{z\to\partial\D^n}\norm{S_z f}_{A^p}.
$$
\end{thm}
\begin{proof}
It is easy to see from Lemma \ref{SOTCon} and the compactness of $M_{\mathcal{A}}$ that
$$
\sup_{x\in M_{\mathcal{A}}\setminus\D^n}\norm{S_xf}_{A^p}=\varlimsup_{z\to\partial\D^n}\norm{ S_z f}_{A^p}.
$$
But, then we get,
$$
\sup_{x\in M_{\mathcal{A}}\setminus\D^n}\norm{S_x}_{\mathcal{L}\left(A^p,A^p\right)}=\sup_{\norm{f}_{A^p}=1}\varlimsup_{z\to\partial\D^n}\norm{S_z f}_{A^p},
$$
and using Theorem \ref{EssentialviaSx} gives the result.
\end{proof}

The next result gives the characterization of compact operators in terms of the Berezin transform and membership in the Toeplitz algebra

\begin{thm}
Let $1<p<\infty$ and $S\in\mathcal{L}\left(A^p,A^p\right)$.  Then $S$ is compact if and only if $S\in\mathcal{T}_{p}$ and $B(S)=0$ on $\partial\D^n$.
\end{thm}
\begin{proof}
If $B(S)=0$ on $\partial\D^n$, then we have that $S_x=0$ for all $x\in M_{\mathcal{A}}\setminus\D^n$.  And if $S\in \mathcal{T}_{p}$, then Theorem \ref{EssentialviaSx} gives that $S$ must be compact.

In the other direction, if $S$ is compact, then we have that $B(S)=0$ on $\partial\D^n$ by \eqref{Vanishing}.  So it only remains to show that $S\in\mathcal{T}_{p}$.  Since every compact operator on $A^p$ can be approximated by finite rank operators, it suffices to show that all rank one operators are in $\mathcal{T}_{p}$.  But, the rank one operators have the form $f\otimes g$ where $f\in A^p$, $g\in A^q$ and
$f\otimes g:A^p\to A^p$ is given by
$$
f\otimes g(h)=\ip{h}{g}_{A^2}f.
$$
We can further suppose that $f$ and $g$ are polynomials since they are dense in $A^p$ and $A^q$ respectively.  But, then
$$
f\otimes g= T_f(1\otimes 1)T_{\overline{g}},
$$
and so it suffices to know that $1\otimes 1\in \mathcal{T}_{p}$.  But, $1\otimes 1(h)=h(0)=T_{\delta_0}h$, where $\delta_0$ is the Dirac mass at zero and it is easy to see by Theorem \ref{Sua2Thm} that $T_{\delta_0}$ belongs to the algebra generated by $\left\{T_a:a\in\mathcal{A}\right\}$, and so $T_{\delta_0}\in\mathcal{T}_{p}$.

\end{proof}

\subsection{The Hilbert Space Case}

When $p=2$, then some of the results can be strengthened in a straightforward manner.  It is easy to see that when $T\in\mathcal{L}\left(A^2,A^2\right)$, $S\in\mathcal{T}_{2}$ and for $x\in M_{\mathcal{A}}$ we have
$$
(ST)_x=S_xT_x\quad (TS)_x=T_xS_x\quad (T^*)_x=T_x^*.
$$
This is accomplished by noting that $\mathcal{T}_{2}$ is a self-adjoint algebra, and using Propositions \ref{WOTCon} and \ref{SOTCon} and that $b_z=1$ by \eqref{bofz}.  Also note that if $S\in\mathcal{L}\left(A^2,A^2\right)$ that we have $\norm{S_z}_{\mathcal{L}\left(A^2,A^2\right)}=\norm{S}_{\mathcal{L}\left(A^2,A^2\right)}$, and so
$$
\norm{S_x}_{\mathcal{L}\left(A^2,A^2\right)}\leq\norm{S}_{\mathcal{L}\left(A^2,A^2\right)}.
$$
Here we have used that $b_z=1$ by \eqref{bofz} when $p=2$ and the definition of $S_z$.

Let $\mathcal{K}$ denote the ideal of compact operators on $A^2$.  Recall that the Calkin algebra is given by $\mathcal{L}\left(A^2, A^2\right)/\mathcal{K}$.  The spectrum of $S$ will be denoted by $\sigma(S)$, and the spectral radius will be denoted by
$$
r(S)=\sup\left\{\abs{\lambda}:\lambda\in\sigma(S)\right\}.
$$
Define the essential spectrum $\sigma_e(S)$ to be the spectrum of $S+\mathcal{K}$ in the Calkin algebra, and the essential spectral radius as
$$
r_e(S)=\sup\left\{\abs{\lambda}:\lambda\in\sigma_e(S)\right\}.
$$
The following result is the improvement that is available in the Hilbert space case.

\begin{thm}
For $S\in\mathcal{T}_{2}$ we have
\begin{equation}
\label{EssentialNorm}
\norm{S}_e=\sup_{x\in M_{\mathcal{A}}\setminus\D^n}\norm{S_x}_{\mathcal{L}\left(A^2, A^2\right)}
\end{equation}
and
\begin{equation}
\label{EssentialRad}
\sup_{x\in M_{\mathcal{A}}\setminus\D^n} r(S_x)\leq\lim_{k\to\infty}
\left(\sup_{x\in M_{\mathcal{A}}\setminus\D^n}\norm{S_x^k}^{\frac{1}{k}}_{\mathcal{L}\left(A^2, A^2\right)}\right)=r_e(S)
\end{equation}
with equality when $S$ is essentially normal.
\end{thm}

\begin{proof}
Since we have that $\left(S^k\right)_x=\left(S_x\right)^k$, then by Theorem \ref{EssentialviaSx} we have
$$
\sup_{x\in M_{\mathcal{A}}\setminus\D^n}\norm{S_x^k}^{\frac{1}{k}}_{\mathcal{L}\left(A^2, A^2\right)}\lesssim \norm{S^k}_e^{\frac{1}{k}}\lesssim \sup_{x\in M_{\mathcal{A}}\setminus\D^n}\norm{S_x^k}^{\frac{1}{k}}_{\mathcal{L}\left(A^2, A^2\right)}.
$$
Taking the limit as $k\to\infty$ gives that
$$
\lim_{k\to\infty}\left(\sup_{x\in M_{\mathcal{A}}\setminus\D^n}\norm{S_x^k}^{\frac{1}{k}}_{\mathcal{L}\left(A^2, A^2\right)}\right)=r_e(S).
$$
While for the inequality one notes that $r(T)\leq \norm{T^k}^{\frac{1}{k}}$ for a generic operator and so we have
$$
\sup_{x\in M_{\mathcal{A}}\setminus\D^n} r(S_x)\leq \sup_{x\in M_{\mathcal{A}}\setminus\D^n}\norm{S_x^k}^{\frac{1}{k}}_{\mathcal{L}\left(A^2, A^2\right)}.
$$
Combining these observations gives \eqref{EssentialRad}.  Suppose now that $S$ is essentially normal.  This means that $S^*S-SS^*$ is compact, and so we have that
$$
S^*_xS_x-S_xS^*_x=\left(S^*S-SS^*\right)_x=0.
$$
So $S_x$ is a normal operator for each $x\in M_{\mathcal{A}}\setminus\D^n$ and consequently we have
$$
\norm{S_x^k}^{\frac{1}{k}}_{\mathcal{L}\left(A^2, A^2\right)}=r(S_x).
$$
This then gives the equality in \eqref{EssentialRad} by noting that
$$
\sup_{x\in M_\mathcal{A}\setminus\D^n}r(S_x)=\lim_{k\to\infty}\sup_{x\in M_\mathcal{A}\setminus\D^n}\norm{S_x^k}^{\frac{1}{k}}_{\mathcal{L}\left(A^2, A^2\right)}=r_e(S).
$$
Now apply the equality in \eqref{EssentialRad} to the operator $S^*S$ and note that
\begin{eqnarray*}
\norm{S}^2_e=\norm{S^*S}_e=r_e(S^*S) & = & \sup_{x\in M_\mathcal{A}\setminus\D^n}r\left(\left(S^*S\right)_x\right)\\
 & = & \sup_{x\in M_\mathcal{A}\setminus\D^n}\norm{S_x^*S_x} _{\mathcal{L}\left(A^2, A^2\right)}\\
 & = & \sup_{x\in M_\mathcal{A}\setminus\D^n}\norm{S_x}^2 _{\mathcal{L}\left(A^2, A^2\right)}.
\end{eqnarray*}
\end{proof}
We have the following Corollary, that can be proved in a similar manner as in \cite{Sua}.  For completeness, we provide the details.
\begin{cor}
\label{OperatorInequality}
Let $S\in \mathcal{T}_{2}$ and $\gamma, \delta\in\R$ be such that $\gamma I_{A^2}\leq S_x\leq \delta I_{A^2}$ for all $x\in M_{\mathcal{A}}\setminus\D^n$.  Then given $\epsilon>0$ there is a compact self-adjoint operator $K$ such that
$$
\left(\gamma-\epsilon\right)I_{A^2}\leq S+K\leq\left(\delta+\epsilon\right)I_{A^2}.
$$
\end{cor}
\begin{proof}
Since $\gamma I_{A^2}\leq S_x\leq\delta I_{A^2}$, then we have
$$
-\left(\frac{\delta-\gamma}{2}\right)I_{A^2}\leq S_x-\left(\frac{\delta+\gamma}{2}\right)I_{A^2}\leq \left(\frac{\delta-\gamma}{2}\right)I_{A^2}
$$
for all $x\in M_{\mathcal{A}}\setminus\D^n$.  By standard operator theory, the spectral radius of a self-adjoint element in a $C^*$ algebra coincides with its norm, and so applying Theorem \ref{EssentialRad} gives
$$
\norm{S-\frac{\delta+\gamma}{2}I_{A^2}}_e\leq\frac{\delta-\gamma}{2}.
$$
Thus, there is a compact operator $K$ such that for any $\epsilon>0$ we have
$$
\norm{S-\frac{\delta+\gamma}{2}I_{A^2}+K}_{\mathcal{L}\left(A^2,A^2\right)}\leq\frac{\delta-\gamma}{2}+\epsilon.
$$
Without loss of generality, we can take $K$ to be self-adjoint (simply replace $K$ by $\frac{K+K^*}{2}$ if necessary).  This means we have
$$
-\left(\frac{\delta-\gamma}{2}+\epsilon\right)I_{A^2}\leq S+K-\left(\frac{\delta+\gamma}{2}\right)I\leq \left(\frac{\delta-\gamma}{2}+\epsilon\right)I_{A^2}
$$
and adding $\frac{\delta+\gamma}{2}I_{A^2}$ to every term in this inequality gives the result.
\end{proof}

Using the tools from above, and repeating the proof in \cite{Sua} we have the following.
\begin{thm}
\label{ThreeEquivs}
Let $S\in\mathcal{T}_{2}$.  The following are equivalent:
\begin{itemize}
\item[(i)] $\lambda\notin \sigma_e(S)$;
\item[(ii)] $$\lambda\notin \bigcup_{x\in M_\mathcal{A}\setminus\D^n}\sigma(S_x)\quad\textnormal{ and }\quad\sup_{x\in M_\mathcal{A}\setminus\D^n}\norm{\left(S_x-\lambda I_{A^2}\right)^{-1}}_{\mathcal{L}\left(A^2,A^2\right)}<\infty;$$
\item[(iii)] there is a number $t>0$ depending only on $\lambda$ such that
$$
\norm{\left(S_x-\lambda I_{A^2}\right)f}_{A^2}\geq t\norm{f}_{A^2}\quad\textnormal{ and }\quad\norm{\left(S_x^*-\overline{\lambda} I_{A^2}\right)f}_{A^2}\geq t\norm{f}_{A^2}
$$
for all $f\in A^2$ and $x\in M_{\mathcal{A}}\setminus\D^n$.
\end{itemize}
\end{thm}
\begin{proof}
Without loss of generality, we may take $\lambda=0$.  First, suppose that (i) holds, i.e., $0\notin\sigma_e(S)$.  Then there is a $Q\in\mathcal{L}\left(A^2,A^2\right)$ such that $QS-I_{A^2}$ and $SQ-I_{A^2}$ are compact operators.  Let $x\in M_{\mathcal{A}}\setminus\D^n$.  Since $S\in\mathcal{T}_2$ we have $\left(SQ\right)_x=S_xQ_x$, $\left(QS\right)_x=Q_xS_x$, and $K_x=0$ for any compact operator $K\in\mathcal{L}\left(A^2,A^2\right)$.  Then we have
$$
Q_xS_x-I_{A^2}=0=S_xQ_x-I_{A^2}
$$
and so $S_x$ is invertible with $Q_x=\left( S_x\right)^{-1}$.  So we have $\norm{\left( S_x\right)^{-1}}_{\mathcal{L}\left(A^2,A^2\right)}=\norm{Q_x}_{\mathcal{L}\left(A^2,A^2\right)}\leq\norm{Q}_{\mathcal{L}\left(A^2,A^2\right)}<\infty$.  This gives that (ii) holds.

Now suppose that (ii) holds with $\lambda=0$.  Hence $S_x$ is invertible and there exists $\gamma^{-1}>0$, independent of $x\in M_{\mathcal{A}}\setminus\D^n$, such that
$$
\norm{\left(S_x^*\right)^{-1}}_{\mathcal{L}\left(A^2,A^2\right)}=\norm{\left(S_x\right)^{-1}}_{\mathcal{L}\left(A^2,A^2\right)}\leq \gamma^{-1}\quad\forall x\in M_{\mathcal{A}}\setminus\D^n.
$$
Then for any $f\in A^2$ and $x\in M_{\mathcal{A}}\setminus\D^n$ we have that 
$$
\gamma^{-1}\norm{S_x f}_{A^2}\geq \norm{\left(S_x\right)^{-1}S_x f}_{A^2}=\norm{f}_{A^2}.
$$
Rearrangement gives that (iii) holds.

Finally, suppose that (iii) holds with $\lambda=0$ and so  
$$
\norm{S_x f}_{A^2}\geq t\norm{f}_{A^2}\quad\forall f\in A^2\quad\forall x\in M_{\mathcal{A}}\setminus\D^n.
$$
Rearrangement gives that 
$$
t^2 I_{A^2}\leq S_x^*S_x\leq\norm{S}_{A^2}^2I_{A^2}.
$$
Given $\epsilon>0$, with $0<\epsilon<t^2$, by applying Corollary \ref{OperatorInequality} there is a self-adjoint compact operator $K$ such that
$$
\left(t^2-\epsilon\right) I_{A^2}\leq S^*S+K\leq\left(\norm{S}_{A^2}^2+\epsilon\right)I_{A^2}.
$$
Since $t^2-\epsilon>0$, we have that $S^*S+K$ is invertible.  Thus, there exists a $Q\in\mathcal{L}\left(A^2,A^2\right)$ such that $\left(QS^*\right)S+QK=I_{A^2}$.  This equality is simply the statement that $S+\mathcal{K}$ is left invertible in the Calkin algebra.  Repeating this argument but with $S_x^*$ one concludes that $S+\mathcal{K}$ is right invertible in the Calkin algebra.  These two statements together give (i).
\end{proof}

The above theorem then yields the following Corollary.
\begin{cor}
If $S\in\mathcal{T}_{2}$ then
$$
\overline{\bigcup_{x\in M_\mathcal{A}\setminus\D^n}\sigma(S_x)}\subset\sigma_e(S)
$$
with equality if $S$ is essentially normal.
\end{cor}
\begin{proof}
Suppose that $0\notin\sigma_e\left(S\right)$.  Then by Theorem \ref{ThreeEquivs} we have that $S_x$ is invertible and there is a $\gamma>0$ such that  $\norm{\left(S_x\right)^{-1}}_{\mathcal{L}\left(A^2,A^2\right)}\leq\gamma^{-1}$ for all $x\in M_{\mathcal{A}}\setminus\D^n$.  Then we have that a similar statement in terms of the spectral radius, 
$$
r\left(\left(S_x\right)^{-1}\right)\leq \norm{\left(S_x\right)^{-1}}_{\mathcal{L}\left(A^2,A^2\right)}\leq\gamma^{-1}.
$$
However, since 
$$
\sigma\left (S_x\right):=\{\xi^{-1}: \xi\in\sigma\left(\left(S_x\right)^{-1}\right)\}
$$
we have that $\abs{\xi}\geq\gamma$ for all $\xi\in \sigma\left(S_x\right)$.  In particular, we see that the open ball centered at the origin of radius $\gamma$ is disjoint from $\sigma\left( S_x\right)$ for all $x\in M_{\mathcal{A}}\setminus\D^n$.  This gives that,
$$
0\notin \overline{\bigcup_{x\in M_\mathcal{A}\setminus\D^n}\sigma(S_x)}.
$$

If $S$ is essentially normal, then $S_x$ is normal for all $x\in M_{\mathcal{A}}\setminus\D^n$.  So if
$$
0\notin \overline{\bigcup_{x\in M_\mathcal{A}\setminus\D^n}\sigma(S_x)}
$$
then there is a $\gamma>0$ such that the open ball of radius $\gamma$ centered at $0$ does not meet $\sigma\left( S_x\right)$. If we can show that $\sup_{x\in M_\mathcal{A}\setminus\D^n}\norm{\left(S_x\right)^{-1}}_{\mathcal{L}\left(A^2,A^2\right)}<\infty$, then by Theorem \ref{ThreeEquivs}, we would have that $0\notin\sigma_e\left(S\right)$.  However, this is easy since we have $r\left(\left(S_x\right)^{-1}\right)\leq\gamma^{-1}$, and since the spectral radius of a normal operator coincides with its norm, we have 
$$
\norm{\left(S_x\right)^{-1}}_{\mathcal{L}\left(A^2,A^2\right)}\leq\gamma^{-1}<\infty.
$$

\end{proof}

\section{Concluding Remarks}

One can also define weighted Bergman spaces on the polydisc.  Let $\vec{\alpha}=\left(\alpha_1,\ldots,\alpha_n\right)$ be a $n$-tuple such that $\alpha_l>-1$ for $1\leq l\leq n$.  Then we define the weighted Bergman space $A^p_{\vec{\alpha}}$ by
$$
\norm{f}_{A^p_{\vec{\alpha}}}^p:=c_{\vec{\alpha}}\int_{\D^n}\abs{f(z)}^p\prod_{l=1}^n\left(1-\abs{z_l}^2\right)^{\alpha_l}\,dv(z)<\infty.
$$
Analogous to what appears above, we have similar notions of Bergman--Carleson measures, normalized reproducing kernels for $A^p_{\vec{\alpha}}$, orthogonal projections from $L^2_{\vec{\alpha}}$ to $A^2_{\vec{\alpha}}$, Toeplitz operators and Toeplitz algebras.  Using the techniques in this paper, along with the modifications in \cite{MSW}, the interested reader can extend all the results in this paper in a straightforward manner to the case of weighted Bergman spaces.  In particular, one can obtain the following Theorem:
\begin{thm}
Let $1<p<\infty$, $\vec{\alpha}=(\alpha_1,\ldots,\alpha_n)$ satisfy $\alpha_l>-1$ for $1\leq l\leq n$ and $S\in\mathcal{L}(A^p_{\vec{\alpha}},A^p_{\vec{\alpha}})$.  Then $S$ is compact if and only if $S\in\mathcal{T}_{p,_{\vec{\alpha}}}$ and $\lim_{z\to\p\D^n} B(S)(z)=0$.
\end{thm}

Also, recall that $\B_n$ and $\D^n$ are examples of bounded symmetric domains $\Omega\subset\C^n$.  The results in \cites{Sua,MSW, E} and this paper provide concrete evidence for an analogous characterization of compact operators on the Bergman space over $\Omega$ in terms of the Berezin transform.  In a future project, the authors will demonstrate that many of the results of \cites{Sua, MSW} and this paper can be generalized to bounded symmetric domains.

%%%%%%%%%%%%
%%%References%%%
%%%%%%%%%%%%

%%%%%%%%%%%%
%%%%END%%%%%%
%%%%%%%%%%%%


\begin{bibdiv}
\begin{biblist}

\bib{Amar}{article}{
   author={Amar, {\'E}ric},
   title={Suites d'interpolation pour les classes de Bergman de la boule et
   du polydisque de ${\bf C}^{n}$},
   language={French},
   journal={Canad. J. Math.},
   volume={30},
   date={1978},
   number={4},
   pages={711--737}
}

\bib{AZ}{article}{
   author={Axler, Sheldon},
   author={Zheng, Dechao},
   title={Compact operators via the Berezin transform},
   journal={Indiana Univ. Math. J.},
   volume={47},
   date={1998},
   number={2},
   pages={387--400}
}

%\bib{AZ2}{article}{
%   author={Axler, Sheldon},
%   author={Zheng, Dechao},
%   title={The Berezin transform on the Toeplitz algebra},
%   journal={Studia Math.},
%   volume={127},
%   date={1998},
%   number={2},
%   pages={113--136}
%}

%\bib{BI}{article}{
%   author={Bauer, Wolfram},
%   author={Isralowitz, Joshua},
%   title={Compactness characterization of operators in the Toeplitz algebra of the Fock space $F^p_\alpha$},
%   eprint={http://arxiv.org/abs/1109.0305v2},
%   status={preprint},
%   date={2011},
%   pages={1--36}
%}

\bib{BCL}{article}{
   author={Berndtsson, Bo},
   author={Chang, Sun-Yung A.},
   author={Lin, Kai-Ching},
   title={Interpolating sequences in the polydisc},
   journal={Trans. Amer. Math. Soc.},
   volume={302},
   date={1987},
   number={1},
   pages={161--169}
}

\bib{CLNZ}{article}{
   author={Choe, Boo Rim},
   author={Lee, Young Joo},
   author={Nam, Kyesook},
   author={Zheng, Dechao},
   title={Products of Bergman space Toeplitz operators on the polydisk},
   journal={Math. Ann.},
   volume={337},
   date={2007},
   number={2},
   pages={295--316}
}

\bib{CR}{article}{
   author={Coifman, R. R.},
   author={Rochberg, R.},
   title={Representation theorems for holomorphic and harmonic functions in
   $L^{p}$},
   conference={
      title={Representation theorems for Hardy spaces},
   },
   book={
      series={Ast\'erisque},
      volume={77},
      publisher={Soc. Math. France},
      place={Paris},
   },
   date={1980},
   pages={11--66}
}

\bib{E92}{article}{
   author={Engli{\v{s}}, Miroslav},
   title={Density of algebras generated by Toeplitz operator on Bergman spaces},
   journal={Ark. Mat.},
   volume={30},
   date={1992},
   pages={227--243}
}

\bib{E}{article}{
   author={Engli{\v{s}}, Miroslav},
   title={Compact Toeplitz operators via the Berezin transform on bounded
   symmetric domains},
   journal={Integral Equations Operator Theory},
   volume={33},
   date={1999},
   number={4},
   pages={426--455}
}

\bib{H}{article}{
   author={Hastings, William W.},
   title={A Carleson measure theorem for Bergman spaces},
   journal={Proc. Amer. Math. Soc.},
   volume={52},
   date={1975},
   pages={237--241}
}

%\bib{I}{article}{
%   author={Issa, Hassan},
%   title={Compact Toeplitz operators for weighted Bergman spaces on bounded
%   symmetric domains},
%   journal={Integral Equations Operator Theory},
%   volume={70},
%   date={2011},
%   number={4},
%   pages={569--582}
%}

\bib{J}{article}{
   author={Jafari, F.},
   title={Carleson measures in Hardy and weighted Bergman spaces of
   polydiscs},
   journal={Proc. Amer. Math. Soc.},
   volume={112},
   date={1991},
   number={3},
   pages={771--781}
}

%\bib{LH}{article}{
%   author={Li, Song Xiao},
%   author={Hu, Jun Yun},
%   title={Compact operators on Bergman spaces of the unit ball},
%   language={Chinese, with English and Chinese summaries},
%   journal={Acta Math. Sinica (Chin. Ser.)},
%   volume={47},
%   date={2004},
%   number={5},
%   pages={837--844}
%}

%\bib{L}{article}{
%   author={Luecking, Daniel H.},
%   title={Representation and duality in weighted spaces of analytic
%   functions},
%   journal={Indiana Univ. Math. J.},
%   volume={34},
%   date={1985},
%   number={2},
%   pages={319--336}
%}

\bib{M}{article}{
   author={Michalak, Artur},
   title={Carleson measures theorems for generalized Bergman spaces on the
   unit polydisk},
   journal={Glasg. Math. J.},
   volume={43},
   date={2001},
   number={2},
   pages={277--294}
}

\bib{MSW}{article}{
   author={Mitkovski, Mishko},
   author={Su\'arez, Daniel},
   author={Wick, Brett D.},
   title={The Essential Norm of Operators on $A^p_\alpha(\mathbb{B}_n)$},
   eprint={http://arxiv.org/abs/1204.5548},
   status={preprint},
   pages={1--32},
   date={2012}
}

\bib{NZ}{article}{
   author={Nam, Kyesook},
   author={Zheng, Dechao},
   title={$m$-Berezin transform on the polydisk},
   journal={Integral Equations Operator Theory},
   volume={56},
   date={2006},
   number={1},
   pages={93--113}
}

\bib{R}{article}{
   author={Raimondo, Roberto},
   title={Toeplitz operators on the Bergman space of the unit ball},
   journal={Bull. Austral. Math. Soc.},
   volume={62},
   date={2000},
   number={2},
   pages={273--285}
}

%\bib{St}{article}{
%   author={Stroethoff, Karel},
%   title={Compact Toeplitz operators on Bergman spaces},
%   journal={Math. Proc. Cambridge Philos. Soc.},
%   volume={124},
%   date={1998},
%   number={1},
%   pages={151--160}
%}

%\bib{SZ}{article}{
%   author={Stroethoff, Karel},
%   author={Zheng, Dechao},
%   title={Toeplitz and Hankel operators on Bergman spaces},
%   journal={Trans. Amer. Math. Soc.},
%   volume={329},
%   date={1992},
%   number={2},
%   pages={773--794}
%}

\bib{Roch}{article}{
   author={Rochberg, Richard},
   title={Interpolation by functions in Bergman spaces},
   journal={Michigan Math. J.},
   volume={29},
   date={1982},
   number={2},
   pages={229--236}
}

\bib{Sua}{article}{
   author={Su{\'a}rez, Daniel},
   title={The essential norm of operators in the Toeplitz algebra on $A^p(\mathbb{B}_n)$},
   journal={Indiana Univ. Math. J.},
   volume={56},
   date={2007},
   number={5},
   pages={2185--2232}
}

\bib{Sua2}{article}{
   author={Su{\'a}rez, Daniel},
   title={Approximation and the $n$-Berezin transform of operators on the
   Bergman space},
   journal={J. Reine Angew. Math.},
   volume={581},
   date={2005},
   pages={175--192}
}

\bib{Sua3}{article}{
   author={Su{\'a}rez, Daniel},
   title={Approximation and symbolic calculus for Toeplitz algebras on the
   Bergman space},
   journal={Rev. Mat. Iberoamericana},
   volume={20},
   date={2004},
   number={2},
   pages={563--610}
}

%\bib{YS}{article}{
%   author={Yu, Tao},
%   author={Sun, Shan Li},
%   title={Compact Toeplitz operators on the weighted Bergman spaces},
%   language={Chinese, with English and Chinese summaries},
%   journal={Acta Math. Sinica (Chin. Ser.)},
%   volume={44},
%   date={2001},
%   number={2},
%   pages={233--240}
%}

\bib{W}{article}{
   author={Wick, Brett D.},
   title={Carleson Measures on Polydiscs via Trees},
   eprint={},
   status={preprint},
   pages={1--5},
   date={2012}
}
   

\bib{Zhu}{book}{
   author={Zhu, Kehe},
   title={Spaces of holomorphic functions in the unit ball},
   series={Graduate Texts in Mathematics},
   volume={226},
   publisher={Springer-Verlag},
   place={New York},
   date={2005},
   pages={x+271}
}

\end{biblist}
\end{bibdiv}
\end{document}